\documentclass[12pt]{amsart}
\usepackage{geometry}
\usepackage{amsmath,amsthm,amscd,amssymb,latexsym,enumerate,bm,lscape,epsfig,nicefrac,hyperref,amsfonts,cite,mathtools,mathrsfs}
\usepackage{enumitem}
\usepackage{lipsum}

\newlength\mylen
\newlist{mycases}{enumerate}{1}
\setlist[mycases,1]{label=\textbf{Case~\arabic*.}, 
	labelwidth=\dimexpr-\mylen-\labelsep\relax,leftmargin=0pt,align=right}

\newcommand{\N}{\mathbb{N}}

\newcommand{\R}{\mathbb{R}}

\newcommand{\E}{\mathbb{E}}
\newcommand{\G}{\mathbb{G}}

\newcommand{\supp}{\text{supp}}

\newtheorem{info}{}
\newtheorem{theorem}[info]{Theorem}
\newtheorem{corr}[info]{Corollary}

\newtheorem{defin}[info]{Definition}

\newtheorem{lem}[info]{Lemma}
\newtheorem{prop}[info]{Proposition}
\newtheorem{open}[info]{Open Problem}
\newtheorem{remark}[info]{Remark}
\numberwithin{info}{section}

\numberwithin{equation}{section}
\renewcommand{\[}{\begin{equation}}
	\renewcommand{\]}{\end{equation}}
\makeatletter

\numberwithin{info}{section}
\renewcommand{\[}{\begin{equation}}
	\renewcommand{\]}{\end{equation}}

\makeatletter
\g@addto@macro\normalsize{%
	\setlength\abovedisplayskip{5pt}
	\setlength\belowdisplayskip{5pt}
	\setlength\abovedisplayshortskip{4pt}
	\setlength\belowdisplayshortskip{4pt}}
\makeatother

\newcommand{\lam}{\lambda}

\renewcommand{\P}{\mathbb{P}}
\newcommand{\cal}{\mathcal}

\newcommand{\<}{\langle}
\renewcommand{\>}{\rangle}
\newcommand{\Sum}{\mathrm{\Sum}}
\newcommand{\Var}{\mathrm{Var}}

\newcommand{\D}{\nabla}

\newcommand{\qc}{q_c}
\newcommand{\zetac}{\zeta_c}

\renewcommand{\b}{\bm}
\newcommand{\tr}{\mathrm{tr}}

	\title[Wandering Exponents in the Elastic Polymer]{Wandering Exponents and the Free Energy of the High-Dimensional Elastic Polymer}

	\author{Gérard Ben Arous}
    \address{(Gérard Ben Arous) Courant Institute, New York University}
    \email{gba1@nyu.edu}

    \author{Pax Kivimae}
    \address{(Pax Kivimae) Courant Institute, New York University and Department of Mathematics, University of Colorado, Colorado Springs}
    \email{pkivimae@uccs.edu}

\begin{document}

	\maketitle
	
	\begin{abstract}
        We study the behavior of the elastic polymer, a model of a directed polymer in a continuous Gaussian random environment that is independent in time and correlated in space, as the dimension of the environment is taken to infinity. We give an explicit asymptotic formula for the free energy, which is given in terms of the distribution of the inner product of two sampled configurations, which we also obtain an implicit formula for. From this, we provide an explicit characterization of both the low- and high-temperature phases of this model in terms of the spatial correlation function of the environment. We find asymptotics for the wandering exponent when the spatial correlation function is either an exponential or a power-law decay. Our results show that when the correlations are either suitably weak or short ranged, the model is asymptotically diffusive. On the other hand, for suitably strong long ranged correlations, the model is asymptotically superdiffusive. Moreover, we show that this transition coincides exactly with another transition where the model goes from being one-step replica symmetry breaking to full-step replica symmetry breaking. This rigorously confirms many of the findings of M\'{e}zard and Parisi \cite{mezardparisi} in the physics literature.
	\end{abstract}
	
	\section{Introduction}
	
	Understanding the behavior of a directed polymer in a random potential is a classical question in the theory of disordered systems, which has been actively studied by both mathematicians and physicists for decades. At its core, the system consists of a random walk (i.e., the polymer) subjected to an independent random potential function, which we call the environment. This environment encourages the walk to move toward the environment's minima, a set that is constantly changing in time. However, this preference is often at odds with the behavior of the original random walk. This leads to a competition between the environment, which seeks to pin the model to rough configurations, and the base path-measure of the random walk, which is often much smoother. The outcome is a random process with a complex behavior that often mixes these two extremes. 
	
	One of the most intriguing aspects of these models is a transition in their behavior from the simple diffusive behavior of a simple random walk to a novel superdiffusive behavior as the strength of the environment is increased. This transition can be seen through a number of statistics of the model, as it simultaneously affects the quenched (typical) free energy of the model, the decay of the polymer's spatial covariance, and the amount that distinct samples localize on the same path. For an introduction to the mathematical theory of random polymers, one may consult \cite{polymerbook1}, as well as \cite{junk2024strong} for an overview of more recent developments. 
	
	Here we study a model for the directed polymer in discrete time and continuous space introduced by Fyodorov, Le Doussal, Rosso, and Texier \cite{YanPolymer1}, as a special case of the elastic manifold model of M\'{e}zard and Parisi \cite{mezardparisi}. This model is very similar to the discrete time and continuous space model introduced by Petermann \cite{Petermann2000} (see also the later work of Méjane \cite{Mejane}) except that the model we consider is massive and periodic. This model can also be seen as a discrete version of the continuous time and continuous space defined in Rovira and Tindel \cite{RoviraTindel}, which has also seen extensive study \cite{polymerContinoustime1,polymerContinoustime2, LacoinCorrelation,gu2022gaussian, MR4443199}. The key feature of all of these models is that they allow the environment to have spatial correlations.
	
	However, despite intense study, most results in this area are still only known when the dimension of the environment is either one or two. So instead, we approach this problem from the other direction: denoting the dimension of the environment by $N$, we will study the behavior of the model in the large $N$ limit (i.e., the mean field limit). The high-dimensional limit of this model has seen extensive study in the physics literature (see \cite{fyodorov-manifold-minimum,fyodorov-polymer,fyodorov2020manifolds,fyodorov-manifold-minimum,fyodorov2025groundstateenergyfluctuations,le2008cusps,doussal1, Carlucci_1996, Le_Doussal_2004} as well as the references therein); however, we focus on the foundational work of M{\'e}zard and Parisi \cite{mezardparisi}. They found that taking this limit allowed them to study the model using tools developed to understand mean-field spin glass models. While non-rigorous, this allowed them to compute both the quenched free energy and the wandering exponents of the model asymptotically. Moreover, they predicted that the Gibbs measure should enter a replica symmetry breaking phase in the low temperature (i.e., strong disorder) regime, where the Gibbs measure can be decomposed into a diverging number of hierarchical states \cite{RSBintro}. 
	
	The purpose of this work is to pursue these results rigorously. In particular, we will compute the value of the quenched free energy for this model in the high-dimensional continuum limit. We give this in terms of a Parisi-type variational problem, which allows us to determine both the overlap distribution of two samples from the model and the mean radius of the path as the unique solution to this variational problem. Moreover, we find a simple analytic description for the low-temperature phase of the model, which is characterized by a transition from replica symmetry to replica symmetry breaking in the overlap distribution (see Definition \ref{def:RS and RSB}).
	
	We then move on to another well-studied aspect in the theory of directed polymers, namely the wandering exponent and the phenomenon of superdiffusivity. Roughly, it is believed that if the disorder is suitably strong, the polymer will drift farther to reach a more favorable environment, which will lead the path to move far more quickly than Brownian motion. In particular, if we denote the polymer by $\b{u}(x)$ and the Gibbs measure by $\<*\>$, then it is conjectured that there is a number $\eta>0$ such that correlations of $\b{u}(x)$ scale as
	\[\E\<\|\b{u}(x)-\b{u}(y)\|^2\> \approx |x-y|^{2\eta},\label{eqn:wandering exponent heuristic}\]
	when the distance $|x-y|$ is large.\footnote{This exponent $\eta$ is also believed to describe the typical size of $\|\b{u}(x)\|$ and $\max_{0\le z\le x}\|\b{u}(z)\|$ as $x\to \infty$, at-least up to lower-order terms. Moreover, it is common for any of these to be used to define $\eta$, even if their equivalence has not been shown. We will ignore this issue, however, in recalling results for $\eta$ below to simplify our exposition.} The exponent $\eta$ is called the wandering exponent. Note that when $\b{u}$ is Brownian motion one has that $\eta=\frac{1}{2}$, and we will say that a polymer with $\eta=\frac{1}{2}$ is diffusive, while we will call a polymer with $\eta>\frac{1}{2}$ is superdiffusive. In general, it is believed that when the model is diffusive, it will converge to Brownian motion after rescaling, as was shown by Comets and Yoshida \cite{comets1} in a discrete model, while in the superdiffusive case, this will fail.
	
	We study this exponent by first providing a general formula for the mean-squared displacement of the model in terms of the overlap distribution (Theorem \ref{theorem:formula for the displacement at massive mu}). Taking the massless limit, we then show that in the entire high temperature (i.e., weak disorder) regime, the model is asymptotically diffusive. We then compute the exact wandering exponent in the low temperature regime when the spatial decay is given by either exponential or power-law decay. 
	
	In the exponential case, we find that the model is always diffusive. However, in the case of the power-law, we find a transition occurs based on the exponent of the power. For large exponents (i.e. short-ranged models), the model is diffusive, whereas for small exponents (i.e. long-range models), the model is superdiffusive with a wandering exponent that is sensitive to the exact exponent of the power. Moreover, we study the behavior of the overlap distribution in these models, and show that this dichotomy in the power-law case exactly coincides with a transition in the behavior of the overlap distribution. When the model is diffusive, the model is one-replica symmetry breaking (1RSB), whereas in the superdiffusive case the model is full replica symmetry breaking (FRSB). We will define these terms in Definitions \ref{def:RS and RSB} and \ref{def:krsb}, but we note now that they are both properties of the model's limiting overlap distribution.
	
	While this picture was first found in the physics literature by M\'{e}zard and Parisi \cite{mezardparisi}, some aspects of this transition have been shown in the mathematical polymer community. For a related model, Lacoin \cite{LacoinCorrelation} showed a lower bound for the wandering exponent in the case of long-ranged power-law correlations at any fixed $N\ge 2$. Our work shows that this lower bound is asymptotically tight. Moreover, Rovira and Tindel \cite{RoviraTindel} showed that for long-ranged power-law correlations, the model is always in the low-temperature phase regardless of temperature, while for short-ranged power-law correlations and $N\ge 3$, Lacoin \cite{LacoinCorrelation} has shown that there is always a non-zero transition temperature. We confirm these results asymptotically, and moreover give an explicit formula for the asymptotic quenched free energy in these cases. 
	
	\subsection{Model}
	
	We now introduce the model we will consider. Choose a length parameter $L\in \N$. Our configuration space will be the space of functions $\b{u}:\Lambda_L\to \R^N$, where
	\[\Lambda_L=\{0, L^{-1/2},2L^{-1/2},\cdots, (L-1)L^{-1/2}\}.\]
	We will denote the set of such functions as $(\R^{N})^{\Lambda_L}$. In what follows, we will consider functions on this set to have periodic boundary conditions so that $\b{u}(L^{1/2})=\b{u}(0)$. We note that $\Lambda_L$ contains precisely $L$ points, evenly spaced apart by a distance of $L^{-1/2}$. Thus, roughly as $L\to \infty$, the distance between any two points tends to zero, while the total length is $L^{1/2}$, which diverges to $\infty$.  
	
	Next, we introduce the base path-measure of our polymer on $(\R^{N})^{\Lambda_L}$. Roughly, this will be a discretization of the Ornstein-Uhlenbeck process. For this, we fix a mass parameter $\mu>0$, an elastic strength parameter $t>0$, and an inverse temperature $\beta>0$. For our base path, we consider the  periodic random walk on $\Lambda_L$ given by 
	\[\P_{\beta,\mu,N,L}(d\b{u})=\frac{1}{\tilde{Z}_{N,L}}\exp \left(-\frac{\beta}{2}\sum_{x\in \Lambda_L}\left( \mu \|\b{u}(x)\|^2-t(\Delta_L \b{u}(x),\b{u}(x))\right)L^{-1/2}\right)d\b{u},\label{eqn:path measure of base rw}\]
	where $\Delta_L f(x)=L(f(x+L^{-1/2})+f(x-L^{-1/2})-2f(x))$ is the discrete Laplacian on $\Lambda_L$, and $\tilde{Z}_{N,L}$ is a normalization constant given by
	\[\tilde{Z}_{N,L}=(\left(2\pi/\beta\right)^LL^{L/2}\det(\mu I-t\Delta_L)^{-1})^{N/2}.\] 
	The component functions $\b{u}_i(x)$ are an i.i.d. family of centered stationary Gaussian random walks on $\Lambda_L$. It is not difficult to show that this process $\b{u}_i(x)$ converges in law to stationary Ornstein-Uhlenbeck process as one takes $L\to \infty$, and that the process $\b{u}_i(x)-\b{u}_i(0)$ converges in law to a Brownian motion in the limit $\mu \to 0$ (infact, these claims routinely follow from the more general Corollaries \ref{corr:parisi greens function: RS} and \ref{corr:wandering: rs} below).
	
	To define our polymer measure, we will fix a random smooth function $V_N:\R^N\to \R$. We define the polymer Gibbs measure as 
	\[G_{\beta,N,L}(d\b{u})=\frac{1}{Z_{N,L}(\beta,\mu,t)}\exp \left(-\sum_{x\in \Lambda_L}\beta V_{N,x}(\b{u}(x))L^{-1/4}\right)\P_{\beta,\mu,N,L}(d\b{u}),\]
	where $(V_{N,x})_{x\in \Lambda_L}$ are family of i.i.d copies of $V_N$, and $Z_{N,L}(\beta,\mu,t)$ is the partition function of the polymer. We will use the standard notations of $\<f(\b{u})\>$ and $\<f(\b{u},\b{u}')\>$ for the expectations of a function $f$ with respect to $G_{\beta,N}$ and $G_{\beta,N}^{\otimes 2}$, respectively. We will also adopt this notation elsewhere in the document when the measure is clear.
	
	We will consider the case where $V_N$ is an isotropic centered Gaussian function. The law of $V_N$ is then specified by its covariance, and for this, we will assume that there is some fixed function $B$ such that 
	\[\E[V_N(u)V_N(u')]=NB(\|u-u'\|^2_N),\]
	where $\|u\|_N^2=\frac{1}{N}\sum_{i=1}^{N}u_i^2$. It is a result of Schoenberg \cite{schoenberg} that for a fixed function $B$, such a function $V_N$ exists for each $N\ge 1$ if and only if $B$ admits a representation of the form
	\[B(x)=c_0+\int_0^{\infty}\exp(-\lam^2 x)\nu(d\lam),\label{eqn:B-decomposition}\]
	where $\nu$ is a finite non-negative measure on $(0,\infty)$ and $c_0\ge 0$. For convenience, we will assume that there is $\epsilon>0$ such that $\int_0^\infty \exp(\lam^2 \epsilon) \nu(d\lam)<\infty$ and that $c_0=0$. 
	
	We note that this Gibbs measure can alternatively be written as 
	\[G_{\beta,N,L}(d\b{u})=\frac{1}{\hat{Z}_{N,L}(\beta,\mu,t)}\exp \left(-\beta \cal{H}_{N,L}(\b{u})\right)d\b{u},\label{eqn:alternative gibbs}\]
	where $\cal{H}_{N,L}:(\R^{\Lambda_L})^N\to \R$ is the Hamiltonian 
	\[\cal{H}_{N,L}(\b{u})=\frac{1}{2}\sum_{x\in \Lambda_L}\left( \mu \|\b{u}(x)\|^2-t(\Delta_L \b{u}(x),\b{u}(x))\right)L^{-1/2}+\sum_{x\in  \Lambda_L}V_{N,x}(\b{u}(x))L^{-1/4},\label{eqn:alternative hamiltonian}\]
	and $\hat{Z}_{N,L}(\beta,\mu,t)$ is a normalization factor. This presentation is more in line with the one used in the physics literature \cite{mezardparisi,mezardparisi2,fyodorov-polymer,fyodorov-manifold}.
	
	In particular, heuristically our Hamiltonian can be viewed as a discretization of the continuum Hamiltonian \[\cal{H}(\b{u})=\frac{1}{2}\int_{\R}\left(\mu \|\b{u}(x)\|^2+t\|\D \b{u}(x)\|^2\right)dx+\int_{\R}V_N(x,\b{u}(x))dx, \label{eqn:continuum Hamiltonian}\]
	where $V_N(x,\b{u}(x))$ is a centered Gaussian function with covariance
	\[\E[V_N(x,\b{u})V_N(y,\b{u}')]=\delta_D(x-x')B(\|\b{u}-\b{u}'\|^2_N),\]
	where $\delta_D$ is the \textbf{Dirac} $\delta$-function. In particular, we can view the first term of $\cal{H}_{N, L}(\b{u})$ as a Riemann sum for the first integral in (\ref{eqn:continuum Hamiltonian}) with spacing $\Delta x=L^{-1/2}$. The second term of $\cal{H}_{N, L}(\b{u})$ can similarly be thought of as a Riemann sum for the second integral in (\ref{eqn:continuum Hamiltonian}), with the scaling factor of $L^{-1/4}=L^{-1/2}L^{1/4}$ coming from the fact that $V_{N,x}(\b{u})$ needs to be rescaled by $L^{-1/4}$, to account for the different normalization conventions between the Dirac and Kronecker $\delta$-functions.

	\begin{remark}
		A natural variant of this model is given by changing the base random walk to one still given by (\ref{eqn:path measure of base rw}), except instead letting $\b{u}$ be periodic (i.e. $\b{u}(0)=\b{u}(L^{1/2})$), one takes the boundary conditions $\b{u}(0)=0$, and simply omits the term $\nabla_Lf((L-1)L^{-1/2})$ from exponential in (\ref{eqn:path measure of base rw}). We believe the results for this model would be similar, and often identical, to the results we obtain here. However, our methods do not generalize directly to this case. The issue is that in this paper we rely heavily on results for the finite-$L$ model from our companion papers \cite{Paper1,Paper2}. Many of the results in \cite{Paper1} would easily generalize to treat this variant. However, the results in \cite{Paper2} crucially rely on the underlying model being translation invariance. These results significantly simplify the formulas we obtain in \cite{Paper1}, and these simplifications make studying the limit as $L\to \infty$ tractable.
	\end{remark}
	
	\subsection{Results on the Gibbs measure and the Wandering Exponent}
	
	We will now give results for a number of statistics of the Gibbs measure in the high-dimensional limit. Our first results will concern the behavior of two statistics of the marginal distribution of the polymer at a fixed site $x\in \Lambda_L$. While these two limiting quantities are of interest in their own right, they will also appear naturally in our computation of the mean-squared displacement below. As the model is translation invariant, this marginal is independent of the choice of site $x\in \Lambda_L$. However, as the functions $\b{u}:\Lambda_L\to \R^N$ often distinct domains, to state our results, it will be useful to define the function $[x]_L$ which rounds $x$ into $\Lambda_L$. Explicitly,
	\[[x]_L=L^{-1/2}([L^{1/2}x] \mod L),\]
	where $[x]$ is the floor of $x$. If $x\in \Lambda_L$ then $x=[x]_L$, but for any $x\in \R$ we have $[x]_L\in \Lambda_L$. 
	
	With this notation, our first result on the Gibbs measure at any given site $x\in \Lambda$, the value of $\|\b{u}(x)\|_N$ concentrates on some fixed value.
	
	\begin{theorem}
		\label{theorem:intro:main:Euclidean parameters identification: radius}
		There exists $\qc\in (0,\infty)$ such that for any $x\in \R$,
		\[\lim_{L\to \infty}\lim_{N\to \infty}\E\left\<\left|\|\b{u}([x]_L)\|^2_N-\qc\right|\right\>=0.\label{eqn:intro:main:Euclidean parameters identification: radius}\]
	\end{theorem}
	
	Thus in the limit, the model is essentially supported on a product of thin annuli of radius $\sqrt{N\qc}$. Our next result studies the the distribution of the overlap, $(\b{u}(x),\b{u}'(x))_N$, when averaged over the choice of disordered $\{V_{N,x}\}$ (i.e. the annealed overlap distribution). We show that this converges in the limit to a fixed probability measure. For this, let $P([0,q))$ denote the set of probability measures on $[0,q]$ whose support does not contain $q$.
	
	\begin{theorem}
		\label{theorem:intro:main:Euclidean parameters identification: overlap} 
		There exists $\zetac\in P([0,\qc))$,  such that for any bounded continuous function $f:\R \to \R$ and any choice of $x\in \R$ we have that
		\[\lim_{L\to \infty}\lim_{N\to \infty}\E\<f((\b{u}([x]_L),\b{u}'([x]_L))_N)\>=\int_{0}^{\qc} f\left(r\right)\zetac(dr).\label{eqn:intro:main:Euclidean parameters identification: overlap}\]
	\end{theorem}
	
	The pair $(\qc,\zetac)$ is thus seen to be an important invariant of the Gibbs measure. Our next result shows that this pair can also be given implicitly as the solution to two fixed equations.
	
	\begin{theorem}[Characterization of The Pair $(\qc,\zetac)$]
		\label{theorem:intro:critical point equations}
		The pair $(\qc,\zetac)$ is the unique solution to the following equations:
		\[\beta \int_0^q\zeta([0,u])du=\frac{1}{\sqrt{\mu t}}, \label{eqn:intro:continuum:minimization eqn:Euclidean: Larkin}\]
		\[\zeta\left(\big\{s\in [0,q]:f_{\beta,q}(s)=\sup_{0<s'<q}f_{\beta,q}(s') \big\}\right)=1,\label{eqn:intro:continuum:minimization eqn:Euclidean: measure}\]
		where for $s\in [0,q)$, we define
		\[F_{\beta,q}(s)=-2B'(2(q-s))-\int_0^s \frac{2 du}{\beta^3 \delta(u)^3t},\;\; f_{\beta,q}(s)=\int_0^s F_{\beta,q}(u)du,\label{eqn:intro:continuum:minimizaation equations def of F} \]
		\[\delta(s)=\int^q_s \zeta([0,s])ds.\label{eqn:def:delta-P}\]
	\end{theorem}
	As this pair $(\qc,\zetac)$ will play an important role below, we will give it a name.
	
	\begin{defin}[Parisi Pair]
		\label{def:Parisi}
		We will call the pair $(\qc, \zetac)$ from Theorems \ref{theorem:intro:main:Euclidean parameters identification: radius}-\ref{theorem:intro:critical point equations} as the Parisi pair. We will call the measure $\zetac$ the Parisi measure.
	\end{defin}
	
	The Parisi measure $\zetac$ in particular will turn out to be of key importance to understanding the model. Indeed, the overlap distribution is known to be a fundamental object in the theory of disordered systems \cite{ParisiRSBReview}.
	
	Our next result will give a formula for the limiting mean-squared displacement of the model. This formula is given in terms of a functional of the Parisi pair $(\qc,\zetac)$, which we now introduce. 
	
	We first define the function
	\[\G_{x,t}(\mu)=\frac{e^{-|x|\sqrt{\frac{\mu}{t}}}-1}{2\sqrt{t \mu}}.\label{eqn:def:regularized continuum green}\]
	This function can be identified in terms of the mean-squared displacement of an Ornstein-Uhlenbeck process. Namely, if the process $U_t$ is given by the stationary solution to $dU_t=-\sqrt{\frac{\mu}{t}}U_tdt+dB_t$, where $B_t$ is a Brownian motion, then $\E[(U(x)-U(y))^2]=-\G_{x,t}(\mu)$. In particular, this function captures the mean-squared displacement of the base random walk in the limit $L\to \infty$. With this function defined, we now define our functional.
	
	\begin{defin}
		\label{def:H}
		Let $(\qc,\zetac)$ be the Parisi pair as in Definition \ref{def:Parisi}. Let $q_*=\sup(\supp(\zetac))<\qc$. For $x\in \R$, we define
		\[H_{\beta,t,\mu}(x)=-\frac{2}{\beta}\G_{x,t}\left(\frac{1}{\beta^2(\qc-q_*)^2t}\right)+\int_0^{q_*}\G_{x,t}'\left(\frac{1}{\beta^2\delta(u)t}\right) \frac{4\;du}{\beta^3\delta(u)^3t},\label{eqn:def:H continuum}\]
		where here
		\[\delta(s)=\int_s^q \zetac([0,u])du.\]
	\end{defin}
	
	With this functional, we now give our results on the mean-squared displacement at a fixed non-zero mass.
	
	\begin{theorem}[Formula for the Mean-Squared Displacement]
		\label{theorem:formula for the displacement at massive mu}
		For any $x,y\in \R$, we have that
		\[\lim_{L\to \infty}\lim_{N\to \infty}\E\<\|\b{u}([x]_L)-\b{u}([y]_L)\|^2_N\>=H_{\beta,t,\mu}(x-y).\]
	\end{theorem}
	
	This result allows us to compute the mean-squared displacement in any model where we suitably understand the Parisi pair $(\qc,\zetac)$. 
	
	The simplest case is when $\zetac$ is supported on a single point (i.e., when $\zetac$ is a Dirac mass). To distinguish this case, we will recall the following terminology from the theory of disordered systems \cite{RSBintro}.
	
	\begin{defin}[RS and RSB]
		\label{def:RS and RSB}
		When $\zetac$ is supported at one point, we call the model replica symmetric (RS). If a model is not RS, we call the model replica symmetry breaking (RSB).
	\end{defin}
	
	The RS phase serves as the high-temperature phase of the model, while the RSB phase is the low-temperature phase. In particular, the Parisi measure is a functional order parameter that detects this phase transition. The importance of this distinction between RS and RSB comes first and foremost from the phenomenon of replica symmetry breaking, famously introduced by Parisi \cite{ParisiSK1,parisiOG} to study the Sherrington-Kirkpatrick model, and further elaborated on in the works of M\'{e}zard, Parisi, Sourlas, Toulouse, and Virasoro \cite{ParisiUltrametric1, RSBintro}. In this picture, $\zetac$ is conjectured to encode deep structural properties of the Gibbs measure. This is because in the RSB phase, the Gibbs measure is expected to decompose into an infinite number of states whose properties are determined by $\zetac$. One may consult \cite{ParisiRSBReview} for a more in-depth review of these general physical predictions, as well as \cite{mezardparisi} for predictions specific to the model studied here.
	
	Returning to our study of the mean-squared displacement, in Theorem \ref{theorem:intro:1-d continuum T>0:RS} below, we show that one can explicitly find the value of the pair $(\qc,\zetac)$ when the model is RS. This will allow us to obtain the following corollary of Theorem \ref{theorem:formula for the displacement at massive mu}.
	
	\begin{corr}[Formula for the Mean-Squared Displacement: RS Case]
		\label{corr:formula for the displacement at massive mu RS Case}
		If the model is RS, then for any $x\in \R$ we have that
		\[H_{\beta,t,\mu}(x)=-\frac{2}{\beta}\G_{x,t}\left(\mu\right)+4B'\left(\frac{2}{\beta\sqrt{\mu t}}\right)\G_{x,t}'\left(\mu\right).\]
	\end{corr}
	
	In particular, in the case of $B=0$, this recovers the fact that the base random walk converges to a collection of i.i.d. Ornstein-Uhlenbeck processes in the limit. The second term represents a perturbation to the mean-squared displacement coming from the quenched disorder.
	
	We now move to discuss our applications of Theorem \ref{theorem:formula for the displacement at massive mu} to the study of the asymptotics of the wandering exponent. For this, we must first define what we precisely mean by the heuristic definition of the wandering exponent given in (\ref{eqn:wandering exponent heuristic}).
	
	Note that for the wandering exponent to be non-zero, one must consider models with zero mass (i.e. $\mu=0$). This is because if we have that $\eta>0$, the value of $\E\<\|\b{u}(x)\|_N^2\>$ must diverge as $x\to \infty$, which Theorem \ref{theorem:intro:main:Euclidean parameters identification: overlap} shows does not happen for fixed $\mu>0$. One may also compare this to the fact that the Ornstein-Uhlenbeck process has a wandering exponent of $0$, whereas when you send its bias parameter to zero, it converges to a Brownian motion, which has a wandering exponent $\frac{1}{2}$.
	
	
	\begin{defin}[The Asymptotic Wandering Exponent in The High-Dimensional Limit]
		For a fixed choice of $\beta,t>0$ and disorder function $B$, the polymer has asymptotic wandering exponent $\eta$ if
		\[\lim_{\mu\to 0}H_{\beta,t,\mu}(x)\asymp |x|^{2\eta},\]
		where here and elsewhere the notation $f(x)\asymp g(x)$ means that there are $C,c>0$ such that for all $|x|\ge C$, $cg(x)\le f(x)\le Cg(x)$.
	\end{defin}
	
	This order of limits taken here (i.e., $N\to \infty$, then $L\to \infty$, then $\mu \to 0$) is the order originally introduced by M\'{e}zard and Parisi \cite{mezardparisi}. In particular, by taking the limit in $N$ first, we simultaneously avoid having to show the existence of a wandering exponent for any finite $N$, as well as dealing with convergence to a continuum model. However, we do believe that the order in which these limits are taken does not affect the final result. Moreover, we note that the sense in which we define the wandering exponent is quite strong, in the sense that we do not allow any sub-polynomial terms in $|x|$. With this definition given, we can now define what we mean by the terms diffusive and superdiffusive. 
	
	\begin{defin}[Diffusivity and Superdiffusivity]
		Assume that the polymer has a wandering exponent of $\eta$. We will say the model is diffusive if $\eta=\frac{1}{2}$. We will say the model is superdiffusive if $\eta>\frac{1}{2}$.
	\end{defin}
	
	Equipped with Theorem \ref{theorem:formula for the displacement at massive mu}, we see that to calculate the asymptotic wandering exponents, we only need to have careful control over the pair $(\qc,\zetac)$ as $\mu\to 0$. As before, the case where our results are clearest is the RS case. We do not define the pair $(\qc,\zetac)$ in the limit $\mu\to 0$, so instead we will take the following definition.
	
	\begin{defin}[Massless RS]
		Consider the model at some fixed $(t,\beta)$. We will say that the massless model is RS at $(t,\beta)$ if the massive model is RS at $(\mu,t,\beta)$ for all sufficiently small $\mu>0$. Otherwise we will say that the massless model is RSB.
	\end{defin}
	
	Our next results show that in the massless RS phase, the wandering exponent is always $\eta=\frac{1}{2}$ (i.e., the model is diffusive). In fact, we do not even need to take the limit $x\to \infty$ in this case, and we find that the limiting correlation function is identical to that of Brownian motion.
	
	\begin{corr}[High Temperature Correlations]
		\label{corr:wandering: rs}
		Fix $(\beta,t)$, and assume that the massless model is $RS$. Then for any choice of $x\in \R$ we have that
		\[\lim_{\mu\to 0}H_{\beta,t,\mu}(x)=\frac{|x|}{\beta}.\]
		In particular, the wandering exponent is $\eta=\frac{1}{2}$ and the model is diffusive.
	\end{corr}
	
	In particular, we see in this case that the limiting mean-squared displacement in the massless RS case is independent of $B$, and coincides with that of the limiting base random walk (i.e., Brownian motion).

	More complex phenomena occur in the RSB phase. For this, we will establish results for the cases of $B$ which are most well studied\cite{mezardparisi,le2008cusps,doussal1, Petermann2000}. 
	\begin{defin}
		\label{def:main cases}
		Choose some $g,a,\gamma>0$. The case of an exponential correlator is given by
		\[B(u)=g e^{-a u},\label{eqn: B case exponential}\]
		and the case of a power-law correlator by
		\[B(u)=g(a+u)^{-\gamma}.\label{eqn: B case power}\]
	\end{defin}
	To show that the model is well-defined for these choices of correlators, one may use the result of Schoenberg \cite{schoenberg} recalled above, which states that $V_N$ exists for all $N\ge 1$ if and only if it can be written in the form (\ref{eqn:B-decomposition}). This is obvious for $B(u)=g e^{- a u}$, in the case of $B(u)=g (a+x)^{-\gamma}$ this follows from the integral relation $(1+u)^{-\gamma}=\frac{2}{\Gamma(\gamma)}\int_0^\infty e^{-\lambda^2u}e^{-\lambda^2}\lambda^{2\gamma-1}d\lambda$. By applying a scaling transformation $\b{u}\to \lambda \b{u}$, it suffices to consider the case of $a=1$. Moreover, we may also assume that $g=1$ by shifting the structure constants as $(\mu,t,\beta)\mapsto (\mu/g,t/g,\beta g)$.
	
	Our first result gives the wandering exponent when $B$ is a power-law correlator. In this case, we find that the model experiences a transition in terms of the exponent $\gamma$. When this exponent is above one (i.e., the field is short-ranged), the model is always asymptotically diffusive. However, when the exponent is below this threshold (i.e., the field is long-ranged), the model becomes superdiffusive with a wandering exponent that is sensitive to the exact exponent of the power law.
	
	\begin{theorem}[Power-law Correlations]
		\label{theorem: wandering: powerlaw}
		Fix $(t,\beta)$ and assume that $B(x)=(1+x)^{-\gamma}$. If $\gamma\ge 1$, we have that $\eta=\frac{1}{2}$, so that the model is diffusive. However, if $\gamma<1$, then $\eta=\frac{3}{2(\gamma + 2)}>\frac{1}{2}$. In particular, the model is superdiffusive.
	\end{theorem}
	
	As mentioned above, this result complements the lower bound produced by Lacoin \cite{LacoinCorrelation} for fixed $N$, which shows in a related model that $\eta\ge \frac{3}{2(\gamma+2)}$ for the same choice of $B$ if $N\ge 2$.
	
	Our next result shows that if we consider the model with exponential correlations, then the model is always asymptotically diffusive.
	
	\begin{theorem}[Exponential Correlations]
		\label{theorem: wandering: exponential}
		Fix $(t,\beta)$ and assume that $B(x)=e^{-x}$. Then $\eta=\frac{1}{2}$, so that the model is diffusive.
	\end{theorem}
	
	Now we note that comparing these results to Corollary \ref{corr:wandering: rs}, we first see that in the short-range cases (Theorem \ref{theorem: wandering: powerlaw} with $\gamma\ge 1$ and Theorem \ref{theorem: wandering: exponential}), the wandering exponent does not distinguish between the RS and RSB phases for these models, and so the model is always diffusive. On the other hand, in the long-range case (Theorem \ref{theorem: wandering: powerlaw} with $\gamma<1$), the model is never diffusive at any temperature. This is consistent with Corollary \ref{corr:wandering: rs}, as we will see that in the long-range case, the massless model is always RSB for any temperature, so that Corollary \ref{corr:wandering: rs} never applies. In fact, there is a surprisingly simple condition that determines when this is the case.
	
	\begin{theorem}
		\label{theorem:massless RS phase transition theorem}
		The massless model is RSB for all $\beta>0$ if and only if
		\[\limsup_{x\to \infty}B''(x)x^3=\infty. \label{eqn: does not have a massless phase transition}\]
		Otherwise, it is RS for sufficiently small $\beta$ and RSB for sufficiently large $\beta$.
	\end{theorem}
	
	It is easily verified that $B(x)=(1+x)^{-\gamma}$ satisfies (\ref{eqn: does not have a massless phase transition}) if and only if $\gamma<1$. Moreover, the choice $B(x)=e^{-x}$ does not satisfy (\ref{eqn: does not have a massless phase transition}).
	
	To show these results, we will need to understand the behavior of $(\qc,\zetac)$, which we do via Theorem \ref{theorem:intro:critical point equations}. The equations in Theorem \ref{theorem:intro:critical point equations} arise as the critical points of a certain functional. In particular, $(\qc,\zetac)$ will appear as the unique critical point for this functional,  which is how we will obtain Theorem \ref{theorem:intro:critical point equations}. This functional appears naturally in the study of the quenched free energy of the model. Infact, we will show that the limiting quenched free energy of this model is given by this functional evaluated on the pair $(\qc,\zetac)$. As such, we now focus on the analysis of the model's quenched free energy.
	
	\subsection{Results for the Quenched Free Energy}\
	
	Our next series of results will involve the quenched free energy of this model. For this, we now define the quenched free energy as 
	\[F_{N,L}(\beta,\mu,t)=\log Z_{N,L}(\beta,\mu,t),\]
	where $Z_{N,L}(\beta,\mu,t)$ is the partition function of the model given by 
	\[Z_{N,L}(\beta,\mu,t)=\int_{(\R^{N})^{\Lambda_L}}\exp \left(-\sum_{x\in \Lambda_L}\beta V_{N,x}(\b{u}(x))L^{-1/4}\right)\P_{\beta,\mu,N,L}(d\b{u}).\]
	To give our computation of $F_{N, L}(\beta,\mu,t)$, we will need to introduce a Parisi-type functional.
	
	\begin{defin}[The Parisi Functional]
		\label{definition:parisi functional}
		For $q\in (0,\infty)$ and $\zeta\in P([0,q))$ let $\delta$ denote the function defined in (\ref{eqn:def:delta-P}). We define the functional
		\[\cal{P}_{\beta}(q,\zeta)=\frac{1}{2}\left(\int_{0}^{q_*} \frac{du}{\beta t\delta(u)^2}-\frac{1}{\beta t(q-q_*)}-2\beta^2\int_0^q\zeta([0,u])B'(2(q-u)) du-\beta\mu q+\frac{2\sqrt{\mu}}{\sqrt{t}}\right),\]
		where here $q_*\in [0,q]$ is any point such that $\zeta([0,q_*])=1$. It is easily verified that $\cal{P}_{\beta}(q,\zeta)$ is independent of the choice of $q_*$.
	\end{defin}
	
	With this definition, we have the following asymptotic formula for the quenched free energy.
	
	\begin{theorem}[Computation of The Asymptotic Free Energy]
		\label{theorem:intro:1-d continuum T>0:free energy T>0}
		We have that
		\[\lim_{L\to \infty}\lim_{N\to \infty}L^{-1/2}N^{-1}F_{N,L}(\beta,\mu,t)=\sup_{q\in (0,\infty)}\left(\inf_{\zeta\in P([0,q))} \cal{P}_{\beta}(q,\zeta)\right),\]
		both almost surely and in expectation.
	\end{theorem}
	
	Moreover, we show that the optimization problem on the right is achieved at a unique critical point, which is infact the Parisi pair $(\qc,\zetac)$.
	
	\begin{theorem}[Critical Point Characterization of $(\qc,\zetac)$]
		\label{theorem:intro:1-d continuum T>0:free energy T>0 variational form: second version}
		We have that
		\[\sup_{q\in (0,\infty)}\left(\inf_{\zeta\in P([0,q))} \cal{P}_{\beta}(q,\zeta)\right)=\cal{P}_{\beta}(\qc,\zetac)\]
		where $(\qc,\zetac)$ is the unique critical point of the functional $\cal{P}_{\beta}(q,\zeta)$.
	\end{theorem}
	
	We will show that the critical points equations for the functional $\cal{P}_{\beta}(q,\zeta)$ can be simplified to (\ref{eqn:intro:continuum:minimization eqn:Euclidean: measure})-(\ref{eqn:intro:continuum:minimizaation equations def of F}). In particular, Theorem \ref{theorem:intro:critical point equations} is essentially a more explicit version Theorem \ref{theorem:intro:1-d continuum T>0:free energy T>0 variational form: second version}.
	
	\subsection{Results on the Replica Symmetric Phase}
	
	Next we present some results on the RS phase. The first characterizes when the model is RS based on an explicit variational problem in one variable, and in addition gives us a formula for $(\qc,\zetac)$ when the model is RS.
	
	\begin{theorem}[Characterization of the Replica Symmetric Phase]
		\label{theorem:intro:1-d continuum T>0:RS}
		The model is RS if and only if
		\[\sup_{0<s<\frac{1}{\sqrt{\mu t}}}g_{\beta,\mu}(s)=g_{\beta,\mu}\left(\frac{1}{\sqrt{\mu t}}\right),\label{eqn:ignore-luna-2}\]
		where here
		\[g_{\beta,\mu}(s)=\beta^2 B\left(\frac{2s}{\beta}\right)-\frac{1}{s t}-s\left(2\beta B'\left(\frac{2}{\beta\sqrt{\mu t}}\right)+\mu\right).\label{eqn:ignore-luna-10}\]
		When the model is RS, its parameters are given by
		\[
		(\qc,\zetac)=\left(\delta_{q_{*}},q_{*}+\frac{1}{\beta\sqrt{\mu t}}\right)\;\;\;\; q_{*}=-B'\left(\frac{2}{\beta\sqrt{\mu t}}\right)\frac{1}{\sqrt{\mu^3 t}}.
		\]
		and its limiting quenched free energy is given by
		\[
		\frac{\beta^2}{2} \left(B(0)-B\left(\frac{2}{\beta\sqrt{\mu t}}\right)\right).
		\]
	\end{theorem}
	
	We will next give some simpler results that allow us to more easily check whether the model is RS or RSB. They are based on the Larkin equation:
	\[ B''\left(\frac{2}{\beta\sqrt{\mu t}}\right)\frac{2}{\sqrt{\mu^3 t}}=1. \label{eqn:Larkin equation}\]
	While the origin of this equation is in physics literature, its present appearance may be explained in terms of the function $g_{\beta,\mu}$ occurring in Theorem \ref{theorem:intro:1-d continuum T>0:RS}. In particular, observe for (\ref{eqn:ignore-luna-2}) to hold, it is necessary that $\frac{1}{\sqrt{\mu t}}$ is a local maximum of $g_{\beta,\mu}$. One may check that $g_{\beta,\mu}'(\frac{1}{\sqrt{\mu t}})=0$, and routine computation gives that
	\[g_{\beta,\mu}'' \left(\frac{1}{\sqrt{\mu t}}\right)=4B'' \left(\frac{2}{\beta \sqrt{\mu t}}\right)-\frac{2\sqrt{\mu^3 t^3}}{t}.\label{eqn:g''}\]
	In particular, by rearranging the equation $g_{\beta,\mu}'' \left(\frac{1}{\sqrt{\mu t}}\right)>0$, we see that the model must be RSB if
	\[B''\left(\frac{2}{\beta\sqrt{\mu t}}\right)\frac{2}{\sqrt{\mu^3t}}>1.\label{eqn:g''>0}\]
	The boundary line for this condition gives (\ref{eqn:Larkin equation}). Remarkably, this local condition is often enough to guarantee global maximization, which controls when the model is RS. The following result gives a sufficient condition for this to be true, and so a partial characterization of the RS phase.
	
	\begin{corr}
		\label{corr:intro:RS larkin}
		Fix the choice of $t$ and $\beta$, and consider (\ref{eqn:Larkin equation}) as an equation in $\mu$.
		\begin{enumerate}
			\item If (\ref{eqn:Larkin equation}) has no solutions, then the model is RS for all choices of $\mu$.
			\item If (\ref{eqn:Larkin equation}) has one or more solutions, we will call the largest solution $\mu_{Lar}(\beta;t)$. Then if $\mu\ge \mu_{Lar}(\beta;t)$ the model is RS.
		\end{enumerate}
	\end{corr} 
	
	When it exists, we will call $\mu_{Lar}(\beta;t)$ the Larkin mass. This quantity is not only important to the RS behavior of the model, but also to the topological complexity of the model. In particular, the zero-temperature Larkin mass (i.e. $\lim_{\beta \to \infty}\mu_{Lar}(\beta;t)$) is also the mass for which the Hamiltonian experiences ``topological trivialization" \cite{topologytrivialization2}. This was predicted in the works of Fyodorov, and Le Doussal \cite{fyodorov-manifold,fyodorov-manifold-minimum}, and made rigorous in the case of finite $L$ by the work of Bourgade, McKenna, and the first author \cite{gerardbenpaul} and Xu and Zeng \cite{xuzengelasticmanifold}. Moreover, the Larkin mass is believed to be related to the critical parameter in the pinning-depinning transition (see \cite{larkin-review}). However, no rigorous results in this direction appear to be known.
	
	\subsection{Results on the Replica Symmetry Breaking Phase}
	
	Having treated the RS phases, we now give some results on the RSB phase. Inside the RSB phase, identification of the Parisi measure is an important, though often difficult, problem. While a explicit formula for the Pairi measure in general is not known even in the $L=1$ case, we are able to obtain some significantly simplifications to critical point equations in Theorem \ref{theorem:intro:critical point equations} under certain conditions on $B$.
	
	For this, we need to recall another definition, based on $\zetac$, which can be used to characterize how simple $\zetac$ is, and generalizes the simpler RS/RSB dichotomy.
	
	\begin{defin}[$k$RSB vs FRSB]
		\label{def:krsb}
		When $\zetac$ is supported at precisely $(k+1)$ points, the model is called $k$-step replica symmetry breaking ($k$RSB). If $\zetac$ is supported on an infinite number of points, the model is full-step replica symmetry breaking (FRSB).
	\end{defin}
	
	In this definition, note that being RS is equivalent to being $0$RSB, while if $k\ge 1$, a $k$RSB measure is automatically RSB. The simplest RSB case is 1RSB, with the opposite end being the FRSB case. The difference between these cases has been well studied, and we will return to a discussion of this after giving the remainder of our results.
	
	We will provide here both a sufficient condition on $B$ for the Parisi measure to be 1RSB in the entire RSB phase (Theorem \ref{theorem:intro:1-d continuum T>0:1RSB}) as well as a complementary condition on $B$ for the Parisi measure to be FRSB in the entire RSB phase (Theorem \ref{theorem:intro:1-d continuum T>0:FRSB}). These conditions suffice to deal with the cases introduced above in Definition \ref{def:main cases}.
	
	We begin by giving the sufficient condition for the model to be 1RSB. This generalizes the criterion found by Crisanti and Sommer \cite{crisantisommersOG} in the spherical pure $p$-spin glass model.
	
	\begin{theorem}
		\label{theorem:intro:1-d continuum T>0:1RSB}
		Let us assume that $(B''(x))^{-1/3}$ is strictly convex or linear, and that the model is not RS. Then the model is 1RSB. Moreover, if we write the Parisi measure as $\zetac=m\delta_{q_0}+(1-m)\delta_{q_*}$, then $(q_0,q_*,\qc,m)$ is the unique solution to the following equations: $0<q_0<q_*<\qc$,
		\[\beta(\qc-q_*+m(q_*-q_0))=\frac{1}{\sqrt{\mu t}},\label{eqn: 1RSB explicit}\]
		\[F(q_0)=F(q_*)=\int_{q_0}^{q_*}F(u)du=0.\label{eqn: 1RSB F equations}\]
		where here $F$ is the function
		\[F(x)=-2B'(2(\qc-x))-2q_0\sqrt{\mu^3 t}-\frac{1}{\beta m t}\left(\frac{1}{\sqrt{\mu t}}+\beta m(q_0-x)\right)^{-2}+\frac{\mu}{\beta m },\]
	\end{theorem}
	
	We now give a similar condition for the model to be FRSB. In this case, our result also gives a more explicit formula for $(\qc,\zetac)$.
	
	\begin{theorem}
		\label{theorem:intro:1-d continuum T>0:FRSB}
		Let us assume that $(B''(x))^{-1/3}$ is strictly concave and fix $\beta,t>0$. If (\ref{eqn:Larkin equation}) has a solution and $\mu<\mu_{Lar}(\beta;t)$ then the model is FRSB. Otherwise the model is RS. Moreover, when the model is FRSB, the Parisi pair is given by 
		\[
		\zetac([0,s])=\begin{cases}
			\beta^{-1}\left(\frac{4}{t}\right)^{1/3}\frac{-B'''(2(\qc-s))}{3\left(B''(2(\qc-s))\right)^{4/3}};\quad s\in [q_0,q_*)\\
			1;\quad s\in [q_*,\qc]\\
			0;\quad s\in [0,q_0)
		\end{cases},\label{eqn:theorem:equation for zeta}
		\]
		where $(q_0,q_*,\qc)$ is the unique triple with $0\le q_0<q_*<\qc$ which satisfies:
		\[\qc-q_*=\frac{1}{\beta\sqrt{\mu_{Lar}(\beta;t)t}},\;\;\; 2B''(2(\qc-q_0))=\sqrt{\mu^3t},\;\; -B'(2(\qc-q_0))=\frac{q_0}{\mu^3t}. \label{eqn:theorem:3 FRSB equations}\]
	\end{theorem}
	
	\begin{remark}
		If one works instead with the one-site model (i.e., $L=1$), then analogs of these results are also known. In this case, it is known that if $(B''(x))^{-1/2}$ is strictly concave, then the model is either RS or FRSB, and if $(B''(x))^{-1/2}$ is strictly convex, then the model is either RS or 1RSB.
		
		For the spherical analog of this model (i.e., the spherical mixed spin glass), this follows from Proposition 1.1 of \cite{subagFRSBgroundstate} in the FRSB case and Proposition 2.1 of \cite{talagrandOG} in the 1RSB case.
		Moreover, in our case, this follows by relating our Euclidean case to the spherical case by fixing $q$. In particular, note that for fixed $q$, $\inf_{\zeta \in P([0,q))}\cal{P}_\beta(q,\zeta)$ coincides up to a constant with the Parisi-functional for the mixed spherical spin glass with mixing parameter $r\mapsto B(2q(1-r))$ (after rescaling $\zeta$ to lie on $[0,1)$ instead of $[0,q)$). So taking $q=\qc$ yields the result.
	\end{remark}
	
	Finally, we note that Theorem \ref{theorem:intro:1-d continuum T>0:FRSB} and Corollary \ref{corr:intro:RS larkin} immediately give a complete description of the RS-RSB phase diagram in this case. In particular, the RS-RSB boundary line is given by taking the largest solution to the Larkin equation (\ref{eqn:Larkin equation}). However, in the cases covered by Theorem \ref{theorem:intro:1-d continuum T>0:1RSB}, Corollary \ref{corr:intro:RS larkin} only provides a partial characterization of the RS-RSB boundary and we are not aware of an explicit formula (see Figure \ref{figure:power}).
	
	\subsection{Examples}
	
	With these general results stated, we can now apply them to our cases of interest, namely $B(x)=(1+x)^{-\gamma}$ and $B(x)=e^{-x}$. We record the fact that the function $x^{\alpha}$ is strictly concave on $(0,\infty)$ if and only if $\alpha<1$ and strictly convex on $(0,\infty)$ if and only if $\alpha>1$. In particular, if $B(x)=(1+x)^{-\gamma}$ then we have that $(B''(x))^{-1/3}$ is strictly concave if and only if $\gamma<1$, and is strictly convex if instead one has that $\gamma>1$. Lasting, note that if $\gamma=1$, then $(B''(x))^{-1/3}$ is linear. In particular, Theorems \ref{theorem:intro:1-d continuum T>0:1RSB} and \ref{theorem:intro:1-d continuum T>0:FRSB} give the following result.
	
	\begin{theorem}
		\label{theorem: parisi: powerlaw}
		Consider the model with $B(x)=(1+x)^{-\gamma}$. If $\gamma< 1$, then the model is always either FRSB or RS. If $\gamma\ge 1$, then the model is always either 1RSB or RS.
	\end{theorem}
	
	\begin{remark}
		Both the transition in Theorem \ref{theorem: parisi: powerlaw} and the one in Theorem \ref{theorem: wandering: powerlaw} occur at the same threshold. In particular, we see that when $B(x)=(1+x)^{-\gamma}$, the model experiences a transition at $\gamma=1$, where for $\gamma<1$ the model is superdiffusive and FRSB in its entire RSB phase, where as for $\gamma\ge 1$ the model is diffusive and 1RSB in its entire 1RSB phase. Moreover, Theorem \ref{theorem:massless RS phase transition theorem} shows that the massless model with $\gamma\ge 1$ experiences a transition from RS to 1RSB in temperature, while the model with $\gamma<1$ remains FRSB at all temperatures.
	\end{remark}
	
	In Figure \ref{figure:power}, we compare the phase diagram for two cases of power-law $B$, one with $\gamma=1/2<1$ and the other with $\gamma=2>1$, and note some of the qualitative differences. 
	
	\begin{figure}
		\centering
		\begin{minipage}{0.4\textwidth}
			\centering
			\includegraphics[width=0.8\textwidth]{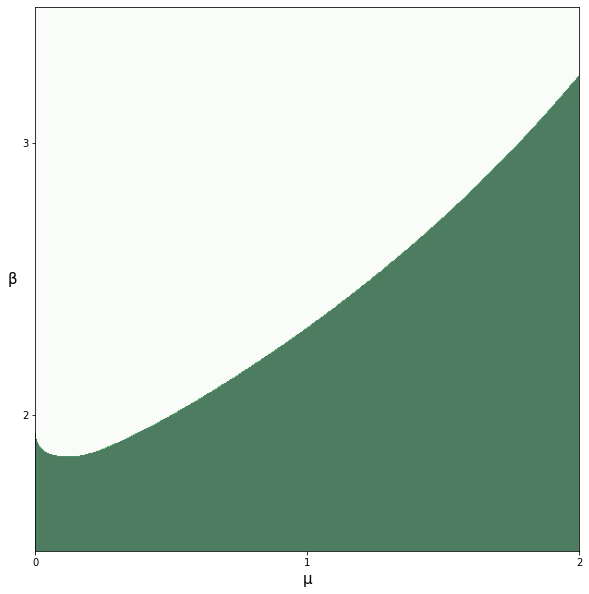} 
		\end{minipage}\hfill
		\begin{minipage}{0.4\textwidth}
			\centering
			\includegraphics[width=0.8\textwidth]{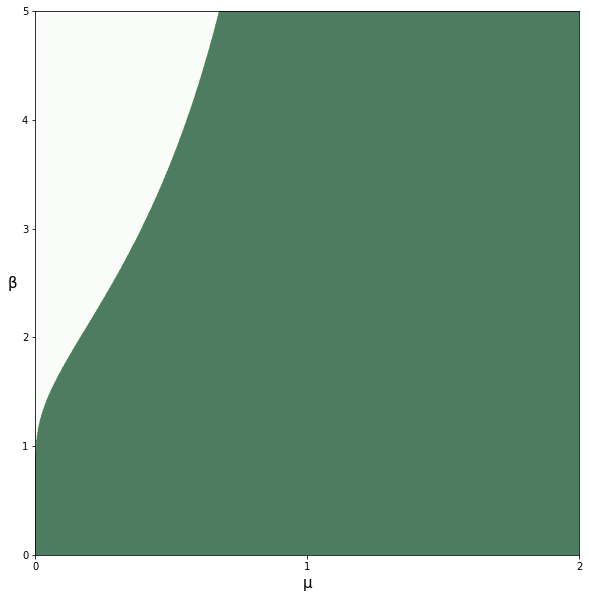}
		\end{minipage}
		\caption{The RS/RSB phase diagram for two choices of $B$. The left diagram is for the choice $B(x)=(1+x)^{-2}$ while the right diagram is for the choice $B(x)=(1+x)^{-1/2}$, and in both cases $t=1$. The RS region is in green while the RSB region is in white. The boundary curve intersects the $y$-axis (i.e., the massless cases) when $\beta=1.7583...$ in the left-hand case, and never intersects in the right-hand case, although it does get incredibly close to it (for example, $\mu_{Lar}(0.5)=0.0000048...$). For a fixed value of $\beta$, the boundary curve may intersect the horizontal line $y=\beta$ at either zero, one, or two points in the left-hand case, while it always intersects at exactly one point in the right-hand side case. If there is an intersection, the right-most point is always given by $(\beta,\mu_{Lar}(\beta;1))$ by Corollary \ref{corr:intro:RS larkin}. However, when the intersection has two points, we are not aware of an explicit description of the other point.}
		\label{figure:power}
	\end{figure}
	
	Moreover, noting that if $B(x)=e^{-x}$ then $(B''(x))^{-1/3}$ is strictly convex, we obtain the following result as well.
	
	\begin{theorem}
		\label{theorem: parisi: exponential}
		Consider the model with $B(x)=e^{-x}$. Then the model is always either 1RSB or RS.
	\end{theorem}
	
	\begin{figure}
		\centering
		\includegraphics[width=0.45\textwidth]{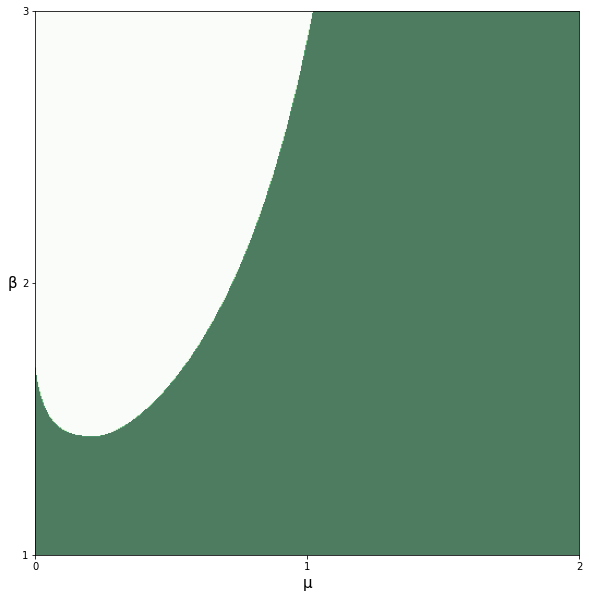} 
		\caption{The RS/1RSB phase diagram for the choice $B(x)=e^{- x}$ with $t=1$. We note that this diagram is qualitatively similar to the one for $B(x)=(1+x)^{-2}$ considered in Figure \ref{figure:power}, and that similar remarks apply. In particular, we also do not have an explicit description for the left half of the boundary curve.}
		\label{figure:exp}
	\end{figure}
	
	In Figure \ref{figure:exp}, we also prove the phase diagram for $B(x)=e^{-x}$. Finally, we observe that in the case of $B(x)=(1+x)^{-\gamma}$ with $\gamma<1$, Theorem \ref{theorem:intro:1-d continuum T>0:FRSB} not only guarantees that the model is FRSB when then model is not RS, but infact can be used to give the following explicit description of $(\qc,\zetac)$ in this case.
	
	\begin{corr}
		\label{corr:FRSB}
		Consider the model $B(x)=(1+x)^{-\gamma}$ for $\gamma<1$ at some fixed $(\beta,t)$. The Larkin mass $\mu_{Lar}(\beta;t)$ exists and is given by the unique $\mu$ which solves the equation
		\[2\gamma(\gamma+1)\left(1+\frac{2}{\beta \sqrt{\mu t}}\right)^{-\gamma-2}=\sqrt{\mu^3t}.\label{eqn:larkin-mass-FRSB-gamma}\] 
		For all $\mu\ge \mu_{Lar}(\beta;t)$ the model is RS, and for all $\mu<\mu_{Lar}(\beta;t)$ the model is FRSB.\\
		
		\noindent Moreover, in the FRSB case, the Parisi pair, $(\qc,\zetac)$, in Theorem \ref{theorem:intro:1-d continuum T>0:FRSB} may be found explicitly. The value for the unique triple $(q_0,q_*,\qc)$ immediately follows from the equations
		\[\qc-q_0=\frac{1}{2}\left(\frac{c_0}{\mu^{\frac{3}{2(2+\gamma)}}}-1\right),\;\;\; q_0=c_1\mu^{\frac{3(\gamma+1)}{2(2+\gamma)}+3},\;\;\; \qc-q_*=\frac{1}{\beta\sqrt{\mu_{Lar}(\beta;t)t}}, \label{eqn:FRSB triple eqns explicit}\]
		with constants given by $c_0=(\frac{2\gamma(\gamma+1)}{\sqrt{t}})^{1/(\gamma+2)}$ and $c_1=\frac{\gamma t}{c_0^{1+\gamma}}$. The measure $\zeta$ can then be found from its CDF, which is given for $s\in [q_0,q_*)$ by 
		\[\zetac([0,s])=\frac{\gamma+2}{3\beta}\left(\frac{4}{t\gamma(\gamma+1)}\right)^{1/3}(1+2\qc-2s)^{\frac{\gamma-1}{3}}, \label{eqn:FRSB zeta eqns explicit}\]
		and otherwise by $\zetac([0,s])=0$ if $s\in [0,q_0)$ and $\zetac([0,s])=1$ if $s\in [q_*,\qc]$.
	\end{corr}
	
	\begin{remark}
		\label{remark:gamma=1 remark}
		In the critical case where $\gamma=1$, our proof will show that if the model is not RS, then (\ref{eqn:FRSB triple eqns explicit}) and (\ref{eqn:FRSB zeta eqns explicit}) still hold. However, in the case where $\gamma=1$, (\ref{eqn:FRSB zeta eqns explicit}) now describes a 1RSB measure. In particular, in this case $\zetac=m\delta_{q_0}+(1-m)\delta_{q_*}$ where $m=\frac{1}{\beta}\left(\frac{2}{t}\right)^{1/3}$ and $(q_0,q_*,\qc)$ are given by (\ref{eqn:FRSB triple eqns explicit}). We will discuss this after the proof of Corollary \ref{corr:FRSB}.
	\end{remark} 
	
	We now complete this section with three open questions. The first involves the generalization of these results to a wider class of $B$.
	
	\begin{open}
		Show that Theorems \ref{theorem: wandering: powerlaw} and \ref{theorem: parisi: powerlaw} hold under weaker assumption that $B(x)\asymp |x|^{-\gamma}$, and show that Theorems \ref{theorem: wandering: exponential} and \ref{theorem: parisi: exponential} hold under the weaker assumption that $B(x)\asymp e^{- \lambda x}$.
	\end{open}
	
	We believe this result is likely to hold, as in the massless model, the behavior of the model at long distances should only depend on the tail behavior of the correlation function. Our second problem is to relate the model in the massless continuum limit to one introduced by Rovira and Tindel \cite{RoviraTindel}. In particular, we believe that the following should hold.
	
	\begin{open}
		\label{open:continuum convergence}
		Show that after centering at the origin (i.e. taking $\b{u}(x)-\b{u}(0)$) our discrete model converges to the continuum polymer model of Rovira and Tindel \cite{RoviraTindel} (with $d=N$, $Q=B$ and $\beta=\beta$) in the limit $\mu \to 0$ and $L\to \infty$ (taken in either order).
	\end{open}
	
	This result would be of great interest as it could allow one to transfer our results on the limiting wandering exponent to this model. In particular, if one could show that one may interchange the limit in $N$ as well, this would allow one to transfer results for the wandering exponent to this model exactly.
	
	Finally, we give the following broader question on the nature of the replica symmetry breaking phase.
	
	\begin{open}
		\label{open: RSB}
		Show that when the model is RSB, the Gibbs measure can be asymptotically decomposed into a hierarchy of states as described in Section 6 of \cite{mezardparisi}.
	\end{open}
	
	This would relate our definition of RSB to the more in-depth picture in the physics literature. We mention that for spin glasses (which is similar to the $L=1$ case here), such a rigorous description has been obtained by Jagannath \cite{AukoshPureStates}.

			\subsection{History and Related Works}
			
			The earliest mathematical work on directed polymers is due to Imbrie and Spencer \cite{ImbrieSpencer}. In a discrete model with no spatial correlations, they showed that for $N\ge 3$, the polymer remains diffusive for all sufficiently high temperatures. Bolthausen \cite{bolt1} strengthened this result and showed that a central limit theorem holds under the same conditions. Comets and Yoshida \cite{comets1} extended this to show that diffusivity is equivalent to a condition on the partition function. As a consequence of this, they established that for $N\ge 3$, diffusivity holds precisely until some critical temperature, and fails for all temperatures below this. Moreover, for $N\le 2$, they show that this model is superdiffusive for all temperatures. This result is expected to hold generally for a spatially correlated environment as long as the correlations decay sufficiently quickly. 
			
			While these results establish superdiffusivity in the correct regime, they do not compute the exact wandering exponent for the model. In the future diffusive case, the wandering exponent has only been rigorously found in dimension one for a very specific choice of environment \cite{actualtwothirds}, and otherwise only bounds for $\eta$ are known (see \cite{LacoinCorrelation} and the references therein). We note that there is a very general bound due to M\'{e}jane \cite{Mejane}, which shows that $\eta\le 3/4$ for essentially any choice of $\beta$, $N$, and $B$ (see Proposition 1.9 of \cite{LacoinCorrelation}). This matches the bound of $\eta\ge 3/2(2+\gamma)$ given by Lacoin \cite{LacoinCorrelation} for models with polynomial decay in the limit $\gamma \to 0$.
			
			We also mention that there is another very well-studied exponent associated with polymer models - the fluctuation exponent. This is given roughly by the exponent in $L$ at which the variance of the quenched free energy scales. In addition to results on the wandering exponent, M\'{e}zard and Parisi \cite{mezardparisi} also give conjectures on the value of the fluctuation exponent for the cases of $B$ considered above, which we recall in the following open question.
			
			\begin{open}
				\label{open: variance lower bound}
				Show that 
				\[\lim_{\mu\to 0}\lim_{N\to \infty}\mathrm{Var}(N^{-1}F_{N,L}(\beta,\mu,t))\asymp L^{\chi},\]
				with $\chi=0$ if either $B(x)=e^{-x}$ or $B(x)=(1+x)^{-\gamma}$ with $\gamma\ge  1$, and with $\chi=\frac{3}{\gamma+2}-1$ if $B(x)=(1+x)^{-\gamma}$ with $\gamma<1$.
			\end{open}
			
			They obtain these results from a more general conjecture, which is that the fluctuation exponent is related to the wandering exponent via the relation $\chi=2\eta-1$. If one knew this relation in general, this open question would follow immediately from our results above. However, such a relation is far from being known in general, though see the work of Chatterjee \cite{MR3010809} for some results in this direction.
			
			We now turn to related works in the area of replica symmetry breaking. To begin, we note that the difference in behavior between the 1RSB and FRSB states has previously been studied from a number of angles. It can first be seen in terms of the structure of the Gibbs measure. This is because in the RSB phase, the Gibbs measure is expected to decompose into an infinite number of states. Moreover, these states are expected to be grouped by proximity into a hierarchy, whose structure resembles an infinite ultrametric tree. The structure of this tree is encoded in $\zetac$, and when the model is 1RSB, this tree is expected to have height one, while it is expected to be of infinite height when $\zetac$ is FRSB \cite{mezardparisi}. Indeed, such a decomposition has since been confirmed rigorously in a wide class of one-site (i.e., $L=1$) spin glass models. In particular, this was shown to hold for a class of generic Ising spin glass models by Talagrand \cite{TalagrandPurestates} and in further generality by Jagannath \cite{AukoshPureStates} (based on the resolution of the Parisi ultrametricity conjecture by Panchenko \cite{PanchenkoUltrametricty}), as well as in the case of the pure spherical $p$-spin model by Subag  \cite{subagPureStates}, and finally in a subset of the 1RSB regime in the mixed spherical spin model by Subag, Zeitouni, and the first author \cite{subag1RSB}.
			
			This difference can also be understood in terms of algorithmic barriers. In \cite{subagFRSBgroundstate}, Subag showed that for a certain class of FRSB mixed spherical spin glass models, there exists a polynomial-time algorithm that finds near-optimal configurations, whose energy is near the true ground-state energy. Conversely, in the context of both the mixed spherical and Ising spin glass models, it was shown by Huang and Sellke \cite{HSbetterBounds, Huang_Sellke_final} that such a near optimal configurations can be found by a ``stable" algorithm - a class which includes both gradient descent, Langevin dynamics, and approximate message passing ran for a finite amount of time - if and only if a certain generalization of the FRSB condition holds at zero temperature. These works build on earlier fundamental works of Montanari \cite{MontanariSK}, El Alaoui, Montanari and Sellke \cite{AMSalgothreshold}, and Gamarnik and Jagannath \cite{GamarnikJagganathOGP}, all of whom worked on the same problem, though with a restricted subset of algorithms and models. 
			
			A final perspective on the meaning of FRSB comes from the concept of marginal stability. One notion of marginal stability is essentially that a Gibbs measure is marginally stable if the Hessian of the Hamiltonian has eigenvalues near zero when evaluated at a typical point of the Gibbs measure. This can be visualized as saying that the Gibbs measure concentrates on a region where the Hamiltonian is almost flat in some directions, as opposed to something like a collection of deep, well-separated valleys. One may consult \cite{Marginal_Stability_2} for a more thorough discussion of marginal stability in the context of spin glass models. In \cite{SellkeFRSB}, Sellke showed that in the spherical spin glass model, marginal stability of the low temperature Gibbs measure implies that the model is FRSB at zero temperature, and moreover, they find a stronger condition on the Parisi measure which is sufficient. As it has been conjectured that finding any non-marginal near-optimal configuration is computationally hard \cite{Tofindastateitmustbemarginallystable}, this result can be viewed as a reason for the above algorithmic results. Heuristically, this connection is due to the fact that non-marginal stability implies that the Gibbs measure should satisfy the Overlap gap property \cite{GamarnikJagganathOGP, Gamarnik_newer}, a well-known criterion for algorithmic intractability in random structures.
			
			Our computation of the quenched free energy builds on the computation in the one-site case. Starting with the spherical model, an expression for the quenched free energy was given by Crisanti and Sommers \cite{crisantisommersOG}. A rigorous derivation of these formulas in the even case was obtained by Talagrand \cite{talagrandOG, talagrandIsingOG}, building upon an interpolation method used by Guerra to obtain the upper bound in the Ising case \cite{guerraOG}. The general Ising case was later obtained by Panchenko \cite{panchenkoUnipartite}, with the general spherical case being obtained by Chen \cite{weikuo}. 
			
			Moreover, a number of computations for the quenched free energy of different multi-site models have recently been obtained. Many of these rely on the foundational work of Panchenko \cite{panchenkoms}, who computed the quenched free energy of the multi-species Ising spin glass model in the convex case. Following this, Bates and Sohn \cite{erik,erikCS} computed the quenched free energy of the multi-species spherical spin glass model, again in the convex case. In the Euclidean case (which is what is actually considered in this paper) an asymptotic formula for the quenched free energy was essentially conjectured by M{\'e}zard and Parisi \cite{mezardparisi,mezardparisi2}, and was further studied in the one-site case by Engel \cite{engelOG}. The one-site computation was then made rigorous in the note of Klimovsky \cite{planeonesiteOG}, by employing the computation of the spherical case.
			
			The elastic polymer has also been studied, both in mathematics and physics literature, through the lens of topological complexity. This began with the work of Fyodorov, Le Doussal, Rosso, and Texier \cite{YanPolymer1}, as well as the works of Fyodorov and Le Doussal \cite{fyodorov-manifold,fyodorov-manifold-minimum} on the more general case of the elastic manifold. Some of these computations were later made rigorous through the work of Bourgade, McKenna, and the first author \cite{gerardbenpaul,gerardbenpaulCompanionPaper}. The works also build on related annealed complexity computations of Fyodorov et. al. \cite{Fyo04, FS07, FB08,fyodorovepointinabox}. Many of these results were later made rigorous and extended by Auffinger and Zeng \cite{tucazeng1}.
			
			We finally mention that this paper is the third in a series of works. In our previous works \cite{Paper1, Paper2}, we studied the discrete elastic manifold. This is roughly a massive Gaussian free field on a fixed discrete $d$-dimensional graph, accompanied by an i.i.d Gaussian potential function at each site. This recovers the model above in the case of $d=1$ up to normalization. In \cite{Paper1, Paper2} we consider the behavior of the model for a $d$-dimensional graph where the dimension of the environment is allowed to go to infinity, whereas in this work we are interested in the special case of $d=1$ where the number of points is allowed to diverge.
			
			We have restricted this work to the case of $d=1$, and leave the study of this model in the case of $d>1$ to further work. This is due in part to the fact that the $d=1$ case has seen the most interest, but also because there are a number of technical issues that make this generalization far from trivial. A key issue is that it is not even clear in what sense the continuum limit should exist. For example, if $d\ge 2$, the limiting value of the ground state energy is expected to be both divergent in the continuum limit and sensitive to the exact method of regularization \cite{fyodorov2025groundstateenergyfluctuations}. However, in the physics literature, the work of M\'{e}zard and Parisi \cite{mezardparisi}, as well as the latter works of Fyodorov and Le Doussal \cite{fyodorov2020manifolds,fyodorov-manifold-minimum} and  Fyodorov, Lacroix-A-Chez-Toine, and Le Doussal \cite{fyodorov2025groundstateenergyfluctuations} give an intricate, though non-rigorous, picture of this model in the $d\ge 2$ case. 
			
			In addition, recent work on such a random surface model has been undertaken in the mathematics literature. In \cite{ron3}, Dembin, Elboim, Hadas, and Peled studied a variant of this model at zero temperature where the environment has finite-range spatial correlation. They followed this up in \cite{ron2}, where Dembin, Elboim, and Peled also considered the model with a more strongly correlated environment. Among many things, they found a number of bounds for fixed $d$ and $N$ on the wandering exponent, as well as a number of bounds on the fluctuation exponent. These bounds become tight in the case of $d>4$, where the model effectively becomes trivial. These works \cite{ron2,ron3} also contain a fairly comprehensive survey of other works on these models, which we recommend to readers interested in the $d\ge 2$ case.

			\subsection{Organization}
			
			We now outline the structure of this paper. In Section \ref{section: continuum limits}, we will recall the computation of the quenched free energy for fixed $L$ from our companion papers \cite{Paper1, Paper2}. After this, we will introduce a number of formulas and functions related to our model for fixed $L$, and give asymptotic formulas for them in the limit as $L\to \infty$. 
			
			In Section \ref{section:proof of continuum free energy} we prove our main results on the continuum quenched free energy and the identification of the Parisi pair. Namely, we prove Theorems \ref{theorem:intro:main:Euclidean parameters identification: radius}, \ref{theorem:intro:main:Euclidean parameters identification: overlap}, \ref{theorem:intro:critical point equations}, \ref{theorem:intro:1-d continuum T>0:free energy T>0}, and \ref{theorem:intro:1-d continuum T>0:free energy T>0 variational form: second version}. The proofs for these statements involve a careful analysis of our formula for the limiting quenched free energy for fixed finite $L$. 
			
			In Section \ref{section: statics of gibbs measure}, we will prove our formula for the mean-squared displacement (Theorem \ref{theorem:formula for the displacement at massive mu}), as well as show our results on the wandering exponent (Corollary \ref{corr:wandering: rs} and Theorems \ref{theorem: wandering: powerlaw} and \ref{theorem: wandering: exponential}). These results will all be based on Proposition \ref{proposition:parisi perturbation}. which allows one to compute certain quadratic statistics of the Gibbs measure for finite $L$. This proposition is shown by studying the quenched free energy of a certain class of perturbations of our Hamiltonian, and then differentiating the resulting formula in the perturbation parameter. The remaining results involve analyzing the formulas in Proposition \ref{proposition:parisi perturbation} in the limit $L\to \infty$. 
			
			Then, in Section \ref{section:RS 1RSB FRSB results} we prove our main results which characterize when the Parisi measure is RS, 1RSB, or FRSB. Specifically, in this section we will prove Theorems \ref{theorem:massless RS phase transition theorem}, \ref{theorem:intro:1-d continuum T>0:RS}, \ref{theorem:intro:1-d continuum T>0:1RSB} and  \ref{theorem:intro:1-d continuum T>0:FRSB}, as well as Corollaries \ref{corr:intro:RS larkin} and \ref{corr:FRSB}. This section essentially consists of explicit computations based on the variational characterization of the Parisi measure given in Theorem \ref{theorem:intro:critical point equations}. Finally, in Appendix \ref{section: approximation proofs} we will prove a number of results used above that involve the behavior of certain fixed functions as one takes $L\to \infty$.
			
			\section{The Continuum Limit \label{section: continuum limits}}
			
			Our companion works \cite{Paper1, Paper2} provide a number of results when $L$ is fixed, so the main obstacle is taking the continuum limit (i.e., the limit $L\to \infty$). The purpose of this section, then, will be to introduce some of our basic approximation results, which we will use to understand how our results for fixed $L$ behave in the continuum limit. In addition, we will also recall some needed results from our companion works \cite{Paper1, Paper2} which concern the case of fixed $L$.
			
			Roughly, though, we will see below that in our results for fixed $L$, the $L$-dependence mostly arises through a number of functions of the discrete Laplacian $\Delta_L$. So studying the behavior of $\Delta_L$ in the limit $L\to \infty$ is key to establishing our results. Heuristically, the main idea is to relate the discrete Laplacian $\Delta_L$ to the continuum Laplacian $\Delta:=\partial_x^2$. However, we do not show such a convergence directly, as most of the functions of $\Delta_L$ we need (e.g., the trace of the resolvent) do not have a clear naive analogue for $\Delta$. Instead, we will show that our functions of $\Delta_L$ converge after suitable normalization to certain regularized functions of $\Delta$. 
			
			To further explain our method, we will recall our result on the quenched free energy (which we will henceforth simply call the free energy) for fixed $L$. For this, we will denote results from \cite{Paper1} using the prefix I, so that Theorem 1 of \cite{Paper1} would be written as Theorem I.1, Lemma 2 of \cite{Paper2} would be written as Lemma I.2, and so on. We will similarly denote results from \cite{Paper2} using the prefix II.
			
			We will now recall the finite-$L$ analogue of the Parisi-type functional $\cal{P}_\beta$ introduced above in Definition \ref{definition:parisi functional}. Fix a choice of $\beta,t>0$ and $L\in \N$. We define
			\[R_{1;L,t}(\mu)=L^{1/2}\tr((\mu I-t\Delta_L)^{-1}),\]
			where here and elsewhere we us $\tr(A)$ to denote the normalized trace (i.e. for $A\in \R^{n\times n}$, $\tr(A)=\frac{1}{n}\sum_{i=1}^n A_{ii}$). This is essentially the trace of the resolvent of $\Delta_L$ up to rescaling. Observe that $R_{1;L,t}:(0,\infty)\to (0,\infty)$ is a decreasing, smooth, diffeomorphism. In particular, it has an inverse function (with the same properties) which we will denote as $K_{L,t}(\mu)$. That is $K_{L,t}$ is defined by the cancellation relations
			\[R_{1;L,t}(K_{L,t}(\mu))=\mu \text{ and } K_{L,t}(R_{1;L,t}(\mu))=\mu.\]
			
			Next, choose $q\in (0,\infty)$ and $\zeta\in P([0,q))$. From the pair $(q,\zeta)$ we define a function $\delta:[0,q]\to [0,\infty)$ as in (\ref{eqn:def:delta-P}). Then for any $q_*\in (0,q)$ such that $\zeta([0,q_*])=1$, we define the functional
			
			\begin{align}
				\cal{P}^L_{\beta}(q,\zeta)=\frac{1}{2}&\bigg( \frac{1}{L^{1/2}}\log \det(\mu I-t\Delta_L)+\beta(q-q_*)K_{L,t}(\beta(q-q_*))\\&-\frac{1}{L^{1/2}}\log \det(K_{L,t}(\beta(q-q_*))I-t\Delta_L)\\
				&+\int_{0}^{q_*}\beta K_{L,t}(\beta\delta(u))du
				-2\beta^2\int_0^q \zeta([0,u])B'(2(q-u)) du-\beta\mu q\bigg).
				\label{eqn:Parisi function for fixed L}
			\end{align}
			It is easily checked that this functional is independent of the choice of $q_*$. We then have the following result from our companion paper \cite{Paper2} (up to rescaling).
			\begin{theorem}[Theorems II.1.1, II.1.4, and II.1.5]
				\label{theorem:paper 2:main free energy result}
				\[\lim_{N\to \infty}L^{-1/2}N^{-1}F_{N,L}(\beta,\mu,t)=\sup_{q\in (0,\infty)}\left(\inf_{\zeta\in P([0,q))}\cal{P}^L_{\beta}(q,\zeta)\right)=\cal{P}^L_{\beta}(q_{L},\zeta_{L}),\label{eqn:free-energy:ignore-34387}\]
				where here $(q_{L},\zeta_{L})$ is the unique pair that solves the following stationary point equations:
				\[\beta \int_0^q\zeta([0,u])du=R_{1; L,t}(\mu), \label{eqn:intro:minimization eqn:Euclidean: Larkin-finite L}\]
				\[\zeta\left(\big\{s\in [0,q]:f_{\beta,q}(s)=\sup_{0<s'<q}f_{\beta,q}(s') \big\}\right)=1,\label{eqn:intro:minimization eqn:Euclidean: measure-finite L}\]
				where for $s\in (0,q)$, we define
				\[f_{\beta,q}(s)=\int_0^sF_{\beta,q}(u)du,\;\;F_{\beta,q}(s)=-2B'(2(q-s))+\int_0^sK'_{L,t}(\beta \delta(u))du\]
				with $\delta$ is as in (\ref{eqn:def:delta-P}).
			\end{theorem}
			\begin{proof}
				As the notation in \cite{Paper2} is different from what is used here, the proof here will explain how to translate the results of \cite{Paper2}. Strictly, Theorems II.1.1, II.1.4, and II.1.5 provide a formula for
				\[\lim_{N\to \infty} L^{-1}N^{-1}\log \int_{(\R^{\Lambda})^N}\exp \left(-\beta \hat{\cal{H}}_{N,L}(\b{u})\right)d\b{u},\]
				both a.s. and in expectation, where $\cal{H}_{N,L}:(\R^{\Lambda_L})^N\to \R$ is the Hamiltonian 
				\[\hat{\cal{H}}_{N,L}(\b{u})=\frac{1}{2}\sum_{x\in \Lambda_L}\left( \mu \|\b{u}(x)\|^2-t(\Delta_L \b{u}(x),\b{u}(x))\right)+\sum_{x\in\Lambda_L}V_{N,x}(\b{u}(x)).\]
				Comparing this to the alternative description of the model in (\ref{eqn:alternative gibbs}) and (\ref{eqn:alternative hamiltonian}) above, we see that if we apply Theorems II.1.1, II.1.4, and II.1.5 in the case of $(\beta,L^{-1/2}\mu,L^{1/2}t,L^{-1/2}B)$ and $d=1$ (in the notation of \cite{Paper2}), we obtain a formula for $\lim_{L\to \infty} L^{-1}N^{-1}\log(\hat{Z}_{N,L})$, where $\hat{Z}_{N,L}(\beta,\mu,t)$ is the normalization factor in (\ref{eqn:alternative gibbs}). This $\hat{Z}_{N,L}(\beta,\mu,t)$ only differs from $Z_{N,L}(\beta,\mu,t)$ by a deterministic constant. Specifically, it is easily checked using (\ref{eqn:path measure of base rw})
				\begin{align}
					F_{N,L}(\beta,\mu,t)&=\log Z_{N,L}(\beta,\mu,t)=\log \hat{Z}_{N,L}(\beta,\mu,t)-\log \tilde{Z}_{N,L}(\beta,\mu,t)\\\
					&=\log \hat{Z}_{N,L}(\beta,\mu,t)-\frac{N}{2}\log \left(\left(2\pi/\beta\right)^LL^{L/2}\det(\mu I-t\Delta_L)^{-1}\right). \label{eqn:shift of Z}
				\end{align}
				In particular, one can routinely obtain (\ref{eqn:free-energy:ignore-34387}) from the formulas in \cite{Paper2} using (\ref{eqn:shift of Z}) and the asymptotic formula for $\log \det(\mu I-t\Delta_L)$ given directly after this proof. The remaining claims follow directly by applying the results in \cite{Paper2} in the case $(\beta,L^{-1/2}\mu,L^{1/2}t,L^{-1/2}B)$ and $d=1$.
			\end{proof}
			
			\begin{prop}
				\label{proposition: divergent factor} $\log (\det(\mu I-t\Delta_L))=L\log(L)+L\log(t)+2L^{1/2}\frac{\sqrt{\mu}}{\sqrt{t}}+O(L^{1/4})$.
			\end{prop}
			
			The proof of this proposition will be given in Appendix \ref{section: approximation proofs}. We see that by using Theorem \ref{theorem:paper 2:main free energy result}, to find the free energy in the limit $L\to \infty$, we only need to compute the limit
			\[\lim_{L\to \infty}\cal{P}^L_{\beta}(q_{L},\zeta_{L}).\]
			Ignoring the fact that $(q_L,\zeta_L)$ depends on $L$, a first step to showing this would be to show that for fixed $(q,\zeta)$, we have that (see Lemma \ref{lem:continuum free energy: limit of Parisi in L})
			\[\lim_{L\to \infty}\cal{P}^L_{\beta}(q,\zeta)=\cal{P}_{\beta}(q,\zeta),\]
			where $\cal{P}_{\beta}$ is the continuum Parisi function from Definition \ref{definition:parisi functional}. To do this, we need to understand the asymptotic behavior of the functions $K_{L,t}$ and $L^{-1/2}\log \det(\mu I-t\Delta_L)$. To understand $K_{L,t}$, it suffices to instead work with its inverse function $R_{1; L,t}$. We will show that roughly that $R_{1;L,t}(\mu)\approx\frac{1}{\sqrt{\mu t}}$ when $L$ is large. However, to do this, we need to better understand the matrix $\Delta_L$. 
			
			To begin, we study the eigenvalues of $\Delta_{L}$. By applying the discrete Fourier transform, we see that the eigenvalues of $\Delta_L$ are given by
			\[\left\{-4L\sin^2\left(\frac{k\pi}{L}\right):0\le k<L\right\}.\label{eqn:eigenvalues of laplacian}\]
			Now let $M_{L}(\R)$ denote the set of $\Lambda_L$-by-$\Lambda_L$ matrices. By the spectral theorem, for any conjugation-invariant function $f:M_{L}(\R)\to \R$
			\[\tr(f(\Delta_L))=\frac{1}{L}\sum_{k=0}^{L-1}f \left(-4L\sin^2\left(\frac{k\pi}{L}\right)\right).\]
			For $y\in [0,\pi/2]$ we have that $\sin^2(y)=\sin^2(1-y)$, so that for $1\le k\le [L/2]$, we have that $\sin^2(\pi k/L)=\sin^2(\pi (L-k)/L)$. Thus, all the eigenvalues come in pairs, except for $k=0$ and $k=L/2$ when $L$ is even. The eigenvalue corresponding to $k=L/2$ is $-4L\sin^2\left(\frac{\pi}{2}\right)=-4L$, so we see that for any conjugation-invariant function $f:M_{n,n}(\R)\to \R$ we have that
			\[\tr f(\mu I-t\Delta_L)=\frac{2}{L}\sum_{k=1}^{[(L-1)/2]}f(\mu+4tL\sin^2(\pi k/L))+\frac{1}{L}\left(f(\mu)+I_{L\text{ is even}}f(\mu+4tL)\right).\label{eqn:continuum functions: linear functions of f}\]
			We will now explain heuristically how this formula can be used to obtain asymptotics for $R_{1; L,t}(\mu)$ (a formal statement is given by Proposition \ref{proposition:continuum free energy proof:continuum functions-1}, whose proof is deferred to Appendix \ref{section: approximation proofs}). Neglecting the last two terms, this implies that
			\[R_{1;L,t}(\mu)\approx \frac{2}{L^{1/2}}\sum_{k=1}^{[(L-1)/2]}\frac{1}{\mu +4tL\sin^2(\pi k/L)}.\]
			One can show that this sum is dominated by the terms coming from the $k$, which are sub-linear in $L$. So we employ the small angle approximation (i.e. $\sin(x)\approx x$) to conclude that
			\[R_{1;L,t}(\mu)\approx \frac{2}{L^{1/2}}\sum_{k=1}^{[(L-1)/2]}\frac{1}{\mu +4tL(\frac{\pi k}{L})^2}= \frac{2}{L^{1/2}}\sum_{k=1}^{[(L-1)/2]}\frac{1}{\mu +4t\pi^2(\frac{k}{L^{1/2}})^2}.\]
			However, this last approximation can be recognized as a right Riemann sum (with spacing $\Delta x=L^{-1/2}$) for the function $\frac{1}{\mu+4t\pi^2 x^2}$, so in particular
			\[\frac{2}{L^{1/2}}\sum_{k=1}^{[(L-1)/2]}\frac{1}{\mu +4t\pi^2(\frac{k}{L^{1/2}})^2}\approx 2\int_{0}^{L^{-1/2}[(L-1)/2]}\frac{dx}{\mu+4t\pi^2 x^2}\approx 2\int_{0}^{\infty}\frac{dx}{\mu+4t\pi^2 x^2}.\]
			Now we simply compute that
			\[2\int_{0}^{\infty}\frac{dx}{\mu+4t\pi^2 x^2}=\frac{1}{\pi\sqrt{\mu t}}\int_{0}^{\infty}\frac{dx}{1+x^2}=\frac{1}{\sqrt{\mu t}}, \]
			so that finally we may conclude that
			\[R_{1;L,t}(\mu)\approx \frac{1}{\sqrt{\mu t}}.\label{eqn: rough R1 approx}\]
			Now as $K_{L,t}(x)$ is the inverse to $R_{1;L,t}(x)$, it should converge to the inverse of $\frac{1}{\sqrt{x t}}$, which is $\frac{1}{x^2 t}$, so that $K_{L,t}(x)\approx \frac{1}{x^2 t}$. Similar remarks apply to approximate $K_{L,t}'(x)\approx\frac{-2}{x^2t}$. Thus, altogether, we have motivated the following proposition. 
			
			\begin{prop}
				\label{proposition:continuum free energy proof:continuum functions-1}
				For any $\mu,t>0$,
				\[\lim_{L\to \infty}R_{1;L,t}(\mu)=\frac{1}{\sqrt{\mu t}},\;\; \lim_{L\to \infty}K_{L,t}(\mu)=\frac{1}{\mu^2t},\;\; \lim_{L\to \infty}K_{L,t}'(\mu)=-\frac{2}{\mu^3t}.\]
				Moreover, if one fixes $t>0$ and chooses some $\epsilon>0$, all these limits converge uniformly for all $\mu\in [\epsilon,\epsilon^{-1}]$.
			\end{prop} 
			
			The proof of this proposition will be given in Appendix \ref{section: approximation proofs}. With these approximations established, the only function in Theorem \ref{theorem:paper 2:main free energy result} left to analyze is $L^{-1/2}\log \det(x I-t\Delta_L)$. However, this function is in fact divergent in $L$ (see Lemma \ref{lem:continuum:asymptotics of the determinant} below). However, note that in $\cal{P}^L_{\beta}(q,\zeta)$ we don't deal with this function directly, but instead the difference of the function at two points. We can write this difference in terms of the derivative of $L^{-1/2}\log \det(\mu I-t\Delta_L)$, which is $R_{1; L,t}(\mu)$. Explicitly,
			\[L^{-1/2}\log \det(xI-t\Delta_L)-L^{-1/2}\log \det(y I-t\Delta_L)=\int_{x}^y R_{1;L,t}(z)dz\]
			Now as 
			\[\int_{x}^y \frac{dz}{\sqrt{z t}}=\frac{2\sqrt{y}}{\sqrt{t}}-\frac{2\sqrt{x}}{\sqrt{t}},\]
			we obtain the following corollary from the uniform convergence of $R_{1;L,t}$ in Proposition \ref{proposition:continuum free energy proof:continuum functions-1}.
			
			\begin{corr}
				\label{corr:continuum free energy proof: continuum functions-2}
				For fixed $x,y,t>0$, we have that
				\[\lim_{L\to \infty}\left(\frac{1}{L^{1/2}}\log \det(xI-t\Delta_L)-\frac{1}{L^{1/2}}\log \det(yI-t\Delta_L)\right)=\frac{2\sqrt{x}}{\sqrt{t}}-\frac{2\sqrt{y}}{\sqrt{t}}.\]
				Moreover, if one fixes $t>0$ and chooses some $\epsilon>0$, all these limits converge uniformly for all choices of $x,y\in [\epsilon,\epsilon^{-1}]$.
			\end{corr}
			
			With the discussion of our results on the free energy complete, we will complete this section by introducing a couple of functions that will appear in our analysis of the Parisi function in the next section. The first is
			\[R_{2;L,t}(\mu)=-\frac{d}{d\mu}R_{1;L,t}(\mu)=L^{1/2}\tr\left((\mu I+t\Delta_L)^{-2}\right).\]
			Like $R_{1;L,t}$, the function $R_{2;L,t}:(0,\infty)\to (0,\infty)$ is also a decreasing smooth diffeomorphism. Differentiating (\ref{eqn: rough R1 approx}) suggest the approximation $R_{2;L,t}(\mu)\approx \frac{1}{2\sqrt{\mu^3t}}$. For the second function, we note that by differentiating the relation $R_{1; L,t}(K_{L,t}(\mu))=\mu$, we see that
			\[-K_{L,t}'(\mu)=\frac{1}{R_{2;L,t}(K_{L,t}(\mu))}.\]
			In particular, we see that $-K_{L,t}':(0,\infty)\to (0,\infty)$ is invertible, and so we can define its inverse function as $U_{L,t}$. By inverting the approximation $-K_{L,t}'(\mu)\approx \frac{2}{\mu^3t}$, we predict that $U_{L,t}(\mu)\approx (\frac{2}{\mu t})^{1/3}$. This motivates the following result.
			
			\begin{prop}
				\label{proposition:continuum free energy proof:continuum functions-3}
				For any $\mu,t>0$,
				\[\lim_{L\to \infty}R_{2;L,t}(\mu)=\frac{1}{2\sqrt{\mu^3 t}},\;\;\;\; \lim_{L\to \infty}U_{L,t}(\mu)=\left(\frac{2}{\mu t}\right)^{1/3}\]
				Moreover, if one fixes $t>0$ and chooses some $\epsilon>0$, all these limits converge uniformly for all $\mu\in [\epsilon,\epsilon^{-1}]$.
			\end{prop} 
			
			\section{Results on The Free Energy \label{section:proof of continuum free energy}}
			
			In this section, we will prove Theorems \ref{theorem:intro:main:Euclidean parameters identification: radius}, \ref{theorem:intro:main:Euclidean parameters identification: overlap}, \ref{theorem:intro:critical point equations}, \ref{theorem:intro:1-d continuum T>0:free energy T>0}, and \ref{theorem:intro:1-d continuum T>0:free energy T>0 variational form: second version}, which constitute our main general results on the free energy as well as our characterizations of the Parisi pair $(\qc,\zetac)$. All of these theorems have finite-$L$ analogues, so we are primarily concerned with understanding how $(q_L,\zeta_L)$ and $\cal{P}_{\beta}^L$ behave in the limit as $L\to \infty$.
			
			To explain our method, we note that by applying the uniform approximations in Proposition \ref{proposition:continuum free energy proof:continuum functions-1} and Corollary \ref{corr:continuum free energy proof: continuum functions-2} to the terms in $\cal{P}_{\beta}^L(q,\zeta)$, one can easily see that \[\lim_{L\to \infty}\cal{P}_{\beta}^L(q,\zeta)=\cal{P}_{\beta}(q,\zeta).\label{eqn:ignore-2349237}\] 
			Thus to show Theorem \ref{theorem:intro:1-d continuum T>0:free energy T>0}, we just need to make this convergence uniform in some sense. Moreover, by doing the same to the stationary point equations which characterize the pair $(q_L,\zeta_L)$ in Theorem \ref{theorem:paper 2:main free energy result}, we obtain the stationary point equations which characterize the pair $(\qc,\zetac)$ in Theorem \ref{theorem:intro:critical point equations}. Thus we should expect that $(q_L,\zeta_L)\to (\qc,\zetac)$ in some sense.
			
			Our proofs will essentially formalize these heuristics. We will do this in three steps. The first step is to show a result (Lemma \ref{lem:continuum free energy: limit of Parisi in L}) which establishes that if $(q_L,\zeta_L)$ converges in an appropriate sense to some limit (possibly along some subsequence of $L$), then the free energy converges as well. The second step is to show that such subsequences exist, which will consist of showing that $(q_L,\zeta_L)$ live in a suitable compact subset of the parameter space. This will be done in Lemmas \ref{lem:continuum free energy: upper and lower bounds on q} and \ref{lem:preliminaries:q is not close to 1}. In the third step, we will show that given such a convergence, $(q_L,\zeta_L)$ must actually converge to $(\qc,\zetac)$. This will involve studying the stationary point equations for $\zeta_L$, and exploiting properties of the concave functions $q\mapsto \inf_{\zeta\in P([0,q))}\cal{P}_\beta^L(q,\zeta)$.
			
			The first step is the following lemma.
			
			\begin{lem}
				\label{lem:continuum free energy: limit of Parisi in L}
				Suppose there are sequences $q^L\in (0,\infty)$ and $\zeta^L\in P([0,q^L))$, and some fixed $q\in (0,\infty)$ and $\zeta\in P([0,q))$. Let us assume that as $L\to \infty$, that $q^L\to q$ and that $\zeta^L\to \zeta$ in the weak sense as probability measures on $[0,\infty)$. Finally, let us assume that there is some sequence of $q_{L,*}\in (0,q^L)$ such that $\zeta^L([0,q_{L,*}])=1$ such \[\inf_{L}\left(q^L-q_{L,*}\right)>0.\]
				Then we have that 
				\[\lim_{L\to \infty}\cal{P}^L_{\beta}(q^L,\zeta_{L})=\cal{P}_{\beta}(q,\zeta).\]
			\end{lem}
			
			\begin{proof}
				Our conditions imply that $\limsup_{L}q_{L,*}<q$, so we see that for some small $\epsilon>0$, we have that $q_{L,*}<\limsup_{L}q_{L,*}+\epsilon<q^L$ for sufficiently large $L$. Working with such large $L$ for the remainder of the proof, we may define $q_*=\limsup_{L}q_{L,*}$, which satisfies $\zeta^L([0,q_*])=1$ and $q_*<q^L-\epsilon$.
				
				Next, we denote by $\delta_L$ the function associated to $(q^L,\zeta^{L})$, reserving the notation $\delta$ for the function associated to $(q,\zeta)$. As $\zeta^L([0,s])\to \zeta([0,s])$ for a.e. $s$, we have that $\delta_L(s)\to \delta(s)$ point-wise for $s\in [0,q_*]$. As each $\delta_L$ and $\delta$ are Lipschitz with constant less than or equal to $1$, we may conclude that $\delta_L\to \delta$ uniformly on $[0,q_*]$. We also note that for all $s\in [0,q_{*}]$
				\[\epsilon\le q^L-q_{*}=\int_{q_{*}}^{q^L}\zeta([0,s])ds\le \delta(s)\le \int_0^{q^L}\zeta([0,s])ds\le q^L.\]
				Thus for some $C,c>0$ and $s\in [0,q_{*}]$ we have that $c\le \delta(s)\le C$. In particular, employing Proposition \ref{proposition:continuum free energy proof:continuum functions-1} and the uniform convergence of $\delta_{L}$ to $\delta$, we see that
				\[\lim_{L\to \infty}\int_{0}^{q_{*}}\beta K_{L,t}(\beta\delta_L(u))du= \int_{0}^{q_{*}}\frac{du}{\beta t \delta(u)^2}.\]
				A simpler application of Proposition \ref{proposition:continuum free energy proof:continuum functions-1} shows that
				\[\lim_{L\to \infty}\left(\beta(q^L-q_*)K_{L,t}(\beta(q^L-q_*))-\beta\mu q^L\right)=\frac{1}{\beta(q-q_*)t}-\beta\mu q,\]
				and by the dominated convergence theorem, we have that
				\[\lim_{L\to \infty}\int_0^{q^L}\zeta^L([0,u])B'(2(q^L-u)) du=\int_0^{q}\zeta([0,u])B'(2(q-u)) du.\]
				This just leaves the determinant terms, for which we use Proposition \ref{proposition:continuum free energy proof:continuum functions-1} and Corollary \ref{corr:continuum free energy proof: continuum functions-2} to compute that
				\begin{align}
					\lim_{L\to \infty} \bigg(\frac{1}{L^{1/2}}\log \det(\mu I-t\Delta_L)&-\frac{1}{L^{1/2}}\log \det(K_{L,t}(\beta(q-q_*))I-t\Delta_L)\bigg)\\
					&=\frac{2\sqrt{\mu}}{\sqrt{t}}-\frac{2}{\sqrt{\beta^2(q-q_*)^2t^2}}=\frac{2\sqrt{\mu}}{\sqrt{t}}-\frac{2}{\beta(q-q_*)t}.
				\end{align}
				Combining all these limits, term by term, completes the proof.
			\end{proof}
			
			We now proceed to the second step. For this, we need to control $(q_{L},\zeta_{L})$, the pair from Theorem \ref{theorem:paper 2:main free energy result}. The first bound we want is for $q_{L}$. We show that this family must live in a compact subset of $(0,\infty)$.
			
			\begin{lem}
				\label{lem:continuum free energy: upper and lower bounds on q}
				There are $C,c>0$ such that for all $L$
				\[c\le q_{L}\le C.\]
			\end{lem}
			\begin{proof}
				Our goal will be to find a good upper bound and lower bound for $\inf_{\zeta\in P([0,q))}\cal{P}_{\beta}^L(q,\zeta)$, which we will show forces $q_{L}$ to satisfy the desired bounds.
				
				For this, we first recall the evaluation of $\cal{P}_{\beta}^L$ on a Dirac $\delta$-function. Namely, Lemma II.6.1 implies that
				\begin{align}
					\cal{P}_{\beta}^L(q,\delta_{q_*})=&\frac{1}{2}\bigg(\frac{1}{L^{1/2}}\log\det(\mu I-t\Delta_L)+\beta q\left(K_{L,t}(\beta(q-q_*))-\mu\right)\\
					&-\frac{1}{L^{1/2}}\log\det(K_{L,t}(\beta (q-q_*))I-t\Delta_L ) + \beta^2\left(B(0)-B(2(q-q_*))\right) \bigg).
				\end{align}
				We trivially have that
				\[\inf_{\zeta\in P([0,q))}\cal{P}_{\beta}^L(q,\zeta)\le \inf_{q_*\in (0,q)}\cal{P}_{\beta}^L(q,\delta_{q_*}).\]
				Now we define the variable $s=K_{L,t}(\beta(q-q_*))$, so that $(q-q_*)=\beta^{-1}R_{1;L,t}(s)$. For $q_*\in (0,q)$, the variable $s$ now ranges over $(K_{L,t}(\beta q),\infty)$. Moreover, in terms of $s$, $\cal{P}_{\beta}^L(q,\delta_{q_*})$ may be more simply expressed as 
				\[
				\frac{1}{2}\bigg(\beta q\left(s-\mu\right)+\frac{1}{L^{1/2}}\log\det(\mu I-t\Delta_L)-\frac{1}{L^{1/2}}\log\det(sI-t\Delta_L ) + \beta^2(B(0)-B(2\beta^{-1}R_{1;L,t}(s))) \bigg).\label{eqn:ignore-32837492863}
				\]
				By convexity of $B$, for any $u>0$, we have that $B(0)-B(u)\le -B'(0)u$. Applying this to the final term of (\ref{eqn:ignore-32837492863}), we see that $\cal{P}_{\beta}^L(q,\delta_{q_*})\le \cal{Q}_{\beta,L,q}(s)$ where here $\cal{Q}_{\beta,L,q}(s)$ is defined as
				\[\frac{1}{2}\bigg(\beta q\left(s-\mu\right)+\frac{1}{L^{1/2}}\log\det(\mu I-t\Delta_L)-\frac{1}{L^{1/2}}\log\det(sI-t\Delta_L )-\beta B'(0)R_{1;L,t}(s) \bigg).\]
				Thus, we have obtained the upper bound
				\[\inf_{\zeta\in P([0,q))}\cal{P}_{\beta}^L(q,\zeta)\le \inf_{s\in (K_{L,t}(\beta q),\infty)}\cal{Q}_{\beta,L,q}(s).\]
				Employing Proposition \ref{proposition:continuum free energy proof:continuum functions-1} and Corollary \ref{corr:continuum free energy proof: continuum functions-2} again we see that
				\[\lim_{L\to \infty}\cal{Q}_{\beta,L,q}(s)=\cal{Q}_{\beta,q}(s),\] 
				where here
				\[\cal{Q}_{\beta,q}(s):=\frac{1}{2}\left(\beta q\left(s-\mu\right)+\frac{2\sqrt{\mu}}{\sqrt{t}}-\frac{2\sqrt{s}}{\sqrt{t}}-\frac{2 \beta B'(0)}{\sqrt{s t}}\right).\]
				In fact, employing the uniformity of these lemmas, we see that
				\[\limsup_{L\to \infty}\left(\inf_{s\in (K_{L,t}(\beta q),\infty)}\cal{Q}_{\beta,L,q}(s)\right)\le \inf_{s\in (\beta^{-2}q^{-2}t^{-1},\infty)}\cal{Q}_{\beta,q}(s),\]
				so that we finally obtain the upper bound
				\[\limsup_{L\to \infty}\left(\inf_{\zeta\in P([0,q))}\cal{P}_{\beta}^L(q,\zeta)\right)\le \inf_{s\in (\beta^{-2}q^{-2}t^{-1},\infty)}\cal{Q}_{\beta,q}(s).\label{eqn:lemma:ignore-1758}\]
				
				Next, we will obtain a simpler lower bound. For this, we note that by Jensen's inequality
				\begin{align}
					\E\int_{(\R^N)^{\Lambda_L}}&\exp\left(-\sum_{x\in \Lambda_L}\beta V_{N,x}(\b{u}(x))L^{-1/4}\right)\P_{\beta,\mu,N,L}(d\b{u})\\
					&\ge \int_{(\R^N)^{\Lambda_L}}\exp \left(-\sum_{x\in \Lambda_L}\beta \E V_{N,x}(\b{u}(x))L^{-1/4}\right)\P_{\beta,\mu,N,L}(d\b{u})=1,
				\end{align}
				so that in particular
				\[\E \log Z_{N,L}(\beta,t,\mu)\ge 0.\]
				Thus by Theorem \ref{theorem:paper 2:main free energy result}, we must have that 
				\[\inf_{\zeta\in P([0,q_L))}\cal{P}_{\beta}^L(q_{L},\zeta)\ge 0. \label{eqn:ignore-23894728937}\]
				
				Now we will use these bounds on $\cal{P}_{\beta}^L$ to obtain our bounds on $q_L$. By plugging the boundary value of the infimum in (\ref{eqn:lemma:ignore-1758}), we get that
				\[\cal{Q}_{\beta,q}(\beta^{-2}q^{-2}t^{-1})=\frac{1}{2}\left(-\frac{1}{\beta qt}-\beta q\mu +\frac{2\sqrt{\mu}}{\sqrt{t}}-2\beta^2 qB'(0)\right),\]
				and so we see that 
				\[\lim_{q\to 0}\cal{Q}_{\beta,q}(\beta^{-2}q^{-2}t^{-1})=-\infty.\label{eqn:lemma:ignore-3934}\]
				Combined with (\ref{eqn:lemma:ignore-1758}) and (\ref{eqn:ignore-23894728937}) this shows that there is some $c>0$ such that $q_L\ge c$. Next, on the point $\beta^{-2}q^{-1}t^{-1}$ we have that
				\[\cal{Q}_{\beta,q}(\beta^{-2}q^{-1}t^{-1})=\frac{1}{2}\left( \frac{1}{\beta t}-\beta q\mu+\frac{2\sqrt{\mu}}{\sqrt{t}}-\frac{2q^{1/2}}{t\beta}-2 \beta^2 B'(0)q^{-1/2}\right).\]
				This lives in the domain of the infimum as long as $q\ge 1$, and one easily sees that 
				\[\lim_{q\to \infty}\cal{Q}_{\beta,q}(\beta^{-2}q^{-1}t^{-1})=-\infty.\label{eqn:lemma:ignore-3935}\]
				Combining this with (\ref{eqn:lemma:ignore-1758}) and (\ref{eqn:ignore-23894728937}) again, this implies there is $C>0$ such that $q_L\le C$.
			\end{proof}
			
			Next, we will bound the quantity $q_{L}-q_{L,*}$ away from zero. This is not only a requirement to apply Lemma \ref{lem:continuum free energy: limit of Parisi in L}, but also if one assumes that $q_L$ converges to some $q$, it can be used to show that the measures $\zeta_L$ actually live in space of probability measures on $[0,q-\epsilon]$ (which is compact) for small enough $\epsilon>0$ and large enough $L$. For this, we will need the following basic inequality.
			
			\begin{lem}
				Let $A$ be any positive semi-definite matrix, and let $K$ be the functional inverse of the trace of its resolvent (i.e. $\tr(A+K(u)I)^{-1}=u$). Then for any $x,y>0$ we have that
				\[
				\frac{K(x)-K(x+y)}{y}\ge \frac{K(x+y)}{x}.\label{eqn:ignore-19699}
				\]
			\end{lem}
			\begin{proof}
				Note that by Hilbert's resolvent identity, we have for $z,w>0$ that
				\[\frac{\tr(A+zI)^{-1}-\tr(A+wI)^{-1}}{w-z}=\tr (A+zI)^{-1}(A+wI)^{-1}\le z^{-1}\tr(A+wI)^{-1}.\]
				Now if we take write $z=K(x+y)$ and $w=K(x)$, we see that 
				\[
				\frac{y}{K(x)-K(x+y)}=\frac{\tr(A+zI)^{-1}-\tr(A+wI)^{-1}}{w-z}\le \frac{\tr(A+wI)^{-1}}{z}=\frac{x}{K(x+y)},
				\]
				which is the reciprocal of (\ref{eqn:ignore-19699}).
			\end{proof}
			
			Note that this lemma applies to the function $K_{L,t}(\mu)$. Moreover, it is easily checked that its result (\ref{eqn:ignore-19699}) applies to the continuum analogue, $\frac{1}{\mu^2t}$, as well. The next result will show that one of the critical point equations in Theorem \ref{theorem:intro:critical point equations} suffices to give a uniform lower bound on $q_{L}-q_{L,*}$. We will give this lemma in more generality, though, as it will be useful to treat arbitrary values of $q$ (i.e., not just $q=q_L$).
			
			\begin{lem}
				\label{lem:preliminaries:q is not close to 1 general}
				Let either $K(\mu)=K_{L,t}(\mu)$ and $U(\mu)=U_{L,t}(\mu)$ for some $L\in \N$ or $K(\mu)=\frac{1}{\mu^2 t}$ and $U(\mu)=\left(\frac{2}{\mu t}\right)^{1/3}$. Next, fix some $q>0$ and $\zeta\in P([0,q))$, and let $q_*\in [0,q)$ denote the supremum of the support of $\zeta$. Then if we have that 
				\[\zeta \left(\{s\in [0,q]: f(s)=\sup_{0<s'<q}f(s')\}\right)=1,\label{eqn:lemma-min-con-1}\]
				where here for $s\in [0,q]$,
				\[F(s)=-B'(2(q-s))+\int_0^s K'(\beta \delta(u))du,\;\;\; f(s)=\int_0^s \zeta([0,u])du,\]
				then we have that
				\[
				\beta(q-q_*)\ge \min\left(U(4B''(0)),\frac{1}{4B''(0)}K\left(\int_{0}^{q}\beta \zeta([0,u])du\right),U\left(\frac{-4B'(0)}{q}\right),\frac{\beta q}{2}\right).
				\]
			\end{lem}
			\begin{proof}
				We first show that (\ref{eqn:lemma-min-con-1}) implies that
				\[\zeta \left(\{s\in [0,q]: F(s)=0\}\right)=1.\label{eqn:lemma-min-con}\]
				For this, observe that 
				\[\{s\in [0,q]: f(s)=\sup_{0<s'<q}f(s')\}\subseteq \{s\in [0,q]: F(s)=0\}\cup \{0,q\},\]
				so we only need to show that $\zeta(\{0,q\})=0$. As $\zeta\in P([0,q))$, we have that $\zeta(\{q\})=0$. Moreover, $F(0)=-B'(0)>0$, so by (\ref{eqn:lemma-min-con-1}), $\zeta(\{0\})=0$ as well. In particular, we have that (\ref{eqn:lemma-min-con}) holds.
				
				From here, the proof will be broken into three separate cases.
				\begin{mycases}
					\item \textbf{$q_*$ is an accumulation point of $\supp(\zeta)\setminus \{q_*\}$.} 
					This assumption implies that there are points in $\supp(\zeta)\setminus\{q_*\}$ arbitrarily close to $q_*$. The continuity of $F$ and the condition (\ref{eqn:lemma-min-con}) forces that $F(s)=0$ for $s\in \supp(\zeta)$, so we must have that $F'(q_*)=0$. Thus
					\[
					0=F'(q_*)=4B''(2(q-q_*))+K'(\beta(q-q_*)).
					\]
					As $U$ is decreasing, this shows that
					\[\beta(q-q_*)=U(4B''(2(q-q_*)))\ge U(4B''(0)).\]
					
					\item \textbf{$q_*$ is not an accumulation point of $\supp(\zeta_{L})\setminus \{q_*\}$ and $\zeta$ is not a $\delta$-mass.}
					By the first assumption, we must have that $m=\zeta(\{q_*\})>0$. By the second assumption we may find some smaller $q'\in \supp(\zeta)$ such that $\zeta((q',q_*))=0$. As the condition (\ref{eqn:lemma-min-con}) forces that $F(q')=F(q_*)=0$, we have that
					\[
					0=F(q_*)-F(q')=-2B'(2(q-q_*))+2B'(2(q-q'))+\int_{q'}^{q_*}K'(\beta \delta(u))du.\label{eqn:ignore-19203}
					\]
					For $u\in (q',q_*)$, we have $\delta(u)=(q-q_*)+m(q_*-u)$, so we have that
					\[
					\int_{q'}^{q_*}K'(\beta \delta(u))du=\frac{-1}{\beta m}\bigg(K(\beta (q-q_*))-K(\beta (q-q_*)+\beta m(q_*-q'))\bigg).
					\label{eqn:ignore-1959}
					\]
					By (\ref{eqn:ignore-19699}) (which holds for either case of $K$) we have that
					\[
					\frac{K(\beta (q-q_*))-K(\beta (q-q_*)+\beta m(q_*-q'))}{\beta m(q_*-q')}\ge \frac{K(\beta (q-q_*)+\beta m(q_*-q'))}{\beta(q-q_*)}.
					\]
					Thus by (\ref{eqn:ignore-19203}) and (\ref{eqn:ignore-1959}) we have that
					\[\frac{-2B'(2(q-q_*))+2B'(2(q-q'))}{q_*-q'}\ge \frac{K(\beta (q-q_*)+\beta m(q_*-q'))}{\beta(q-q_*)}.\label{eqn:ignore-1029302}\]
					We further note that
					\[\beta (q-q_*)+\beta m(q_*-q')=\int_{q'}^{q_*}\beta \zeta([0,u])du\le \int_{0}^{q}\beta \zeta([0,u])du.\]
					As $K$ is decreasing, this gives that
					\[
					K(\beta (q-q_*)+\beta m_\beta(q_*-q'))\ge K\left(\int_{0}^{q}\beta \zeta([0,u])du\right).\label{eqn:ignore-23049874092834}
					\]
					On the other hand, by the mean value theorem, we have that
					\[\frac{-2B'(2(q-q_*))+2B'(2(q-q'))}{q_*-q'}\le \|4B''\|_{\infty}= 4B''(0).\]
					With these, (\ref{eqn:ignore-1029302}) implies that
					\[4B''(0)\ge \frac{1}{\beta(q-q_*)}K\left(\int_{0}^{q}\beta \zeta([0,u])du\right),\]
					which may be rearranged to give our bound.
					
					\item \textbf{$\zeta$ is a Dirac measure}. 
					In this case, we must have that $\zeta=\delta_{q_*}$. The equation(\ref{eqn:lemma-min-con}) implies that
					\[0=F(q_*)=-2B'(2(q-q_*))+q_*K'(\beta(q-q_*))=0.\]
					It is sufficient to consider the case that $q_*>q/2$, as if $q_*\le q/2$, we have that $\beta(q-q_*)\ge \beta q/2$. However if $q_*>q/2$ then we have that
					\[-2B'(0)\ge -2B'(2(q-q_*))=-q_*K'(\beta(q-q_*))\ge \frac{q}{2}\left(-K'(\beta(q-q_*))\right),\]
					Dividing by $q/2$ and then applying the decreasing function $U$ completes this case.
				\end{mycases}
			\end{proof}
			
			When $q=q_L$, we get the following tighter result, as we have an additional stationary point equation to use.
			
			\begin{lem}
				\label{lem:preliminaries:q is not close to 1}
				If $q_{L,*}\in [0,q)$ be the supremum of the support of $\zeta_L$. Then we have that
				\[
				\beta(q_{L}-q_{L,*})\ge \min\left(U_{L}(4B''(0)),\frac{\mu}{4B''(0)},R_{1;L,t}(\mu)\right).
				\]
			\end{lem}
			\begin{proof}
				The proof of the first two cases follows exactly as in Lemma \ref{lem:preliminaries:q is not close to 1 general}. However, in Case 3 (i.e., when $\zeta_L$ is a Dirac mass), we may employ the equation (see Theorem \ref{theorem:paper 2:main free energy result})
				\[R_{1;L,t}(\mu)=\beta\int_0^{q_L}\zeta([0,u])du=\beta(q_L-q_{L,*}),\label{eqn:ignore-323847}\]
				so that
				\[\beta (q_L-q_{L,*})=R_{1;L,t}(\mu).\]
				By using this observation instead of the original proof of Case 3, and simplifying the result in the second case by again using (\ref{eqn:ignore-323847}), we obtain the desired bound.
			\end{proof}
			
			This completes the second step. For the final step, we first want to show that for fixed $q$, the minimizer of $\cal{P}_{\beta}^L$ converges to the minimizer of $\cal{P}_{\beta}$ as $L\to \infty$. For this, we first recall from our companion paper some basic facts about the functional $\cal{P}_{\beta}^L$.
			
			\begin{lem}[Corollary II.4.5]
				\label{lem: fixed L unique minimizer at fixed q}
				Fix $q>0$ and $L\in \N$. Then the function $\cal{P}_{\beta}^L$ is strictly convex on $P([0,q))$ with a unique minimizer $\zeta_{q,L}$ which is the unique solution to the equation 
				\[\zeta \left(\{s\in [0,q]:f(s)=\sup_{0<s'<q}f(s')\}\right)=1,\]
				where here for $s\in [0,q]$,
				\[F(s)=-B'(2(q-s))+\int_0^s K'_{L,t}(\beta \delta(u))du,\;\;\; f(s)=\int_0^sF(u)du,\]
				and $\delta$ is as in (\ref{eqn:def:delta-P}).
			\end{lem}
			
			We will also need an analogue of this fact for $\cal{P}_{\beta}$.
			
			\begin{lem}
				\label{lem: continuum unique minimizer at fixed q}
				Fix $q>0$. Then the function $\cal{P}_{\beta}$ is strictly convex on $P([0,q))$ with a unique minimizer $\zeta_{q}$ which is the unique solution to the equation 
				\[\zeta \left(\{s\in [0,q]:f(s)=\sup_{0<s'<q}f(s')\}\right)=1,\label{eqn:ignore-023478628736}\]
				where here for $s\in [0,q]$,
				\[F(s)=-B'(2(q-s))-\int_0^s \frac{2}{(\beta \delta(u))^3t}du,\;\;\; f(s)=\int_0^sF(u)du,\]
				and $\delta$ is as in (\ref{eqn:def:delta-P}).
			\end{lem}
			
			The proof of this lemma is omitted as it is a routine adaptation of the proof of Corollary II.4.5 in \cite{Paper2}, if one simply replaces all the functions with their continuum analogues.\footnote{In fact, as we are only generalizing a restricted version of Corollary II.4.5 here, one only needs to modify the proof of Theorem II.1.14 (in the case of $h=0$) to the functional $\zeta\mapsto \cal{P}_{\beta}(q,\zeta)$, which is done easily.} The next result shows that the minimizer of $\cal{P}_{\beta}^L$ converges to the minimizer of $\cal{P}_{\beta}$ for fixed $q$.
			
			\begin{lem}
				\label{lem: limit of infimum for fixed q}
				Fix a choice of $q\in (0,\infty)$. Let $\zeta_q^L\in P([0,q))$ be the unique minimizer of $\cal{P}^L_{\beta,q}$ and let $\zeta_q$ be the unique minimizer of $\cal{P}_{\beta}$. Then $\zeta_q^L\to \zeta_q$ in the weak sense.
			\end{lem}
			\begin{proof}
				To show the first statement, it suffices to show that for any subsequence of $L$, there is a further subsequence of $L$, such that $\zeta_q^L$ converges weakly to some $\zeta\in P([0,q))$ which satisfies (\ref{eqn:ignore-023478628736}). Indeed, by Lemma \ref{lem: continuum unique minimizer at fixed q}, we must have that $\zeta=\zeta_q$, and so as all these subsequences converge to the same element, this implies that the entire sequence converges to $\zeta_q$.
				
				By Lemma \ref{lem: fixed L unique minimizer at fixed q}, we may apply Lemma \ref{lem:preliminaries:q is not close to 1 general} and the inequality $\int_{0}^{q} \zeta_q^L([0,u])du\le q$ to conclude that $\supp\left(\zeta_{q}^L\right)\subseteq [0,q-\epsilon_L]$, where
				\[\epsilon_L= \frac{1}{\beta}\min\left(U_{L,t}(B''(0)),\frac{K_{L,t}(\beta q)}{4B''(0)},U_{L,t}\left(\frac{-4B'(0)}{q}\right),\frac{\beta q}{2}\right). \label{eqn:lemma:bound on continuum gap finite} \]
				By Propositions \ref{proposition:continuum free energy proof:continuum functions-1} and \ref{proposition:continuum free energy proof:continuum functions-3}, we see that
				\[\epsilon:=\lim_{L\to \infty}\epsilon_L=\frac{1}{\beta }\min\left(\left(\frac{2}{B''(0)t}\right)^{1/3},\frac{1}{\beta^2q^2tB''(0)},\left(\frac{2q}{-4B'(0)t}\right)^{1/3},\frac{\beta q}{2}\right)>0.\label{eqn:lemma:bound on continuum gap}\]
				Thus for all sufficiently large $L$ we have that $\supp\left(\zeta_{q}^L\right)\subseteq [0,q-\epsilon/2]$. We will now assume that $L$ is large enough that this holds.
				
				Now let $P([0,s])$ denote the space of probability measures on $[0,s]$. If we are given some subsequence of $L$, we may employ Prokhorov's theorem to pass to a further subsequence such that $\zeta_{q}^L$ weakly converges to some $\zeta\in P([0,q-\epsilon/2])\subseteq P([0,q))$. Now let us denote by $\delta_L$ the function associated to $(q,\zeta_q^L)$ as in (\ref{eqn:def:delta-P}) and define the functions
				\[F_{L}(u)=-B'(2(q-u))-\int_0^u K'_{L,t}(\beta \delta_L(s))ds,\;\; f_{L}(s)=\int_0^s F_{L}(u)du.\]
				By Lemma \ref{lem: fixed L unique minimizer at fixed q}, we have that
				\[\zeta_{q}^L\left(\big\{s\in [0,q]:f_{L}(s)=\sup_{0<s'<q}f_{L}(s') \big\}\right)=1.\label{eqn:lemma:1544}\]
				If we further denote by $\delta$ the function associated to $(q, \zeta)$ as in (\ref{eqn:def:delta-P}), and then we define
				\[F(u)=-B'(2(q-u))-\int_0^u \frac{2ds}{(\beta \delta(s))^3t},\;\; f(s)=\int_0^s F(u)du,\]
				we see that to complete the proof, we must show that
				\[\zeta\left(\big\{s\in [0,q]:f(s)=\sup_{0<s'<q}f(s') \big\}\right)=1.\]
				For this, define for each $\delta>0$ the set
				\[S_\delta=\big\{s\in [0,q]:|f(s)-\sup_{0<s'<q}f(s')|\le \delta \big\}.\]
				Then it suffices to show for every $\delta>0$ that $\zeta\left(S_\delta\right)=1$.
				
				Now as $\lim_{L\to \infty}\zeta_q^L([0,s])=\zeta([0,s])$ for a.e. $s$, we see by the dominated convergence theorem that for any $u\in [0,q]$, we have that $\lim_{L\to \infty}\delta_L(u)=\delta(u)$. As each $\delta_L$ and $\delta$ is Lipschitz with constant bounded above by 1, we further see that $\delta_L\to \delta$ uniformly. Moreover, for $\epsilon'<\epsilon/2$, and any $u\in [0,q-\epsilon']$, we have that $\delta_L(u)\ge \epsilon'$ and $\delta(u)\ge \epsilon'$. We also always have the trivial upper bound $\delta(u),\delta_L(u)\le q$. Thus by Proposition \ref{proposition:continuum free energy proof:continuum functions-1}, we see that $K_{L,t}'(\beta \delta_L(s))\to \frac{2}{(\beta \delta(s))^3t}$ uniformly for$s\in [0,q-\epsilon']$ for any $\epsilon'>0$. In particular, both $F_L\to F$ and $f_L\to f$ uniformly on $[0,q-\epsilon']$.
				
				Now as $\zeta_q^L([0,q-\epsilon/2])=1$, we see by (\ref{eqn:lemma:1544}) that we must have that
				\[\sup_{0<s'<q}f_L(s')=\sup_{0<s'\le q-\epsilon/2}f_L(s').\]
				By using the uniform convergence of $f_L$ to $f$ on $[0,q-\epsilon']$, which holds for any $\epsilon'>0$, this implies that
				\[\sup_{0<s'<q}f(s')=\sup_{0<s'\le q-\epsilon/2}f(s')\]
				as well. Thus using the uniform convergence of $f_L\to f$ on $[0,q-\epsilon/2]$, we see that for any fixed $\delta>0$ and large $L$,
				\[\{s\in [0,q]:f_L(s)=\sup_{0<s'<q}f_L(s')\}\subseteq S_\delta.\]
				Thus by (\ref{eqn:lemma:1544}) we see that $\zeta(S_\delta)=1$, which completes the proof.
			\end{proof}
			
			We are now ready to give the proof of Theorem \ref{theorem:intro:1-d continuum T>0:free energy T>0}.
			
			\begin{proof}[Proof of Theorem \ref{theorem:intro:1-d continuum T>0:free energy T>0}]
				We define the functions
				\[\mathscr{P}^L(q)=\inf_{\zeta\in P([0,q))}\cal{P}^L_{\beta}(q,\zeta)=\cal{P}^L_{\beta}(q,\zeta_{q}^L),\;\;\; \mathscr{P}(q)=\inf_{\zeta\in P([0,q))}\cal{P}_{\beta}(q,\zeta)=\cal{P}_{\beta}(q,\zeta_q),\]
				where $\zeta_q^L$ and $\zeta_q$ are as in Lemma \ref{lem: limit of infimum for fixed q}. By Lemma \ref{lem: limit of infimum for fixed q}, we have that $\zeta_q^L\to \zeta_q$ in the weak sense. Moreover, it is clear from the proof of Lemma \ref{lem: limit of infimum for fixed q} (or alternatively Lemma \ref{lem:preliminaries:q is not close to 1 general}), that there is some $\epsilon:=\epsilon(q)>0$ such that $\zeta_{q}([0,q-\epsilon])=\zeta_q^L([0,q-\epsilon])=1$ for large $L$.
				Thus, by employing Lemma \ref{lem:continuum free energy: limit of Parisi in L} we see that $\lim_{L\to \infty}\mathscr{P}^L(q)=\mathscr{P}(q)$.
				
				We note that this point-wise convergence implies that
				\[\sup_{q\in (0,\infty)}\mathscr{P}(q)\le \liminf_{L\to \infty}\sup_{q\in (0,\infty)}\mathscr{P}^L(q).\]
				Moreover, by Lemma II.4.8, each $\mathscr{P}^L$ is concave in $q\in (0,\infty)$. In particular, $\mathscr{P}(q)$ must also be concave in $q$, and so we have that $\mathscr{P}^L\to \mathscr{P}$ uniformly over compact subsets of $q\in (0,\infty)$. Combining this with Lemma \ref{lem:continuum free energy: upper and lower bounds on q}, we conclude that we may find $C,c>0$ such that 
				\[\limsup_{L\to \infty}\sup_{q\in (0,\infty)}\mathscr{P}^L(q)=\limsup_{L\to \infty}\sup_{q\in [c,C]}\mathscr{P}^L(q)\le \sup_{q\in [c,C]}\mathscr{P}(q)\le \sup_{q\in (0,\infty)}\mathscr{P}(q).\]
				In particular,
				\[\lim_{L\to \infty}\sup_{q\in (0,\infty)}\mathscr{P}^L(q)=\sup_{q\in (0,\infty)}\mathscr{P}(q),\]
				completing the proof of the convergence in expectation. 
				
				To show that this convergence also holds almost surely, we first note that $\Var(\cal{H}_{N, L}(\b{u}))=L^{1/2}\beta^2 B(0)$, and so by applying Proposition I.A.2, we have for any $\epsilon>0$ that
				\[\P\left(|L^{-1/2}N^{-1}\log Z_{N,L}(\beta,\mu t)-L^{-1/2}N^{-1}\E\log Z_{N,L}(\beta,\mu t)|>\epsilon\right)\le \exp \left(-\frac{L^{1/2}N^2\epsilon^2}{4\beta^2 B(0)}\right).\]
				This easily yields the desired convergence by using the Borel-Cantelli lemma.
			\end{proof}
			
			With the proof of Theorem \ref{theorem:intro:1-d continuum T>0:free energy T>0} complete, we now turn to our remaining results. We note that all of these results involve either $\qc$ or $\zetac$. However, to know these quantities are well-defined, we would need to prove Theorems \ref{theorem:intro:main:Euclidean parameters identification: radius} and \ref{theorem:intro:main:Euclidean parameters identification: overlap} to even be defined. For this reason, it will be much more technically convenient to instead first show that $\cal{P}_{\beta}$ has a unique critical point, $(q_e,\zeta_e)$, and then show that $(q_L,\zeta_L)$ converges to $(q_e,\zeta_e)$. This will allow us to prove Theorems \ref{theorem:intro:critical point equations} and \ref{theorem:intro:1-d continuum T>0:free energy T>0 variational form: second version}, except with $(\qc,\zetac)$ replaced by $(q_e,\zeta_e)$. Once these have been shown, we will show Theorems \ref{theorem:intro:main:Euclidean parameters identification: radius} and \ref{theorem:intro:main:Euclidean parameters identification: overlap} and prove that $(q_e,\zeta_e)=(\qc,\zetac)$ by using some analogous finite-$L$ results from \cite{Paper2}.
			
			We first need to control the critical points of $\cal{P}_{\beta}$. For this, we first need the following analytic result, which generalizes Corollary II.4.7 and Lemma II.4.8 (which are for fixed finite $L$).
			
			\begin{lem}
				\label{lem:uniqueness of maximizer of P}
				The function $\mathscr{P}(q):=\inf_{\zeta\in P([0,q))}\cal{P}_{\beta}(q,\zeta)$ is concave in $q$. Moreover, the function $\mathscr{P}(q)$ has a unique maximizer, which is given by the solution to the equation
				\[\beta \int_0^q \zeta_q([0,u])du=\frac{1}{\sqrt{\mu t}},\label{eqn:ignore-2347823647}\]
				where $\zeta_q\in P([0,q))$ is the minimizer given in Lemma \ref{lem: continuum unique minimizer at fixed q}.
			\end{lem}
			\begin{proof}
				In the proof of Theorem \ref{theorem:intro:1-d continuum T>0:free energy T>0} we showed that the function $\mathscr{P}(q)$ is concave in $q$. To show the remaining claims, we begin by showing that the support of $\zeta_{q}$ is bounded away from $0$ uniformly in $q$. Let $m\in [0,q)$ be the largest value such that $\zeta_{q}([0,m))=0$. Using the function $F$ from Lemma \ref{lem: continuum unique minimizer at fixed q} from the choice of $\zeta=\zeta_q$, we note that 
				\[F(0)=-B'(2q)>0,\]
				so Lemma \ref{lem: continuum unique minimizer at fixed q} implies that $m>0$. Moreover, if we let $m_0\in [0,q)$ be the smallest solution to $F(m')=0$, then Lemma \ref{lem: continuum unique minimizer at fixed q} implies that $m\ge m_0>0$. Now note that for $s\in [0,m_0]$, we have that $\delta(s)=\delta(0)$. In particular, we have that
				\[0=F(m_0)=-B'(2(q-m_0))-\frac{m_0}{\beta^3\delta(0)^3t}.\]
				This implies that
				\[m_0=-B'(2(q-m_0))\beta^3\delta(0)^3t\ge -B'(2q)\beta^3\delta(0)^3t.\]
				Now using the proof of Lemma \ref{lem: limit of infimum for fixed q}, we have that $\zeta([0,q-\epsilon_q])=1$ where $\epsilon_q<q$ is given by (\ref{eqn:lemma:bound on continuum gap}). In particular, we have that $\delta(0)\ge q-\epsilon_q$, so that 
				\[m\ge m_0\ge -B'(2q)\beta^3(q-\epsilon_q)^3t.\]
				If we denote the rightmost quantity as $m_q$, which is a continuous positive function on $(0,q)$, we have that $\zeta_q([0,m_q))=0$.
				
				With this established, we will use the same method we used to show the existence of a unique maximizer of $q\mapsto \inf_{\zeta\in P([0,q))}\cal{P}^L_{\beta}(q,\zeta)$ in Proposition II.4.6. For this, let us fix a measure $\zeta\in P([0,q))$, and define for any $r\in (-\inf(\supp(\zeta)),\infty)$, the measure $\zeta^r\in P([0,q+r))$ by letting $\zeta^r([0,s+r])=\zeta([0,s])$ for $s\in [0,q]$ and if $r\ge 0$ also setting $\zeta^r([0,r))=0$. This clearly gives a well-defined measure, as we have by assumption that $\zeta([0,r))=0$. Intuitively, $\zeta^r$ is a translation of the measure $\zeta$ by $r$. Now, in the proof of Proposition II.4.6, we show that the function $r\mapsto \cal{P}^L_{\beta}(q,\zeta^r)$ is linear and given by (see eqn. II.4.42)
				\[\cal{P}^L_{\beta}(q,\zeta^r)=\cal{P}^L_{\beta}(q,\zeta)+r\left(\beta K_{L,t}(\beta \delta_L(0))-\beta \mu\right),\]
				where $\delta_L$ is the function associated to the pair $(q_L,\zeta_L)$ from Theorem \ref{theorem:paper 2:main free energy result} by (\ref{eqn:def:delta-P}).
				
				Thus, by Lemma \ref{lem:continuum free energy: limit of Parisi in L} we see that
				\[\cal{P}_{\beta}(q,\zeta^r)=\cal{P}_{\beta}(q,\zeta)+r\left( \frac{1}{\beta^2 \delta(0)^2 t}-\mu\right).\]
				As we have just shown that $\zeta_q([0,m_q))=0$ for some $m_q>0$, this shows that for a given $q$ to be a maximizer of $\mathscr{P}(q)$, we must have that
				\[\frac{1}{\beta^2 \delta(0)^2 t}- \mu=0,\]
				or equivalently,
				\[\beta \delta(0)=\frac{1}{\sqrt{\mu t}},\]
				which is equation (\ref{eqn:intro:continuum:minimization eqn:Euclidean: measure}).
				
				Conversely, assume that (\ref{eqn:intro:continuum:minimization eqn:Euclidean: measure}) is satisfied for some $q$. Then if we let $\zeta_q$ be as in Lemma \ref{lem: continuum unique minimizer at fixed q}, for any $s\in [q-m_q/2,q+m_q/2]$ we have that
				\[\mathscr{P}(q)=\cal{P}_{\beta}(q,\zeta_q)=\cal{P}_{\beta}(s,\zeta_q^{s-q})\ge \inf_{\zeta\in P([0,s))}\cal{P}_{\beta}(s,\zeta)=\mathscr{P}(s). \label{eqn:ignore-23847283}\]
				Thus $q$ is a local maximizer of $\mathscr{P}(q)$, and so by concavity, it is a global maximizer. Thus, we have shown that maximizers of $\mathscr{P}(q)$ are precisely the solutions to (\ref{eqn:ignore-2347823647}) with $\zeta=\zeta_q$.
				
				Finally, we show there is a unique solution to (\ref{eqn:intro:continuum:minimization eqn:Euclidean: measure}). We showed in the proof of Theorem \ref{theorem:intro:1-d continuum T>0:free energy T>0} that there is $C,c>0$ such that 
				\[\sup_{q\in (0,\infty)}\mathscr{P}(q)=\sup_{q\in [c,C]}\mathscr{P}(q).\]
				As concavity implies that $\mathscr{P}(q)$ is continuous, we see there is at least one maximizer. If there is more than one maximizer, concavity implies that there is some $q\in (0,\infty)$ and $\epsilon>0$ such that all the points $[q,q+\epsilon]$ are all maximizers of $\mathscr{P}$. As all maximizers satisfy the equation (\ref{eqn:intro:continuum:minimization eqn:Euclidean: measure}), we see, as in (\ref{eqn:ignore-23847283}), that for $s\in [q,q+\epsilon]$
				\[\mathscr{P}(q)=\cal{P}_{\beta}(q,\zeta_q)=\cal{P}_{\beta}(s,\zeta_q^{q-s})\ge \cal{P}_{\beta}(s,\zeta_s)=\mathscr{P}(s).\]
				However, as $s$ is also a maximizer, these are all equalities. Moreover, as the infimizer in $\zeta$ is unique, this implies that for each $s \in [q,q+\epsilon]$, we have that $\zeta_{s}=\zeta_{q}^{s-q}$. 
				
				We will show that this contradicts the minimization condition (\ref{eqn:ignore-023478628736}). Indeed, let us fix some $\zeta\in P([0,q))$ and consider the functions $F_q$ and $F_{q+\epsilon}$ defined as Lemma \ref{lem: continuum unique minimizer at fixed q}, with respect to the pairs $(q, \zeta_q)$ and $(q+\epsilon, \zeta^{\epsilon}_q)$, respectively. Do the same for the functions $\delta$ and $\delta_{\epsilon}$. It is easy to see that $\delta_{\epsilon}(u+\epsilon)=\delta(u)$ for $u\in [0,q]$, while $\delta_{\epsilon}(u)=\delta(0)$ for $u\in [0,\epsilon]$. From this, one sees that
				\[F_{q+\epsilon}(s+\epsilon)=-2B(2(q-s))-\int_0^{s+\epsilon}\frac{2du}{\beta^3\delta_\epsilon(u)^3t}=F_q(s)-\frac{2\epsilon}{\beta^3\delta(0)^3t}.\]
				In particular, we see that 
				\[\{s\in [0,q):F_{q}(s)=0\}\cap \{s\in [\epsilon,q+\epsilon):F_{q+\epsilon}(s+\epsilon)=0\}=\varnothing.\]
				Proceeding as in the beginning of the proof of Lemma \ref{lem:preliminaries:q is not close to 1 general}, we see that (\ref{eqn:ignore-023478628736}) implies that 
				\[\zeta_{q}\left(\{s\in [0,q):F_q(s)=0\}\right)=1,\]
				and similarly, using the minimizer condition for $\zeta_{q+\epsilon}=\zeta_q^\epsilon$ we see that
				\[\zeta_{q}^\epsilon \left(\{s\in [0,q+\epsilon):F_{q+\epsilon}(s)=0\}\right)=1.\]
				As $\zeta_q([0,\epsilon))=0$, we see that
				\[\zeta_{q}^{\epsilon}\left(\{s\in [\epsilon,q+\epsilon):F_{q+\epsilon}(s)=0\}\right)=1.\]
				However, we also note that 
				\[\zeta_{q}^{\epsilon}\left(\{s\in [\epsilon,q+\epsilon):F_{q+\epsilon}(s)=0\}\right)=\zeta_{q}\left(\{s\in [0,q):F_{q+\epsilon}(s+\epsilon)=0\}\right).\]
				Thus, we have shown that $\zeta_q$ has full measure on two disjoint sets, which is a contradiction.
			\end{proof}
			
			Lemmas \ref{lem: continuum unique minimizer at fixed q} and \ref{lem:uniqueness of maximizer of P} will yield the following important analytic result for $\cal{P}_\beta$.
			
			\begin{prop}
				\label{prop: functional result for P_beta}
				The functional $\cal{P}_\beta$ has a unique critical point, $(q_e,\zeta_e)$ which is the unique solution to eqns. (\ref{eqn:intro:continuum:minimization eqn:Euclidean: Larkin}) and (\ref{eqn:intro:continuum:minimization eqn:Euclidean: measure}) from Theorem \ref{theorem:intro:critical point equations}. Moreover, 
				\[\sup_{q\in (0,\infty)}\left(\inf_{\zeta\in P([0,q))} \cal{P}_{\beta}(q,\zeta)\right)=\cal{P}(q_e,\zeta_e).\label{eqn:critical pair in prop has computation}\]
			\end{prop}
			\begin{proof}
				Lemmas \ref{lem: continuum unique minimizer at fixed q} and \ref{lem:uniqueness of maximizer of P} imply that there is a unique pair $(q_e,\zeta_e)$, which satisfies eqns. (\ref{eqn:intro:continuum:minimization eqn:Euclidean: Larkin}) and (\ref{eqn:intro:continuum:minimization eqn:Euclidean: measure}). Moreover, they show that $(q_e,\zeta_e)$ is such that $q_e$ is the unique maximizer of $q\mapsto \inf_{\zeta\in P([0,q))}\cal{P}_{\beta}(q,\zeta)$ and $\zeta_e$ is the unique minimizer of $\cal{P}_{\beta}(q_e,\zeta)$. This shows that the pair $(q_e,\zeta_e)$ is a critical point of $\cal{P}_\beta$, and also that (\ref{eqn:critical pair in prop has computation}) holds. All that is left is to show that this is the unique critical point. However, if $(q',\zeta')$ is any critical point, Lemma \ref{lem:uniqueness of maximizer of P} implies that $q'=q_e$ (as any critical point of a concave function is a maximum), and then Lemma \ref{lem: continuum unique minimizer at fixed q} implies that $\zeta'=\zeta_e$.
			\end{proof}
			
			Now that we have a characterization of the unique critical point of $\cal{P}_\beta$, the next step is to show the following convergence.
			
			\begin{prop}
				\label{prop:intro:1-d continuum T>0:identification of limits:weak}
				For each $L\in \N$, let $(q_{L},\zeta_{L})$ denote the pair from Theorem \ref{theorem:paper 2:main free energy result}, and let $(q_e,\zeta_e)$ denote the unique critical point of $\cal{P}_{\beta}$ from Proposition \ref{prop: functional result for P_beta}. Then we have that
				\[\lim_{L\to \infty}q_L=q_e,\;\;\;\;\lim_{L\to \infty}\zeta_{L}=\zeta_e, \]
				where the convergence in the second limit is in the weak sense when all measures are extended to probability measures on $[0,\infty)$.
			\end{prop}
			\begin{proof}
				Let $\mathscr{P}^L$ and $\mathscr{P}$ be as in the proof of Theorem \ref{theorem:intro:1-d continuum T>0:free energy T>0}, so that by definition $q_L$ is the unique maximizer of $\mathscr{P}^L(q)$ and $q_e$ is the unique maximizer of $\mathscr{P}(q)$. By Theorem \ref{theorem:intro:1-d continuum T>0:free energy T>0}, we have that
				\[\lim_{L\to \infty}\mathscr{P}^L(q_L)=\mathscr{P}(q_e).\label{eqn:ignore-324893894}\]
				Moreover, Lemma \ref{lem: limit of infimum for fixed q} implies that for fixed $q\in (0,\infty)$, we have that $\lim_{L\to \infty}\mathscr{P}^L(q)=\mathscr{P}(q)$, and in the proof of Theorem \ref{theorem:intro:1-d continuum T>0:free energy T>0} we showed that this convergence is in fact uniform on compact subsets of $(0,\infty)$. As $\{q_L\}_{L\in \N}$ lies in a compact subset of $(0,\infty)$ by Lemma \ref{lem:continuum free energy: upper and lower bounds on q}, this convergence and (\ref{eqn:ignore-324893894}) suffice to establish that $\lim_{L\to \infty}q_L=q_e$ by using a subsequence argument. To show convergence of $\zeta_L$ to $\zetac$, we note that the proof of Lemma \ref{lem: limit of infimum for fixed q} immediately generalizes to show that if $q_L\to q$, then $\zeta_{q_L}^L\to \zeta_q$. As $\zeta_L=\zeta^L_{q_L}$, this completes the proof.
			\end{proof} 
			
			Finally, we will obtain Theorems \ref{theorem:intro:main:Euclidean parameters identification: radius} and \ref{theorem:intro:main:Euclidean parameters identification: overlap} from analogous results for the finite-$L$ pair $(q_L,\zeta_L)$.
			
			\begin{proof}[Proof of Theorems \ref{theorem:intro:main:Euclidean parameters identification: radius} and \ref{theorem:intro:main:Euclidean parameters identification: overlap}]
				We will show that both of these theorems are satisfied with the choice $(\qc,\zetac)=(q_e,\zeta_e)$.
				By Theorem II.1.6, we have for any $x\in \Lambda_L$ that
				\[\lim_{N\to \infty}\E\left\<\left|\|\bm{u}(x)\|^2_N-q_L\right|\right\>=0,\]
				and moreover, for any bounded continuous function $f:\R \to \R$ that
				\[\lim_{N\to \infty}\E\<f((\bm{u}(x),\bm{u}'(x))_N)\>=\int_{0}^{q_L} f\left(r\right)\zeta_L(dr).\]
				By taking the limit $L\to \infty$, our claims follow these equalities and Proposition \ref{prop:intro:1-d continuum T>0:identification of limits:weak}.
			\end{proof}
			
			Note that this proof has established both that $(\qc,\zetac)$ exists in the sense of Theorems \ref{theorem:intro:main:Euclidean parameters identification: radius} and \ref{theorem:intro:main:Euclidean parameters identification: overlap}, and also that $(\qc,\zetac)=(q_e,\zeta_e)$. This will also allow us to show our remaining results by simply translating the above results for $(q_e,\zeta_e).$
			
			\begin{proof}[Proof of Theorems \ref{theorem:intro:critical point equations} and \ref{theorem:intro:1-d continuum T>0:free energy T>0 variational form: second version}]
				These follow from Propositions \ref{prop: functional result for P_beta} and \ref{prop:intro:1-d continuum T>0:identification of limits:weak} and the equality $(\qc,\zetac)=(q_e,\zeta_e)$.
			\end{proof}
			
			Finally, we record the following corollary of Proposition \ref{prop:intro:1-d continuum T>0:identification of limits:weak} which we will need below.
			
			\begin{corr}
				\label{corr:intro:1-d continuum T>0:identification of limits}
				For each $L\in \N$, let $(q_{L},\zeta_{L})$ denote the pair from Theorem \ref{theorem:paper 2:main free energy result}, and let $(\qc,\zetac)$ denote the Parisi pair from Definition \ref{def:Parisi}. Then we have that
				\[\lim_{L\to \infty}q_L=\qc,\;\;\;\;\lim_{L\to \infty}\zeta_{L}=\zetac, \]
				where the convergence in the second limit is in the weak sense when all measures are extended to probability measures on $[0,\infty)$.
			\end{corr}

			\section{Results on The Wandering Exponent \label{section: statics of gibbs measure}}
			
			We will now show our results involving the mean-squared displacement of the Gibbs measure (Theorem \ref{theorem:formula for the displacement at massive mu}), as well as our results on the asymptotic wandering exponents (Corollary \ref{corr:wandering: rs} and Theorems \ref{theorem: wandering: powerlaw} and \ref{theorem: wandering: exponential}). To do this, we will first give a result which allows one to calculate certain statistics of the Gibbs measures which are quadratic in $\b{u}$, namely Proposition \ref{proposition:parisi perturbation}. This proposition is for finite $L$, and could actually easily be generalized to the more general setting of our previous works \cite{Paper1, Paper2}, where the elastic polymer is replaced with a more general (discrete) $d$-dimensional elastic manifold. This result will be specialized to give the mean-squared displacement of the Gibbs measure for finite $L$, which is given in Corollary \ref{corr:parisi greens function}.
			
			To then obtain Theorem \ref{theorem:formula for the displacement at massive mu}, we study Corollary \ref{corr:parisi greens function} in the limit $L\to \infty$. A difficulty in doing this is that the quantities present Corollary \ref{corr:parisi greens function} are functions whose asymptotics are not covered by Proposition \ref{proposition:continuum free energy proof:continuum functions-1}. Instead, the asymptotics for these functions is given in Proposition \ref{prop: second approximation: the one for G}, whose proof is deferred to Appendix \ref{section: approximation proofs}.
			
			After we prove Theorem \ref{theorem:formula for the displacement at massive mu}, we will focus on the proof of Corollary \ref{corr:wandering: rs} and Theorems \ref{theorem: wandering: powerlaw} and \ref{theorem: wandering: exponential}. All of these results involve studying Theorem \ref{theorem:formula for the displacement at massive mu} carefully as one takes $\mu \to 0$. The proof of Corollary \ref{corr:wandering: rs} is quite easy, as Corollary \ref{corr:parisi greens function} simplifies quite nicely in this case. The other theorems are less simple. They rely heavily on the characterizations in Corollary \ref{corr:FRSB} and Theorem \ref{theorem:intro:1-d continuum T>0:1RSB} to reduce the quantities in Theorem \ref{theorem:formula for the displacement at massive mu} to something whose limits are more explicit and manageable.
			
			We now begin by giving our main result for quantities at fixed $L$. Recall that a matrix $A$ is circulant if the value of $A_{x,y}$ only depends on $x-y$. Circulant matrices serve as a discrete analogue of translation-invariant operators. This analogy is further strengthened by the classical fact that a matrix is circulant if and only if it may be diagonalized by the discrete Fourier transform. Our next result allows us to calculate the expectation with respect to the Gibbs measure of the quadratic form associated with a symmetric circulant matrix on $\Lambda_L$.
			
			\begin{prop}
				\label{proposition:parisi perturbation}
				Fix $L\in \N$ and some $A\in \R^{\Lambda_L \times \Lambda_L}$ which is a symmetric circulant matrix. Define functions
				\[R_{i,A}(u)=L^{1/2}\tr((u I-t\Delta_L)^{-i}A).\]
				Let $(q_L,\zeta_L)$ be the Parisi pair from Theorem \ref{theorem:paper 2:main free energy result}, let $\delta_L$ be the function associated to $(q_L,\zeta_L)$ by (\ref{eqn:def:delta-P}), and finally let $q_*$ be any point such that $\zeta_L([0,q_{L,*}])=1$. Then we have that 
				\begin{align}
					\lim_{N\to \infty}\frac{1}{L}\sum_{x,y\in \Lambda_L}\E\< &A_{x,y}(\b{u}(x),\b{u}(y))_N\>=\\
					&\beta^{-1}R_{1,A}(K_{L,t}(\beta(q_L-q_{L,*})))
					-\int_{0}^{q_{L,*}} R_{2,A}(K_{L,t}(\beta\delta_L(u)))K'_{L,t}(\beta\delta_L(u))du. \label{eqn:circulant matrix expectation}
				\end{align}
			\end{prop}
			
			We defer the proof of Proposition \ref{proposition:parisi perturbation} to the end of this section. To apply this proposition to the study of the mean-squared displacement, we need a function. For $x\in \Lambda_L$, let
			\[G_{x,t,L}(\mu)=L^{1/2}[(\mu I-\Delta_L)^{-1}]_{[x]_L,0}-L^{1/2}[(\mu I-t\Delta_L)^{-1}]_{0,0}.\label{eqn:def:GL}\]
			This is the discrete Green's function for the heat equation regularized on the diagonal. In particular is the discrete analogue of the continuum function $\G_{x,t}(\mu)$ defined in the introduction (see (\ref{eqn:def:regularized continuum green})). 
			
			Next we introduce a finite-$L$ version of the functional $H_{\beta,t,\mu}$ (see Definition \ref{def:H}). Namely, with notation as in Proposition \ref{proposition:parisi perturbation}, we define the function
			\[H_{L,\beta,t,\mu}(x)=-\frac{2}{\beta}G_{x,t,L}(K_{L,t}(\beta(q_L-q_{L,*}))-2\int_{0}^{q_{L,*}} G_{x,t,L}'(K_{L,t}(\beta\delta_L(u)))K_{L,t}'(\beta\delta_L(u))du.\]
			
			With this function defined, we can now give our main application of Proposition \ref{proposition:parisi perturbation}.
			\begin{corr}
				\label{corr:parisi greens function}
				For any $x,y\in \Lambda_L$, we have that 
				\begin{align}\lim_{N\to \infty}\E&\< \|\b{u}(x)-\b{u}(y)\|^2_N\>=H_{L,\beta,t,\mu}(x-y).\label{parisi greens function}
				\end{align}
			\end{corr}
			\begin{proof}
				Let us form the matrix
				\[[E^{x,y}]_{x',y'}:=\sum_{w\in \Lambda_L}\left(2\delta_{x',w}\delta_{y',w}-\delta_{x',x+w}\delta_{y',y+w}-\delta_{x',y+w}\delta_{y',x+w}\right),\]
				where $\delta$ denotes the Kronecker $\delta$-function. This matrix is clearly symmetric and circulant. Using translation invariance of the model on $\b{u}$ repeatably, we see that
				\begin{align}
					&\frac{1}{L}\sum_{x',y'\in \Lambda_L}\E\< E^{x,y}_{x',y'}(\b{u}(x'),\b{u}(y'))_N\>\\
					&=\frac{1}{L}\sum_{w\in \Lambda_L}\big(2\E\<(\b{u}(w),\b{u}(w))_N\>-\E\<(\b{u}(x+w),\b{u}(y+w))_N\>-\E\<(\b{u}(y+w),\b{u}(x+w))_N\>\big)\\
					&=2\E\<(\b{u}(x),\b{u}(x))_N\>-2\E\<(\b{u}(x),\b{u}(y))_N\>=\E\<\|\b{u}(x)-\b{u}(y)\|^2_N\>
				\end{align}
				Similarly, we compute that
				\[R_{1,E^{x,y}}(\mu)=2L^{1/2}[\mu I+t\Delta_L]^{-1}_{x,x}-2L^{1/2}[\mu I+t\Delta_L]^{-1}_{x,y}=-2G_{x-y,L,t}(\mu),\]
				and as $R_{1,E^{x,y}}'(\mu)=-R_{2,E^{x,y}}(\mu)$ we see that $R_{2,E^{x,y}}(\mu)=2G_{x-y,t,L}'(\mu)$. Thus, applying Proposition \ref{proposition:parisi perturbation} in the case of $A=E^{x,y}$ gives the desired result.
			\end{proof}

			When a model is RS, the expression in Corollary \ref{corr:parisi greens function} has a significantly simpler description. The formulas for the RS phase in Theorem II.1.7 give that \[\beta \delta_L(0)=\beta \delta_L(q_{L,*})=\beta (q_L-q_{L,*})=R_{1; L,t}(\mu).\]
			Using these equations and the fact that $-K_{L,t}'(R_{1;L,t}(\mu))R_{2;L,t}(\mu)=1$, we obtain the following corollary.
			\begin{corr}
				\label{corr:parisi greens function: RS}
				If the model is RS, then for any $x\in \Lambda_L$,
				\[H_{L,\beta,t,\mu}(x)=-\frac{2}{\beta}G_{x,t,L}(\mu)-4B'\left(\frac{2}{\beta}R_{1;L,t}(\mu)\right)G_{x,t,L}'(\mu).\]
			\end{corr}
			
			These results give a concrete formula for the mean-squared displacement when $L$ is fixed, and so we now proceed to study what happens when in the continuum limit. For this, we first need the following result, which establishes that $G_{x, L,t}$ converges to the function $\G_{x,t}$ in the continuum limit.
			
			\begin{prop}
				\label{prop: second approximation: the one for G}
				For any $\mu,t>0$ and $x\ge 0$, we have that
				\[\lim_{L\to \infty}G_{x,t,L}(\mu)=\G_{x,t}(\mu),\;\;\;\;\;
				\lim_{L\to \infty}G_{x,t,L}'(\mu)=\G_{x,t}'(\mu).\]
				Moreover, if one fixes $x$ and $t$ and chooses some $\epsilon>0$, these limits converge uniformly for all choices of $\mu\in [\epsilon,\epsilon^{-1}]$.
			\end{prop}
			
			The proof of this result will be deferred to Appendix \ref{section: approximation proofs}. Equipped with this result, we are ready to establish Theorem \ref{theorem:formula for the displacement at massive mu}, which is the continuum analogue of Corollary \ref{corr:parisi greens function}.
			
			\begin{proof}[Proof of Theorem \ref{theorem:formula for the displacement at massive mu}]
				In sight of Corollary \ref{corr:parisi greens function}, to prove Theorem \ref{theorem:formula for the displacement at massive mu}, we need to show that for any $x\in \R$, we have that
				\[\lim_{L\to \infty}H_{L,\beta,t,\mu}(x)=H_{\beta,t,\mu}(x).\]
				To write this out more explicitly, let $(q_L,\zeta_L,q_{L,*},\delta_L)$ be as in Proposition \ref{proposition:parisi perturbation}. Moreover, let $(\qc,\zetac)$ denote the Parisi pair for the continuum model (see Definition \ref{def:Parisi}), and denote by $\delta$ the functions associated to  $(\qc,\zetac)$ by (\ref{eqn:def:delta-P}). Finally, let $q_{*}<\qc$ be any value such that $\zetac([0,q_*])=1$. We only need to show that
				\[
				\begin{split}
					\lim_{L\to \infty}\left(\beta^{-1}G_{x,t,L}(K_{L,t}(\beta(q_L-q_{L,*})))
					+\int_{0}^{q_{L,*}} G_{x,t,L}'(K_{L,t}(\beta\delta_L(u)))K'_{L,t}(\beta\delta_L(u))du\right)\\
					=\beta^{-1}\G_{x,t} \left(\frac{1}{(\beta(\qc-q_*))^2t}\right)
					+\int_{0}^{q_*} \G_{x,t}' \left(\frac{1}{(\beta\delta(u))^2t}\right)\left(\frac{-2}{(\beta\delta(u))^3t}\right)du.
				\end{split} \label{eqn:ignore-237568712}
				\]
				For this, note that by Corollary \ref{corr:intro:1-d continuum T>0:identification of limits} we have that $\lim_{L\to \infty}q_L=\qc$ and $\lim_{L\to \infty}\zeta_L=\zetac$, where the latter convergence is in the weak sense as measures on $[0,\infty)$. In particular, we see that $\lim_{L\to \infty}\delta_{L}(u)=\delta_{L}(u)$. Moreover, by Lemma \ref{lem:preliminaries:q is not close to 1} and the asymptotics of $R_{1;L,t}$ and $U_L$ in Propositions \ref{proposition:continuum free energy proof:continuum functions-1} and \ref{proposition:continuum free energy proof:continuum functions-3}, we may choose $q_{*,L}$ and $q_*$ so that $\lim_{L\to \infty}q_{*,L}=q_*$ and such that $\inf_{L}(q_L-q_{*,L})>0$. As this value is both bounded above and bounded away from zero, we may use Propositions \ref{proposition:continuum free energy proof:continuum functions-1} and \ref{prop: second approximation: the one for G} to conclude that
				\[\lim_{L\to \infty}\beta^{-1}G_{x,t,L}(K_{L,t}(\beta(q_L-q_{L,*})))=\beta^{-1}\G_{x,t} \left(\frac{1}{(\beta(\qc-q_*))^2t}\right).\label{eqn:ignore-3239862}\]
				Similarly, we note that for $u\in [0,q_{L,*}]$,
				\[q_{L}-q_{L,*}=\delta_{L}(q_*)=\delta_{L}(u)\le \delta_L(0)\le q_L,\]
				so that $\delta_{L}(u)$ is also bounded above and bounded away from zero. In particular, using Propositions \ref{proposition:continuum free energy proof:continuum functions-1}, \ref{proposition:continuum free energy proof:continuum functions-3} and \ref{prop: second approximation: the one for G} again, we see that for $u\in [0,q_*]$
				\[\lim_{L\to \infty} G_{x,t,L}'(K_{L,t}(\beta\delta_L(u)))K'_{L,t}(\beta\delta_L(u))\\
				=\G_{x,t}' \left(\frac{1}{(\beta\delta(u))^2t}\right)\left(\frac{-2}{(\beta\delta(u))^3t}\right).\]
				So by applying the dominated convergence theorem, and recalling (\ref{eqn:ignore-3239862}), we conclude (\ref{eqn:ignore-237568712}).
			\end{proof}
			
			Now that the proof of Theorem \ref{theorem:formula for the displacement at massive mu} is complete, we proceed to prove our remaining results on asymptotic wandering exponents. These results involve understanding the quantities in Theorem \ref{theorem:formula for the displacement at massive mu} in the limit $\mu \to 0$. For this, we will need the following lemma concerning $\G_{x,t}(\mu)$, whose proof we omit.
			
			\begin{lem}
				\label{lem: technical facts about G}
				Fix $t>0$ and $x\in \R$. Then $\G_{x,t}(\mu)$ is a negative, increasing, strictly concave function of $\mu\in (0,\infty)$. This function and its derivative have 
				\[\lim_{\mu \to \infty}\G_{x,t}(\mu)=\lim_{\mu \to \infty}\G_{x,t}'(\mu)=0,\]
				\[\lim_{\mu\to 0}\G_{x,t}(\mu)=-\frac{|x|}{2}, \;\;\;\;\;\lim_{\mu\to 0}\left(\G_{x,t}'(\mu)-\frac{x^2}{8\sqrt{\mu t^3}}\right)=-\frac{|x|^3}{12t^2}.\]
				In particular, for $t,x\in (0,\infty)$, the function $\G_{x,t}(\mu)$ is bounded on $[0,\infty)$, while $\G_{x,t}'(\mu)=O(\mu^{-1/2})$ as $\mu \to 0$.
			\end{lem}
			
			We note that Corollary \ref{corr:formula for the displacement at massive mu RS Case} immediately follows from Lemma \ref{lem: technical facts about G} and Corollary \ref{corr:parisi greens function: RS}. We will now give the proof of Corollary \ref{corr:wandering: rs}.
			
			\begin{proof}[Proof of Corollary \ref{corr:wandering: rs}]
				As in Corollary \ref{corr:parisi greens function: RS}, we may use Theorem \ref{theorem:intro:1-d continuum T>0:RS} to reduce the statement of Theorem \ref{theorem:formula for the displacement at massive mu} to
				\[\lim_{L\to \infty}H_{L,\beta,t,\mu}(x)=-\frac{2}{\beta}\G_{x,t,L}(\mu)-4B'\left(\frac{2}{\beta\sqrt{\mu t}}\right)\G_{x,t,L}'(\mu).\]
				By Lemma \ref{lem: technical facts about G}, 
				\[\lim_{\mu\to 0}\left(-\frac{2}{\beta}\G_{x,t,L}(\mu)\right)=\frac{|x|}{\beta},\]
				so to complete the proof, it suffices to show that
				\[\limsup_{\mu\to 0}4B'\left(\frac{2}{\beta\sqrt{\mu t}}\right)\G_{x,t,L}'(\mu)= 0.\]
				Using Lemma \ref{lem: technical facts about G} again, we see that this follows if 
				\[\limsup_{\mu\to 0}B'\left(\frac{1}{\sqrt{\mu}}\right)\frac{1}{\sqrt{\mu}}=0 \text{ or equivalently if } \limsup_{x\to \infty}B'\left(x\right)x=0.\label{eqn:ignore-27895612783}\]
				However, as we have assumed that the massless model is RS, by Theorem \ref{theorem:massless RS phase transition theorem}, we must have that
				\[\lim_{x\to \infty}B''(x)x^3<\infty.\]
				Letting $\sup_{x\to \infty}B''(x)x^3=c$, we see that
				\[B'(x)=\int_x^{\infty}B''(y)dy\le \int_x^{\infty}\frac{c}{y^3}dy=\frac{c}{2x^2}, \label{eqn:ignore-integral_b}\]
				so that $B'(x)x^2$ is bounded as well, which implies (\ref{eqn:ignore-27895612783}).
			\end{proof}
			
			Next, we will move onto the proof of Theorem \ref{theorem: wandering: powerlaw} in the case of $\gamma<1$. As before, we must consider the formulas in Theorem \ref{theorem:formula for the displacement at massive mu} as $\mu \to 0$. To do this, we may rely on the explicit formulas for the Parisi measure given in Corollary \ref{corr:FRSB}. This allows us to rewrite the expression in Theorem \ref{theorem:formula for the displacement at massive mu} explicitly for fixed $\mu$. This reduces the proof to a calculus problem, which we may solve explicitly.
			
			\begin{proof}[Proof of Theorem \ref{theorem: wandering: powerlaw} with $\gamma<1$]
				
				First, we simplify the formula in Theorem \ref{theorem:formula for the displacement at massive mu} by using the formulas in Corollary \ref{corr:FRSB}. In particular, we fix $(\beta,t)$ and assume that $\mu<\mu_{Lar}(\beta;t)$ for the remainder of the proof. We will let $(q_0,q_*,\qc,\zetac)$ be the quantities from Corollary \ref{corr:FRSB}, and we let $\delta$ be the function associated
				to $(\qc,\zetac)$ by (\ref{eqn:def:delta-P}).

				By squaring the third equation in (\ref{eqn:FRSB triple eqns explicit}),
				\[-\frac{2}{\beta}\G_{x,t}\left(\frac{1}{\beta^2(\qc-q_*)^2t}\right)=-\frac{2}{\beta}\G_{x,t}\left(\mu_{Lar}(\beta,t)\right).\]
				As this is independent of $\mu$, this effectively deals with the first term in Theorem \ref{theorem:formula for the displacement at massive mu}, and so we will turn our focus on the second term in Theorem \ref{theorem:formula for the displacement at massive mu}. This term is an integral, and we study it by splitting it into two integrals, one from $0$ to $q_0$, and the other from $q_0$ to $q_*$, which we study separately.
				
				For $u\in (0,q_0)$, $\beta\delta(u)=\beta\delta(0)=\frac{1}{\sqrt{\mu t}}$ where in the last equality we used (\ref{eqn:intro:continuum:minimization eqn:Euclidean: Larkin}). Using this and (\ref{eqn:FRSB triple eqns explicit}) we compute that
				\[\int_0^{q_0}\G_{x,t}'\left(\frac{1}{\beta^2 \delta(u)^2t}\right)\frac{4}{\beta^3 \delta(u)^3t}du=q_0\G_{x,t}'\left(\mu\right)4\sqrt{\mu^3 t}= \G_{x,t}'\left(\mu\right)c_1\mu^{\frac{3}{2}+\frac{3(\gamma+1)}{2(2+\gamma)}+3}(4\sqrt{t}).\]
				In particular, by Lemma \ref{lem: technical facts about G}, we see that 
				\[\lim_{\mu\to 0}\int_0^{q_0}\G_{x,t}'\left(\frac{1}{\beta^2 \delta(u)^2t}\right)\frac{4}{\beta^3 \delta(u)^3t}du=\lim_{\mu\to 0}\G_{x,t}'\left(\mu\right)c_1\mu^{\frac{3}{2}+\frac{3(\gamma+1)}{2(2+\gamma)}+3}(4\sqrt{t})=0.\]
				Next, we consider the second part of the integral. As a preliminary step, note the third equation of (\ref{eqn:FRSB triple eqns explicit}), we have that $\mu=\beta^{-2}(\qc-q_*)^{-2}t^{-1}$ solves (\ref{eqn:larkin-mass-FRSB-gamma}), or more explicitly
				\[2\gamma(\gamma+1)\left(1+2\qc-2q_*\right)^{-\gamma-2}=\frac{1}{\beta^3(\qc-q_*)^3t}.\label{eqn:ignore-37839}\]
				Now by using (\ref{eqn:FRSB zeta eqns explicit}) we have for $u\in (q_0,q_*)$ that
				\begin{align}\beta\delta(u)=&\frac{\gamma+2}{3}\left(\frac{4}{t\gamma(\gamma+1)}\right)^{1/3}\int_{u}^{q_*}(1+2\qc-2s)^{\frac{\gamma-1}{3}}ds+\beta(\qc-q_*)\\
					&=\left(\frac{4}{t\gamma(\gamma+1)}\right)^{1/3}\frac{1}{2}\left((1+2\qc-2u)^{\frac{\gamma+2}{3}}-(1+2\qc-2q_*)^{\frac{\gamma+2}{3}}\right)+\beta(\qc-q_*)\\
					&=\frac{1}{(2t\gamma(\gamma+1))^{1/3}}(1+2\qc-2u)^{\frac{\gamma+2}{3}},
				\end{align}
				where in the final equality we have used (\ref{eqn:ignore-37839}) to cancel the second and third terms. From this we see that for $u\in (2(\qc-q_0)+1,2(\qc-q_*)+1)$,
				\[\beta\delta(\qc+1/2-u/2)=\frac{u^{\frac{\gamma+2}{3}}}{(2t\gamma(\gamma+1))^{1/3}}.\]
				Thus, by a change of variables, we may simplify the second integral as 
				\begin{align}\int_{q_0}^{q_*}\G_{x,t}' \left(\frac{1}{\beta^2\delta(u)^2t}\right)\frac{4}{\beta^3\delta(u)^3t}du&=\int_{2(\qc-q_*)+1}^{2(\qc-q_0)+1}\G_{x,t}'\left(\frac{(2t\gamma(\gamma+1))^{2/3}}{ u^{2(\gamma+2)/3}t}\right)\frac{4(2t\gamma(\gamma+1))}{u^{\gamma+2}t}\frac{du}{2}\\
					&=\int_{\frac{2}{\beta \sqrt{\mu_{Lar}(\beta;t)t}}+1}^{c_0 \mu^{-\frac{3}{2(\gamma+2)}}}\G_{x,t}'\left(\frac{(2\gamma(\gamma+1))^{2/3}}{u^{2(\gamma+2)/3}t^{1/3}}\right)\frac{4\gamma(\gamma+1)}{ u^{\gamma+2}}du,
				\end{align}
				where in the second step we used (\ref{eqn:FRSB triple eqns explicit}). Note that the integrand is independent of $\mu$, which only appears in the integral bounds. By Lemma \ref{lem: technical facts about G}, we see that for large $u$, the integrand is positive and bounded above by $Cu^{-2(\gamma+2)/3}\le C u^{-4/3}$ for some $C:=C(x,t,\beta)>0$. In particular, this integral is convergent at $\infty$. From this we see that
				\begin{align}
					\lim_{\mu\to 0}\int_{\frac{2}{\beta \sqrt{\mu_{Lar}(\beta;t)t}}+1}^{c_0 \mu^{-\frac{3}{2(\gamma+2)}}}\G_{x,t}'\left(\frac{(2\gamma(\gamma+1))^{2/3}}{u^{2(\gamma+2)/3}t^{1/3}}\right)\frac{4\gamma(\gamma+1)}{ u^{\gamma+2}}du\\
					=\int_{\frac{2}{\beta \sqrt{\mu_{Lar}(\beta;t)t}}+1}^{\infty}\G_{x,t}'\left(\frac{(2\gamma(\gamma+1))^{2/3}}{u^{2(\gamma+2)/3}t^{1/3}}\right)\frac{4\gamma(\gamma+1)}{ u^{\gamma+2}}du.
				\end{align}
				This completes our simplification of all the terms in Theorem \ref{theorem:formula for the displacement at massive mu}. Collecting terms, we conclude that
				\[\begin{split}H_{\beta,\mu,t}(x)=
					&-\frac{2}{\beta}\G_{x,t}\left(\mu_{Lar}(\beta,t)\right)+\int_{\frac{2}{\beta \sqrt{\mu_{Lar}(\beta;t)t}}+1}^{\infty}\G_{x,t}'\left(\frac{(2\gamma(\gamma+1))^{2/3}}{u^{2(\gamma+2)/3}t^{1/3}}\right)\frac{4\gamma(\gamma+1)}{ u^{\gamma+2}}du.
				\end{split}\label{eqn:ignore-3892370}\]
				Now, finally, we take the limit $x\to \infty$. For this, we note that for fixed $t,y>0$, the function $\G_{x,t}(y)$ is exponentially decaying in $x$. Thus, we may ignore the first term on the right-hand side of (\ref{eqn:ignore-3892370}). To deal with the second term, we must make its $x$-dependence clear. Using the definition of $\G_{x,t}$ in (\ref{eqn:def:regularized continuum green}), one may routinely check that for $x>0$
				\[\G_{x,t}'(y)=\frac{-e^{-x\frac{\sqrt{y}}{\sqrt{t}}}+1}{4\sqrt{y^3 t}}-\frac{xe^{-x\frac{\sqrt{y}}{\sqrt{t}}}}{4\sqrt{y^2 t^2}}=\frac{f\left(\frac{x}{\sqrt{t}}\sqrt{y}\right)}{4\sqrt{t y^3 }},\]
				where $f(z)=1-e^{-z}-ze^{-z}$. So we may rewrite the integral term in (\ref{eqn:ignore-3892370}) as
				\begin{align}
					\int_{\frac{2}{\beta\sqrt{\mu_{Lar}(\beta;t)t}}+1}^{\infty}&f\left(\frac{x}{\sqrt{t}}\frac{( 2\gamma(\gamma+1))^{1/3} }{u^{(\gamma+2)/3}t^{1/6}}\right)\left(\frac{ u^{(\gamma+2)}t^{1/2}}{4\sqrt{t}2\gamma(\gamma+1)}\right)\frac{4\gamma(\gamma+1)}{ u^{\gamma+2}}du\\
					&=\int_{\frac{2}{\beta\sqrt{\mu_{Lar}(\beta;t)t}}+1}^{\infty}f\left(\frac{x}{ u^{(\gamma+2)/3}}\frac{( 2\gamma(\gamma+1))^{1/3} }{t^{2/3}}\right)du\\
					&=x^{\frac{3}{\gamma+2}}\int_{x^{-(\gamma+2)/3}(\frac{2}{\beta\sqrt{\mu_{Lar}(\beta;t)t}}+1)}^{\infty}f\left(\frac{( 2\gamma(\gamma+1))^{1/3} }{ u^{(\gamma+2)/3}t^{2/3}}\right)du.
				\end{align}
				As $y\to 0$ we have that $f(y)=O(y^2)$ and in addition $\lim_{y\to \infty}f(y)=1$. This shows that the integrand in the final integral is convergent on $(0,\infty)$. In particular
				\begin{align}
					\lim_{x\to \infty}\int_{x^{-(\gamma+2)/3}\frac{2}{\beta\sqrt{\mu_{Lar}(\beta;t)t}}+1}^{\infty}f\left(\frac{( 2\gamma(\gamma+1))^{1/3} }{ u^{(\gamma+2)/3}t^{2/3}}\right)du=\int_{0}^{\infty}f\left(\frac{( 2\gamma(\gamma+1))^{1/3} }{ u^{(\gamma+2)/3}t^{2/3}}\right)du\label{eqn:ignore-3248237}\\
					=\left(\frac{2\gamma(\gamma+1)}{t^2}\right)^{1/(\gamma+2)}\int_{0}^{\infty}f\left(u^{-(\gamma+2)/3}\right)du.
				\end{align}
				
				To find the final integral, note that for $\alpha>\frac{1}{2}$ if we make the substitution $v=u^{-\alpha}$ and then integrate by parts we see that 
				\[\int_0^\infty f(u^{-\alpha})du=\int_{0}^{\infty}f\left(v\right)\frac{v^{-1-1/\alpha}}{\alpha}dv=\int_{0}^{\infty}f'\left(v\right)v^{-1/\alpha}dv=\int_{0}^{\infty}v^{1-1/\alpha}e^{-v}dv=\Gamma\left(2-\frac{1}{\alpha}\right).\]
				Altogether, we conclude that
				\[\lim_{x\to \infty}\frac{H_{\beta,\mu,t}(x)}{x^{\frac{3}{\gamma+2}}}=\left(\frac{2\gamma(\gamma+1)}{t^2}\right)^{1/(\gamma+2)}\Gamma\left(2-\frac{3}{\gamma+2}\right),\]
				which suffices to complete the proof.
			\end{proof}
			
			To complete our results on the wandering exponents, what remains is the proof of Theorem \ref{theorem: wandering: powerlaw} in the case $\gamma \ge 1$ and the proof of Theorem \ref{theorem: wandering: exponential}. As these proofs are nearly identical, we will combine them into one.
			
			Again, we must consider the quantities in Theorem \ref{theorem:formula for the displacement at massive mu} as $\mu \to 0$. As the measure is either RS or 1RSB in this case, the formulas in Theorem \ref{theorem:formula for the displacement at massive mu} may be simplified a lot. On the other hand, as we do not have explicit formulas for the exact Parisi measure, we must rely instead on the implicit characterization in Theorem \ref{theorem:intro:1-d continuum T>0:1RSB}. In particular, we establish the limit as $\mu \to 0$ by establishing certain bounds on the measure using the characterization in Theorem \ref{theorem:intro:1-d continuum T>0:1RSB}.
			
			\begin{proof}[Proof of Theorem \ref{theorem: wandering: exponential} and Theorem \ref{theorem: wandering: powerlaw} with $\gamma\ge 1$]
				Let $(\qc,\zetac)$ be the Parisi pair and let $\delta$ be the function associated to $(\qc,\zetac)$ by (\ref{eqn:def:delta-P}). For the moment, let us assume that the model is 1RSB for all $\mu$ sufficiently small. To begin, we will simplify the expression in Theorem \ref{theorem:formula for the displacement at massive mu} using the fact that the measure is 1RSB with form as in Theorem \ref{theorem:intro:1-d continuum T>0:1RSB}. In particular, we let $(q_0,q_1,\qc,m)$ be the quantities from Theorem \ref{theorem:intro:1-d continuum T>0:1RSB}, so that $\zetac=m\delta_{q_0}+(1-m)\delta_{q_1}$.

				 Using (\ref{eqn: 1RSB explicit}) we find that for $u\in (0,q_0]$, $\beta\delta(u)=\beta\delta(0)=\frac{1}{\sqrt{\mu t}}$, so that
				\[\int_{0}^{q_0}\G_{x,t}'\left(\frac{1}{\beta^2\delta(u)^2t}\right)\frac{4}{\beta^3\delta(u)^3t}du=q_0\G_{x,t}'(\mu)4\sqrt{\mu^{3}t}.\]
				Moreover, for $u\in (q_0,q_*)$, we see that $\delta(u)=\qc-q_*+m(q_*-u)$, and as $\beta\delta(q_0)=\frac{1}{\sqrt{\mu t}}$, we have
				\[\int_{q_0}^{q_*}\G_{x,t}'\left(\frac{1}{\beta^2\delta(u)^2t}\right)\frac{4}{\beta^3\delta(u)^3t}du=\frac{2}{\beta m}\left(\G_{x,t}\left(\frac{1}{\beta^2(\qc-q_*)^2t}\right)-\G_{x,t}(\mu)\right).\]
				Collecting these computations, Theorem \ref{theorem:formula for the displacement at massive mu} shows that
				\[\begin{split}&H_{\beta,\mu,t}(x)=\\
					&-\frac{2}{\beta}\G_{x,t}\left(\frac{1}{\beta^2(\qc-q_*)^2t}\right)+q_0\G_{x,t}'(\mu)4\sqrt{\mu^{3}t}-\frac{2}{\beta m}\left(\G_{x,t}(\mu)-\G_{x,t}\left(\frac{1}{\beta^2(\qc-q_*)^2t}\right)\right).\label{eqn:1RSB fluctuation expression}
				\end{split}\]
				We must now take the limit as $\mu\to 0$ of the expression on the right-hand side of (\ref{eqn:1RSB fluctuation expression}) $\mu\to 0$. Our first step will be to show that the second term is negligible. Note that by combining (\ref{eqn: 1RSB explicit}) with the fact that $0\le m\le 1$, we have that 
				\[\qc-q_0\ge \frac{1}{\beta\sqrt{\mu t}}\ge \qc-q_*.\label{eqn: soft resolve inequalities}\]
				Now we note that by the first equation in (\ref{eqn: 1RSB F equations}), $F(q_0)=0$, can be explicitly written out as 
				\[F(q_0)=-2B'(2(\qc-q_0))-2q_0\sqrt{\mu^3t}=0.\]
				Rearranging this and noting that $-B'$ is positive and decreasing, we see that 
				\[q_0=\frac{-2B'(2(\qc-q_0))}{2\sqrt{\mu^3t}}\le \frac{-2B'(\frac{2}{\beta\sqrt{\mu t}})}{2\sqrt{\mu^3t}},\label{eqn:ignore-23486923}\]
				where in the last step we used (\ref{eqn: soft resolve inequalities}).
				In particular, we have that 
				\[0\le q_0\G_{x,t}'(\mu)4\sqrt{\mu^{3}t}\le  -4B'\left(\frac{2}{\beta\sqrt{\mu t}}\right)\G_{x,t}'(\mu).\]
				Using Lemma \ref{lem: technical facts about G}, it is easily checked for either choice $B$ that we have 
				\[\lim_{\mu\to 0}\left(4B'\left(\frac{2}{\beta\sqrt{\mu t}}\right)\G_{x,t}'(\mu)\right)=0,\]
				so that
				\[\lim_{\mu\to 0}\left(-q_0\G_{x,t}'(\mu)4\sqrt{\mu^{3}t}\right)=0.\label{eqn:ignore-83892}\]
				Thus, the second term in (\ref{eqn:1RSB fluctuation expression}) is negligible as $\mu \to 0$.
				
				To deal with the remaining terms, we will need to collect some bounds. First, we will bound $(\qc-q_*)$ above. For this, observe that by Rolle's theorem, (\ref{eqn: 1RSB F equations}) implies that there is some $r\in [q_0,q_*]$ such that $F'(r)=0$. We may rewrite this equation as 
				\[4B''(2(\qc-r))=\frac{2}{t(\frac{1}{\sqrt{\mu t}}+\beta m(q_0-r))^3}=\frac{2}{t(\beta(\qc-q_*)+\beta m(q_*-r))^3},\]
				where in the second equality we have used (\ref{eqn: 1RSB explicit}). As $m\le 1$, this implies that
				\[4B''(2(\qc-r))\ge \frac{2}{t(\beta(\qc-q_*)+\beta (q_*-r))^3}=\frac{2}{t\beta^3(	\qc-r)^3}, \label{eqn:ignore-38723}\]
				or equivalently that
				\[(	\qc-r)^3B''(2(\qc-r))\ge \frac{1}{2t\beta^3}.\]
				Now $x^3B''(2x)$ is positive and that $\lim_{x\to 0}x^3B''(2x)=0$. Moreover, if either $B(x)=(1+x)^{-\gamma}$ for $\gamma>1$ or $B(x)=e^{-x}$ for any $\lambda>0$ it can be routinely verified that $\lim_{x\to \infty}x^3B''(2x)=0$. Therefore, for either of these cases of $B$ there is some $C:=C(\beta,t)>0$, such that
				\[\qc-r\le C.\]
				Now as $r\le q_*$, this implies that
				\[\qc-q_*\le C.\label{eqn:ignore-2989243}\]
				However, in the remaining case of $B(x)=(1+x)^{-1}$, one has that $\qc-q_*$ is given by a $\mu$-independent quantity by Remark \ref{remark:gamma=1 remark}. Thus (\ref{eqn:ignore-2989243}) holds for all cases of $B$.
				
				Now we will use this bound to show that $m$ cannot go to zero. From (\ref{eqn: 1RSB F equations}) we have that $F(q_*)=0$, which using (\ref{eqn: 1RSB explicit}) we can rewrite as
				\[0=-2B'(2(\qc-q_*))-2q_0\sqrt{\mu^3 t}-\frac{1}{m t\beta^3(\qc-q_*)^2}+\frac{\mu}{\beta m}.\]
				As $-B'$ is positive and decreasing, we can use (\ref{eqn:ignore-2989243}) to see that this implies that
				\[-2B'(0)\ge-2B'(2(\qc-q_*)) -2q_0\sqrt{\mu^3 t}=\frac{1}{m t\beta^3(\qc-q_*)^2}-\frac{\mu}{\beta m}\ge \frac{1}{m}\left(\frac{1}{ t\beta^3C^2}-\frac{\mu}{\beta }\right),\]
				which shows that
				\[m\ge \frac{1}{-2B'(0)}\left(\frac{1}{ t\beta^3C^2}-\frac{\mu}{\beta }\right).\]
				When $\mu$ is sufficiently small that the right-hand side is positive, and so long as $\mu$ is sufficiently small, there is $c:=c(t,\beta)>0$ such that
				\[m\ge c. \label{eqn:ignore-328978234}\]
				
				We now return to the task of bounding the terms on the right-hand side of (\ref{eqn:1RSB fluctuation expression}). As a reminder, we have shown that the second term is negligible. As $\G_{x,t}(y)$ exponentially decays in $x$ for any fixed $y>0$, we see by (\ref{eqn:ignore-2989243}) that $\G_{x,t}\left(\frac{1}{\beta^2(\qc-q_*)^2t}\right)$ exponentially decays in $x$. Combined with (\ref{eqn:ignore-328978234}), we see $\frac{2}{\beta m}\G_{x,t}\left(\frac{1}{\beta^2(\qc-q_*)^2t}\right)$ decays exponentially in $x$ as well. In particular, there is $C>0$ such that for all $x$,
				\[\lim_{\mu \to 0}\frac{2}{\beta m}|\G_{x,t}\left(\frac{1}{\beta^2(\qc-q_*)^2t}\right)|\le C.\] 
				Thus the only term that is not negligible in (\ref{eqn:1RSB fluctuation expression}) is $\frac{2}{\beta m}\G_{x,t}(\mu)$. As $c\le m\le 1$, we see that
				\[\frac{2}{\beta }\G_{x,t}(\mu)\ge  \frac{2}{\beta m}\G_{x,t}(\mu)\ge \frac{2}{\beta c}\G_{x,t}(\mu).\]
				These observations combined with (\ref{eqn:ignore-83892}) and the formula $\lim_{\mu\to 0}\G_{x,t}(\mu)=-\frac{x}{2}$ show that for some, possibly smaller $c>0$, and sufficiently large $x$ we have that
				\[\limsup_{\mu\to 0}H_{\beta,\mu,t}(x)\le \frac{x}{\beta c} \text{ and that }\frac{c x}{\beta}\le \liminf_{\mu\to 0}H_{\beta,\mu,t}(x).\label{eqn:ignore-1250}\]
				This suffices to show the desired claim, in the case that the model is 1RSB for all sufficiently small $\mu$.
				
				If instead the model is RS for all sufficiently small $\mu$, then the result follows from Corollary \ref{corr:wandering: rs}. The last case is where the model switches between RS and 1RSB an infinite number of times as $\mu \to 0$. We do not believe that this case can occur, but we can easily deal with it to complete the proof. For this, note the above work shows that (\ref{eqn:ignore-1250}) still holds if the limit $\mu \to 0$ is restricted to a sequence of $\mu\to 0$ where the model is 1RSB (for some constant $c>0$ that is independent of $\mu$). However, we may use Corollary \ref{corr:parisi greens function: RS} and Lemma \ref{lem: technical facts about G} to conclude  (\ref{eqn:ignore-1250}) holds if the limit in $\mu$ is instead restricted to a sequence of $\mu \to 0$ where the model is RS. Combining these cases, we see that (\ref{eqn:ignore-1250}) holds without modification, which completes the proof in this case.
			\end{proof}
			
			Finally, we complete this section by providing our proof of Proposition \ref{proposition:parisi perturbation}. Before this, we will briefly explain its key points. The main idea is first to perturb the Hamiltonian by adding a factor of $\epsilon \sum_{x,y\in \Lambda_L}A_{x,y}(\b{u}(x),\b{u}(y))$. The methods in our companion papers \cite{Paper1, Paper2} suffice to give a formula for the limiting free energy of this model, which is similar to Theorem \ref{theorem:paper 2:main free energy result}, except that the Parisi functional is now $\epsilon$-dependent (see (\ref{eqn: parisi perturbed by A})). Now we wish to take the derivative of both sides of this free energy formula at $\epsilon=0$ to obtain something like (\ref{eqn:circulant matrix expectation}). However, this is not easy, as to take the derivative of our free energy formula, we must pass the derivative in $\epsilon$ through both a supremum and an infimum. However, if we only had to deal with an infimum, it would be simpler, as one may show that the free energy at each $N$ is convex in $\epsilon$, and the infimum of convex functions is easier to control. However, the supremum in this case is over $q\in (0,\infty )$, and the unique maximizing $\qc$ is roughly the squared radius where the Gibbs measure is concentrated (see Theorem \ref{theorem:intro:main:Euclidean parameters identification: radius}). So to remove this supremum, we will show that one may replace the Gibbs measure on the left-hand side of (\ref{eqn:circulant matrix expectation}) with a Gibbs measure of the same Hamiltonian, but restricted to a product of thin annuli of radius $\sqrt{N\qc}$. If we then compute the free energy of this restricted Hamiltonian, and take the thickness of the annuli to zero, this allows us to remove the supremum over $q$ in the free energy formula, which is sufficient for us to compute the derivative. This method is similar in form to the one we employed to prove Theorem II.1.10 in \cite{Paper2}.
			
			\begin{proof}[Proof of Proposition \ref{proposition:parisi perturbation}]
				We fix a choice of finite $L$ for the duration of this proof. We will denote by $(q_L,\zeta_L)$ the pair from Theorem \ref{theorem:paper 2:main free energy result}. Moreover, we let $q_{L,*}$ be a point such that $\zeta_L([0,q_{L,*}])=1$, and we let $\delta_L$ be the function associated to $(q_L,\zeta_L)$ by (\ref{eqn:def:delta-P}).
				
				For convenience, we first make a preliminary reduction. If we consider the case, $A=I$, then (\ref{eqn:circulant matrix expectation}) is reduced to
				\begin{align}
					\lim_{N\to \infty}\frac{1}{L}\sum_{x\in \Lambda_L}\E\< \|\b{u}(x)\|^2_N\>=&\beta^{-1} R_{1;L,t}(K_{L,t}(\beta(q_L-q_{L,*})))
				\\&-\int_{0}^{q_{L,*}}  R_{2;L,t}(K_{L,t}(\beta\delta_L(u)))K_{L,t}'(\beta\delta_L(u))du.
				\end{align}
				As $R_{1; L,t}(K_{L,t}(y))=y$ and $R_{2; L,t}(K_{L,t}(y))K_{L,t}'(y)=-1$ (which follows by differentiating the previous expression), this reduces further to 
				\[\lim_{N\to \infty}\frac{1}{L}\sum_{x\in \Lambda_L}\E\< \|\b{u}(x)\|^2_N\>=\beta^{-1} \beta(q_L-q_{L,*})
				-\int_{0}^{q_{L,*}}  (-1)du= q_L.\]
				By translation invariance, this follows from Theorem \ref{theorem:intro:critical point equations}, so we have established the proposition in the case of $A=I$. Now note that both sides of (\ref{eqn:circulant matrix expectation}) are linear in $A$. Thus, for a given $A$, it suffices to establish (\ref{eqn:circulant matrix expectation}) for $A+\nu I$, where we get to choose $\nu$. In particular, without loss of generality, we may now assume that $A$ is positive semi-definite with zero as an eigenvalue, which we now do. We note that the matrix $-\Delta_L$ satisfies these properties, and thus the matrix $-t\Delta_L+A$ does as well.
				
				Next, for small $\kappa>0$ and $q\in (0,\infty)$, let us define the annulus
				\[D_N^{\kappa}(q)=\{\b{u}\in (\R^N)^{\Lambda_L}:|\|\b{u}\|^2_N-q|\le \kappa\}.\]
				We further define
				\[Z_{N,\beta}^{\kappa,q}:=\int_{D_{N}^{\kappa}(q)}e^{-\beta\cal{H}_{N,L}(\b{u})}d\b{u}, \;\;\;\; \<f(\b{u})\>^{\kappa,q}=\frac{1}{Z_{N,\beta}^{\kappa,q}}\int_{D_{N}^{\kappa}(q)}f(\b{u})e^{-\beta \cal{H}_{N,L}(\b{u})}d\b{u}.\]
				Then we have the following
				
				\begin{lem}
					For any $\delta>0$ we have that
					\[\lim_{N\to \infty}\left|\sum_{x,y\in \Lambda_L}\E\< A_{x,y}(\b{u}(x),\b{u}(y))_N\>-\sum_{x,y\in \Lambda_L}\E\< A_{x,y}(\b{u}(x),\b{u}(y))_N\>^{\kappa,q_L}\right|=0.\]
				\end{lem}
				\begin{proof}
					Employing Proposition II.5.2, this follows easily from Lemma II.5.3 as in the proof of Lemma II.5.4.
				\end{proof}	
				In particular, this lemma shows that
				\[\lim_{N\to \infty}\sum_{x,y\in \Lambda_L}\E\< A_{x,y}(\b{u}(x),\b{u}(y))_N\>=\lim_{\kappa\to 0}\lim_{N\to \infty}\sum_{x,y\in \Lambda_L}\E\< A_{x,y}(\b{u}(x),\b{u}(y))_N\>^{\kappa,q_L}.\label{eqn:delta-to-zero}\]
				For our next step, we will perturb the Hamiltonian. In particular, for any $\epsilon\in \R$ we define  \[\cal{H}_{N,L}^{\epsilon}(\b{u})=\frac{1}{2}\sum_{x,y\in \Lambda_L}(\mu I-t\Delta+\epsilon A)_{xy}(\b{u}(x),\b{u}(y))L^{-1/2}+\sum_{x\in \Lambda_L}V_{N,x}(\b{u}(x))L^{-1/4},\label{eqn:def:main-model-D-epsilon}\]	
				where $V_{N,x}$ are a collection of i.i.d copies of $V_N$. When $\epsilon=0$, this coincides with our original Hamiltonian. Our next step will be to produce a formula for the asymptotic free energy of this Hamiltonian.
				
				For this, let us take some $\epsilon\in \R$ and define
				\[R_{1;L,t}^{\epsilon}(u)=L^{1/2}\tr((uI-t\Delta_L+\epsilon A)^{-1}).\]
				We denote the functional inverse of this function as $K_{L,t}^{\epsilon}$. Existence and basic properties of these functions are given in Appendix I.C of \cite{Paper1}, which studies a more general class of functions.
				
				Then, for any $q\in (0,\infty)$, and $\zeta\in P([0,q))$, we choose $q_*<q$ such that $\zeta([0,q_*])=1$. Let $\delta$ now denote the function associated to $(q,\zeta)$. We define the functional
				\begin{align}
					\cal{P}_{\beta}^{\epsilon,L}(q,\zeta)=\frac{1}{2}\bigg(&L^{1/2}\log \left(\frac{2\pi}{\beta}L\right)+\beta(q-q_*)K_{L,t}^\epsilon(\beta(q-q_*))\\
					&-\frac{1}{L^{1/2}}\log \det \left(K_{L,t}^{\epsilon}(\beta(q-q_*))I-t\Delta_L+\epsilon A\right)\\&+\int_{0}^{q_*}\beta K_{L,t}^{\epsilon}(\beta\delta(u))du-2\beta^2\int_0^q \zeta([0,u])B'(2(q-u)) du-\beta \mu  q\bigg). \label{eqn: parisi perturbed by A}
				\end{align}
				We note that when $\epsilon=0$, this recovers our original Parisi function for fixed $L$ (\ref{eqn:Parisi function for fixed L}) up to a constant term (only dependent on $(\mu,t, L)$). This constant term arises due to a difference in normalization (see the remarks in the proof of Theorem \ref{theorem:paper 2:main free energy result}). We change this normalization here to introduce as few $A$-dependent factors as possible.
				
				We also note that it is easily checked that this definition is independent of the choice of $q_*$. We then have the following result.
				\begin{prop}
					\label{proposition:parisi perturbation-in-proof}
					Fix $\epsilon\in (-\mu \|A\|_{op}^{-1},\infty)$, where $\|A\|_{op}$ denotes the operator norm of $A$. Then for any $q\in (0,\infty)$ we have that
					\[\lim_{\kappa\to 0}\limsup_{N\to \infty}\left|L^{-1/2}N^{-1}\E \log Z_{N,\beta}^{\kappa,q}(\epsilon)-\inf_{\zeta\in P([0,q))}\cal{P}_{\beta}^{\epsilon,L}(q,\zeta)\right|=0. \label{eqn:ignore-329403789}\]
					Moreover, the functional $\cal{P}_{\beta}^{\epsilon,L}(q,\zeta)$ is strictly convex with a unique minimizer.
				\end{prop}
				\begin{proof}
					This statement follows by replacing $\mu I-t \Delta_L$ with $(\mu I-t \Delta_L+\epsilon A)L^{1/2}$ in a number of results in our companion papers \cite{Paper1, Paper2}. Indeed, much of \cite{Paper1, Paper2} already consider a more general Hamiltonian where the term $\mu I-t \Delta_L$ is replaced with a general positive semi-definite matrix $D$, which is ``transitive" in the sense of Definition II.1.2. This fact is pointed out after the statement of Theorem I.1.6 in \cite{Paper1}, and before Definition II.1.2 in \cite{Paper2}. We note that the assumption that $A$ is circulant implies that $\mu I-t\Delta_L+\epsilon A$  is ``transitive" in this sense, and the restriction $\epsilon>-\mu \|A\|_{op}^{-1}$ ensures that this matrix is positive definite.
					
					We now explain how this generalization is done in more detail. The first claim (\ref{eqn:ignore-329403789}) is a direct generalization of Lemma II.5.7, and so we will explain how to modify the proof of this lemma. If instead of thin annuli, one considers the model on a product of spheres, i.e. 
					\[Z_{N,\beta}^{\mathrm{Sph},q}(\epsilon):=\int_{S_N(q)^{\Lambda_L}}e^{-\beta\cal{H}_{N,L}^{\epsilon}(\b{u})}\omega(d\b{u}),\]
					where $S_N(q)=\{x\in \R^N:\|x\|=\sqrt{Nq}\}$ and $\omega$ is the surface measure, then the results of our companion papers already apply without modification. In particular, when $q=1$, the restricted Hamiltonian is a case of the general Hamiltonian considered in eqn. II.2.1, with $D=(\mu I -t \Delta_L +\epsilon A)L^{1/2}$, $\b{\xi}=(B_1)_{x\in \Lambda_L}$, with $B_1(r)=L^{-1/2}B(2(1-r))$, and $\b{h}=0$. $B_1$ is ``generic" in the sense of eqn. II.1.46 by eqns. II.5.42 and II.5.46, and so one may combine Theorems I.1.6 and II.2.4 to find that
					\[\lim_{N\to \infty}L^{-1/2}N^{-1}\E \log Z_{N,\beta}^{\mathrm{Sph},1}(\epsilon)=\inf_{\zeta\in P([0,1))}\cal{P}_{\beta}^{\epsilon,L}(1,\zeta).\]
					For general, $q$, one may use a rescaling of $\b{u}\mapsto \frac{1}{\sqrt{q}} \b{u}$ to relate the partition function on $S_{N}(q)$ to one on $S_N(1)$. Details of this relation are similar to those around equation II.4.6 and Lemma I.2.3. With this, one may apply Theorems I.1.6 and II.2.4 again to obtain that
					\[\lim_{N\to \infty}L^{-1/2}N^{-1}\E \log Z_{N,\beta}^{\mathrm{Sph}}(q,\epsilon)=\inf_{\zeta\in P([0,q))}\cal{P}_{\beta}^{\epsilon,L}(q,\zeta).\]
					Thus to show (\ref{eqn:ignore-329403789}) we only need to show that
					\[\lim_{\kappa\to 0}\limsup_{N\to \infty}\left|L^{-1/2}N^{-1}\E \log Z_{N,\beta}^{\kappa,q}(\epsilon)-L^{-1/2}N^{-1}\E \log Z_{N,\beta}^{\mathrm{Sph}}(q,\epsilon)\right|=0.\]
					However, this follows by applying Proposition
					I.A.7, which compares the free energy of a Hamiltonian on a product of thin annuli with the free energy of the same Hamiltonian on a product of spheres, as in the proof of Lemma II.5.7. This completes our proof of (\ref{eqn:ignore-329403789}).
					
					The claim about $\cal{P}_{\beta}^{\epsilon, L}(q_L,\zeta)$ being strictly convex with a unique minimizer follows by generalizing the proof of Theorem II.1.14. This proof considers the case of $\beta=1$ (which suffices by rescaling the parameters), and if one replaces $\mu I-t\Delta_L$ with $(\mu I-t\Delta_L+\epsilon A)L^{1/2}$, instances of $R_1$ and $K$, with $R_{1; L,t}^{\epsilon}$ and $K_{L,t}^{\epsilon}$, takes $h=0$, and replaces integral bounds from $0$ to $1$ with bounds from $0$ to $q$, the proof generalizes immediately.
				\end{proof}
				
				Next we study the behavior of $\cal{P}_{\beta}^{\epsilon,L}(q,\zeta)$ in $\epsilon$ for fixed $(q,\zeta)$. Note that when $A,B\in \R^{\Lambda_L\times \Lambda_L}$ commute, and $\|AB^{-1}\|<1$, we have that
				\[\tr(A+B)^{-1}=\sum_{n=0}^{\infty}(-1)^n\tr(A^{-1-n}B^n).\]	
				As $\Delta_L$ and $A$ commute (since they are both circulant matrices), we see that
				\[\frac{d}{d\epsilon}R_{1;L,t}^{\epsilon}(u)\big|_{\epsilon=0}=-R_{2,A}(u).\]
				Then expanding defining relation of $K_{L,t}^{\epsilon}$ to first order we see that
				\[-R_{2;L,t}\left(K_{L,t}(u)\right)\frac{d}{d\epsilon}K_{L,t}^{\epsilon}(u)\big|_{\epsilon=0}+\frac{d}{d\epsilon}R_{1;L,t}^{\epsilon}(K_{L,t}(u))\big|_{\epsilon=0}=0,\]
				so that in particular
				\[\frac{d}{d\epsilon}K_{L,t}^{\epsilon}(u)\big|_{\epsilon=0}=-\frac{R_{2,A}(K_{L,t}(u))}{R_{2;L,t}(K_{L,t}(u))}=R_{2,A}(K_{L,t}(u))K_{L,t}'(u).\]
				Finally, we note that
				\begin{align}
					\frac{d}{d\epsilon}&\frac{1}{L^{1/2}}\log \det\left(K_{L,t}^{\epsilon}(u)I-t\Delta_L +\epsilon A\right)\big|_{\epsilon=0}\\
					&=L^{1/2}\tr\left((K_{L,t}(u)I-t\Delta_L)^{-1}\left(\frac{d}{d\epsilon}K_{L,t}^{\epsilon}(u)\big|_{\epsilon=0}I+A\right)\right)\\
					&=uR_{2,A}(K_{L,t}(u))K_{L,t}'(u)+R_{1,A}(K_{L,t}(u)).
				\end{align}
				From this we may routinely compute for fixed $(q,\zeta)$ that
				\[\frac{d}{d\epsilon}\cal{P}_{\beta}^{\epsilon,L}(q,\zeta)\big|_{\epsilon=0}=
				\frac{1}{2}\left(-R_{1,A}(K_{L,t}(\beta(q-q_*)))
				-\int_{0}^{q_*}\beta R_{2,A}(K_{L,t}(\beta\delta(u)))K_{L,t}'(\beta\delta(u)))du\right).\]
				Next, note that by Gaussian integration by parts
				\[\frac{d}{d\epsilon}L^{-1/2}N^{-1}\E \log Z_{N,\beta}^{\kappa,q}(\epsilon)|_{\epsilon=0}=-\frac{\beta}{2L}\sum_{x,y\in \Lambda_L}\E\<A_{xy}(\b{u}(x),\b{u}(y))_N\>^{\kappa,q}.\]
				Moreover, by H\"{o}lder's inequality the function $L^{-1/2}N^{-1}\E \log Z_{N,\beta}^{\kappa,q}(\epsilon)$ is also convex in $\epsilon$. As the point-wise limit of a family of convex functions is convex, Proposition \ref{proposition:parisi perturbation-in-proof} implies that the function
				\[\mathscr{P}(\epsilon):=\inf_{\zeta\in P([0,q_L))}\cal{P}_{\beta}^{\epsilon,L}(q_L,\zeta)\]
				is convex in $\epsilon$. However, Proposition \ref{proposition:parisi perturbation-in-proof} shows that for fixed $\epsilon$, the function $\cal{P}_{\beta}^{\epsilon,L}(q_L,\zeta)$ is strictly convex in $\zeta$, with a unique minimizer, which we denote $\zeta_{\epsilon}$. Now if denote the left and right derivatives in $\epsilon$ as $\frac{d}{d\epsilon}^-$ and $\frac{d}{d\epsilon}^+$, we see that convexity of $\mathscr{P}(\epsilon)$ implies that
				\[\frac{d}{d\epsilon}^+ \mathscr{P}(\epsilon)=\liminf_{h\to 0^+}\frac{\mathscr{P}(\epsilon+h)-\mathscr{P}(\epsilon)}{h}.\]
				However for $h>0$
				\[\frac{\mathscr{P}(\epsilon+h)-\mathscr{P}(\epsilon)}{h}=\frac{\cal{P}_{\beta}^{\epsilon,L}(q_L,\zeta_{\epsilon+h})-\cal{P}_{\beta}^{\epsilon,L}(q_L,\zeta_\epsilon)}{h}\le \frac{\cal{P}_{\beta}^{\epsilon,L}(q_L,\zeta_{\epsilon})-\cal{P}_{\beta}^{\epsilon,L}(q_L,\zeta_\epsilon)}{h},\]
				so that 
				\[\frac{d}{d\epsilon}^+ \mathscr{P}(\epsilon)\le \left(\frac{d}{d\epsilon}\cal{P}_{\beta}^{\epsilon,L}\right)(q_L,\zeta_\epsilon).\]
				A similar argument shows that 
				\[\frac{d}{d\epsilon}^- \mathscr{P}(\epsilon)\ge \left(\frac{d}{d\epsilon}\cal{P}_{\beta}^{\epsilon,L}\right)(q_L,\zeta_\epsilon).\]
				Combined, we see that
				\[ \left(\frac{d}{d\epsilon}\cal{P}_{\beta}^{\epsilon,L}\right)(q_L,\zeta_\epsilon)\le \frac{d}{d\epsilon}^- \mathscr{P}(\epsilon)\le \frac{d}{d\epsilon}^+ \mathscr{P}(\epsilon)\le \left(\frac{d}{d\epsilon}\cal{P}_{\beta}^{\epsilon,L}\right)(q_L,\zeta_\epsilon),\label{eqn:ignore-23894298037}\]
				where in the second equality we have used the convexity of $\mathscr{P}(\epsilon)$. In particular, (\ref{eqn:ignore-23894298037}) is actually a chain of equalities, which implies that  $\mathscr{P}(\epsilon)$ is differentiable in $\epsilon$ with derivative 
				\[ \frac{d}{d\epsilon} \mathscr{P}(\epsilon)= \left(\frac{d}{d\epsilon}\cal{P}_{\beta}^{\epsilon,L}\right)(q_L,\zeta_\epsilon),\]
				and so if $\epsilon=0$,
				\[\frac{d}{d\epsilon}\mathscr{P}(\epsilon)\big|_{\epsilon=0}=\frac{d}{d\epsilon}\cal{P}_{\beta}^{\epsilon,L}(q_L,\zetac)\big|_{\epsilon=0}.\label{eqn:convex-derivative-result}\]
				Now we use the fact that if $f_n\to f$ is a pointwise limit of differentiable convex functions, such that $f$ is also differentiable, then $\D f_n\to \D f$ (see Theorem 25.7 of \cite{convexanalysis}). Thus, combining (\ref{eqn:delta-to-zero}) with the computation of (\ref{eqn:convex-derivative-result}) above completes the proof.
			\end{proof}

			\section{Results on The Parisi Measure \label{section:RS 1RSB FRSB results}} 
			
			In this section, we will provide proofs for all our results concerning the structure of the Parisi measure. We will begin by proving Theorems \ref{theorem:massless RS phase transition theorem}, \ref{theorem:intro:1-d continuum T>0:RS}, \ref{theorem:intro:1-d continuum T>0:1RSB} and  \ref{theorem:intro:1-d continuum T>0:FRSB},. Then we will derive Corollaries \ref{corr:intro:RS larkin} and \ref{corr:FRSB}. Each proof will primarily consist of a careful analysis of the stationary point equations in Theorem \ref{theorem:intro:1-d continuum T>0:free energy T>0}.
			
			To begin, we give the proof of Theorem \ref{theorem:intro:1-d continuum T>0:RS}, which follows essentially from the same argument used in the finite $L$ case in \cite{Paper2}.
			
			\begin{proof}[Proof of Theorem \ref{theorem:intro:1-d continuum T>0:RS}]
				The proof is essentially identical to that of Theorem II.1.7, which characterizes the RS-phase for the finite-$L$ model, if one replaces all the quantities $R_1,$ $K,$ etc. by their corresponding continuum quantities in Propositions \ref{proposition:continuum free energy proof:continuum functions-1} and \ref{proposition:continuum free energy proof:continuum functions-3}
			\end{proof}
			
			We now proceed to the proofs of other theorems, for which we have no finite $L$ analogue. For both proofs, though, we will introduce a useful function:
			\[U_B(s)=(2B''(2s)t)^{-1/3}. \label{eqn:def of UB}\]
			We note that $(B''(s))^{-1/3}$ is strictly concave (resp. convex) if and only if $U_B$ is strictly concave (resp. convex). However, it is the function $U_B$ which will appear naturally in our analysis of the stationary point equations in Theorem \ref{theorem:intro:critical point equations}.
			
			\begin{proof}[Proof of Theorem \ref{theorem:intro:1-d continuum T>0:FRSB}]
				By Corollary \ref{corr:intro:RS larkin}, we only need to consider the case where $\mu_{Lar}(\beta;t)$ is well defined, and where we also have $\mu<\mu_{Lar}(\beta;t)$. We will assume this for the remainder of the proof.
				
				Let us first show that there exists a unique triple $(q_0,q_*,q)$, which has $0<q_0<q_*<q$ which satisfies the three equations of (\ref{eqn:theorem:3 FRSB equations}) where we replace instances of $\qc$ with $q$. The first equation of (\ref{eqn:theorem:3 FRSB equations}) can be solved for a unique value $q-q_*>0$, the second equation can be solved for a unique value of $q-q_0>0$, and, fixing this value of $q-q_0$, the third equation may be solved for a unique value of $q_0>0$. This shows that there exists a unique triple, $(q_0,q_*,q)$, which solves (\ref{eqn:theorem:3 FRSB equations}), but we still need to show that this triple satisfies $0<q_0<q_*<q$. However, we have solved for values with $q-q_*>0$ and $q_0>0$, so we only need to show that $q_0<q_*$, or equivalently that $q-q_*<q-q_0$.
				
				For this, note that by re-arranging the defining equation for $\mu_{Lar}(\beta;t)$ (namely (\ref{eqn:Larkin equation})) we have that 
				\[2B'' \left(\frac{2}{\beta\sqrt{\mu_{Lar}(\beta;t)t}}\right)=\sqrt{\mu_{Lar}(\beta;t)^3t}.\] 
				Taking both sides to the power of $-1/3$, we see that
				\[U_B \left(\frac{1}{\beta\sqrt{\mu_{Lar}(\beta;t)t}}\right)=\frac{1}{\sqrt{\mu_{Lar}(\beta;t)t}}.\label{eqn:ignore-1928934}\]
				We may use this and first equation of (\ref{eqn:theorem:3 FRSB equations}) to conclude that
				\[U_B(q-q_*)=U_B \left(\frac{1}{\beta\sqrt{\mu_{Lar}(\beta;t)t}}\right)=\frac{1}{\sqrt{\mu_{Lar}(\beta;t)t}}.\label{eqn:ignore-3948393}\]
				However, we may also rearrange the second equation of  (\ref{eqn:theorem:3 FRSB equations}) as 
				\[U_B(q-q_0)=\frac{1}{\sqrt{\mu t}}.\label{eqn:ignore-32487923}\]
				As $\mu<\mu_{Lar}(\beta;t)$ by assumption, these equations yield that $U_B(q_*-q_0)>U_B(q-q_0)$.
				However,
				\[U_B'(s)=(2t)^{-1/3}\left(\frac{-2B'''(2s)}{3}\right)(B''(2s))^{-4/3}>0,\]
				so that $U_B$ is increasing. This and our previous claim show that $q-q_*>q-q_0$, as desired. With this, we have shown the existence and uniqueness of the desired triple $(q_0,q_*,q)$.
				
				Next, given such a triple $(q_0,q_*,q)$, we verify that the expression for $\zetac$ we claim on the right-hand side of (\ref{eqn:theorem:equation for zeta}) is actually a probability measure (again, replacing $\qc$ with $q$). For this, let $\zeta$ denote the (possibly signed) measure described on the right-hand side of (\ref{eqn:theorem:equation for zeta}). We observe that we may rewrite this as
				\[\zeta([0,s])=\begin{cases}
					\beta^{-1}U_B'(q-s);\;\; s\in [q_0,q_*)\\
					1;\;\; s\in [q_*,q]\\
					0;\;\; s\in [0,q_*)\\
				\end{cases}. \label{eqn:zeta def in U}\]
				We have checked that $U_B'(s)>0$, and as $(B''(s))^{-1/3}$ is strictly concave we have that $U_B''(s)<0$ a.e. Thus $\zeta$ defines a positive measure supported on all of $[q_0,q_*]$. We note that this is the only location in the proof where we need $(B''(s))^{-1/3}$ to be strictly concave instead of simply requiring it to be concave. 
				
				Thus, to show that $\zeta$ is a probability measure, we only need to show that $U_B'(q-q_*)\le \beta$. For this, note by using (\ref{eqn:ignore-3948393}) and the first equation of (\ref{eqn:theorem:3 FRSB equations}), we have that 
				\[U_B(q-q_*)=\frac{1}{\sqrt{\mu_{Lar}(\beta;t)t}}=\beta(q-q_*).\label{eqn:ignore-3892473489}\]
				Now note that any concave function $f$ satisfies
				\[f(x)\ge f'(x)x+f(0).\]
				So as $U_B$ is concave, we have that \[U_B(q-q_*)\ge U_B'(q-q_*)(q-q_*)+U_B(0).\] 
				However, if we note that $U_B(0)=(2B''(0)t)^{-1/3}>0$, then this and the equation $U_B(q-q_*)=\beta(q-q_*)$ ensures that $U_B'(q-q_*)\le \beta$.
				
				Lastly, we show that $(\qc,\zetac)=(q,\zeta)$. To do this, we show that $(q,\zeta)$ is a solution to the stationary equations of Theorem \ref{theorem:intro:1-d continuum T>0:free energy T>0}. Let $\delta$ be the function corresponding to the pair $(q,\zeta)$ as in (\ref{eqn:def:delta-P}). For $s\in [q_0,q_*]$, we see that
				\begin{align}
					\delta(s)&=\int_s^{q}\zeta([0,u])du=(q-q_*)+\int_{s}^{q_*}\beta^{-1}U_B'(q-u)du\\
					&=(q-q_*)-\beta^{-1}U_B(q-q_*)+\beta^{-1}U_B(q-u)=\beta^{-1}U_B(q-u),
				\end{align}
				where in the last equality we have again used (\ref{eqn:ignore-3892473489}). Moreover, for $u\in [0,q_0)$ we have that
				\[\delta(u)=\delta(q_0)=\beta^{-1}U_B(q-q_0)=\frac{1}{\beta\sqrt{\mu t}},\label{eqn:delta for gamma FRSB}\]
				where in the last equality we have used (\ref{eqn:ignore-32487923}). The case of $u=0$ shows the first stationary point equation (\ref{eqn:intro:continuum:minimization eqn:Euclidean: Larkin}).
				
				To show the second identity, we will need to study the function $F_{q}$ defined directly after the statement of this equation. First, note that for $s\in [q_0,q_*]$
				\[F_{q}'(u)=4B''(2(q-u))-\frac{2}{(\beta \delta(u))^3t}=4B''(2(q-u))-\frac{2}{\left(U_B(q-u)\right)^3t}=0.\]
				Second, note that for $u\in (0,q_0]$, 
				\[F_{q}(u)=-2B'(2(q-u))-\int_0^u \frac{2ds}{(\beta \delta(s))^3t}=-2B'(2(q-u))-2u\sqrt{\mu^3t}.\]
				In particular, the third equation of (\ref{eqn:theorem:3 FRSB equations}) may be rewritten as
				\[F_{q}(q_0)=-2B'(2(q-q_0))-2q_0\sqrt{\mu^3 t}=0.\]
				Moreover, for $u\in (0,q_0)$ we have that
				\[F_{q}''(u)=-8B'''(2(q-u))>0,\]
				so that $F_{q}$ is convex on $[0,q_0]$.
				Combining this convexity with the prior observations that $F_{q}(q_0)=0$ and $F_{q}'(q_0)=0$, we see that that $F_{q}(u)\le 0$ for $u\in [0,q_0]$. Moreover, as we have shown that $F'_{q}(u)=0$ for $u\in [q_0,q_*]$, the fact that $F_{q}(q_0)=0$, shows that $F_{q}(u)=0$ for $u\in [q_0,q_*]$. 
				
				Finally, note that for $u\in (q_*,q)$ we have that
				\[F_{q}'(u)=4B''(2(q-u))-\frac{2}{(\beta(q-u))^3t}.\]
				However, the defining property of $\mu_{Lar}(\beta;t)$ can be rearranged to show that for any  $s<\frac{1}{\sqrt{\mu_{Lar}(\beta;t)t}}=\beta(q-q_*)$, one has that
				\[4B''\left(\frac{2s}{\beta}\right)-\frac{2}{s^3t}<0.\]
				In particular, for $u\in (q_*,q)$, we see that $F_{q}'(u)<0$. As we have shown that $F_{q}(q_*)=0$, this shows that $F_{q}(u)\le 0$ for $q\in [q_*,q]$. Collecting all these observations, we see that $F_{q}(u)<0$ if either $u\in [0,q_0)$ or $u\in (q_*,q]$, and that $F_{q}(u)=0$ if $u\in [q_*,q]$. This shows that
				\[\{s\in [0,q]:f_{q}(s)=\sup_{0<s'<q}f_{q}(s')\}=[q_0,q_*].\]
				Now as $\zetac([q_0,q_*])=1$, this verifies the second stationary point equation (\ref{eqn:intro:continuum:minimization eqn:Euclidean: measure}).
			\end{proof}
			
			\begin{proof}[Proof of Theorem \ref{theorem:intro:1-d continuum T>0:1RSB}]
				As before, we will use the notation $f_{q}$ and $F_{q}$ to refer to the functions defined in Theorem \ref{theorem:intro:critical point equations}. 
				
				We first prove that if the model is not RS, then it must be 1RSB. This proof is a modification of the one given by Crisanti and Sommers \cite{crisantisommersOG} for the one-site spherical $p$-spin model. For this, assume that the model is not RS. Let $(\qc,\zetac)$ denote the Parisi pair, and let $\delta$ be the function associated to this pair. As $\zetac$ is not a Dirac mass by the RSB assumption, we may choose two points $q_0<q_*$ such that $q_0,q_*\in \supp(\zetac)$. Note that if $q_0=0$ then (\ref{eqn:intro:continuum:minimization eqn:Euclidean: measure}) implies that $F_{\qc}(0)\ge 0$, however $F_{\qc}0)=-2B''(2q)>0$, which is a contradiction. Thus, we must have that $q_0>0$. Now for (\ref{eqn:intro:continuum:minimization eqn:Euclidean: measure}) to hold, $f_{\qc}$ must be maximized at both $q_0$ and $q_*$, which implies that
				\[
				F_{\qc}(q_0)=F_{\qc}(q_*)=\int_{q_*}^{q}F_{\qc}(ds)=0. \label{eqn:abstract 1RSB condition}
				\]
				This implies that $F'_{\qc}(s)=0$ for at-least two points $s\in [q_0,q_*]$.
				
				Now if $\supp(\zetac)$ contains more than two points. We may apply the argument again to conclude that $F'_{\qc}(s)=0$ has at least three solutions for $s\in (0,q)$. However, $F'_{\qc}(s)=0$ may be rewritten as the equation
				\[
				-4B''(2(\qc-s))=-\frac{2}{(\beta \delta(s))^3t},
				\]
				or equivalently, the equation
				\[
				U_B(\qc-s)=\beta \delta(s), \label{eqn:ignore-3849374}
				\]
				Now, in the case where $(B''(x))^{-1/3}$ is strictly convex, the left-hand side is also strictly convex, and we simply note that the right-hand side is concave. Hence, there can only be at most two solutions to this equation, which gives a contradiction. Thus $\supp(\zetac)$ can contain at most two points, so the measure must be 1RSB. 
				
				In the case where $(B''(x))^{-1/3}$ is linear, the left-hand side of (\ref{eqn:ignore-3849374}) is also linear in $s$. The right-hand side is concave, and for it to coincide with a linear function at three points, say $s_1<s_2<s_3$, it must actually coincide at all points $s\in [s_1,s_3]$. Thus if we let $s_i$ and $s_f$ denote the smallest and largest solution to $F'_{\qc}(s)=0$, we have that (\ref{eqn:ignore-3849374}) holds for all $s\in [s_i,s_f]$. However, the above argument shows that there is some point $s_*\in (s_i,s_f)\cap\supp(\zetac)$. As $\beta \delta'(s)=-\zetac([0,s])$ for a.e. $s$, we see that for $\epsilon>0$, we have $\beta \delta'(s-\epsilon)>\beta \delta'(s+\epsilon)$. However, for small enough $\epsilon>0$, this contradicts (\ref{eqn:ignore-3849374}) as the line on the left-hand size has a fixed slope.
				
				Next, fix some pair $(q,\zeta)$, such that $\supp(\zeta)=\{q_0,q_*\}$ and $0\le q_0<q_*<q$. Let $\delta$ now denote the function associated to $(q,\zeta)$ by (\ref{eqn:def:delta-P}). We will show that if (\ref{eqn:abstract 1RSB condition}) and $\beta \delta(0)=R_{1;L,t}(\mu)$ hold, then $(q,\zeta)=(\qc,\zetac)$. The above argument shows that $F'_q(s)=0$ has at most two solutions. However, combined with (\ref{eqn:abstract 1RSB condition}), we see that $F'_q(s)=0$ must have exactly two solutions, and that both must lie in $(q_0,q_*)$. Let us denote these solutions as $q_0<r_0<r_1<q_*$. We have that $F_q$ is either increasing or decreasing on each of the intervals $[0,r_0]$, $[r_0,r_1]$ and $[r_1,q]$. Now as $F_q(0)=-2B'(2q)>0$ and $F_q(q_0)=0$, we have that $F_q$ is decreasing on the interval $[0,r_0]$. However, (\ref{eqn:abstract 1RSB condition}) to hold $F_q$ must alternate between increasing and decreasing on these intervals, so $F_q$ is also decreasing on the interval $[r_1,q]$, and increasing on the interval $[r_0,r_1]$. Combined with (\ref{eqn:abstract 1RSB condition}) these facts easily show that $f_q(u)=\int_0^u F_q(s)ds$ is maximized at the points $\{q_0,q_*\}$, which shows that $(q,\zeta)$ satisfies (\ref{eqn:intro:continuum:minimization eqn:Euclidean: measure}). As we are assuming the other critical point identity, we see that $(q,\zeta)$ is indeed the pair from Theorem \ref{theorem:intro:critical point equations}, so that $(q,\zeta)=(\qc,\zetac)$. Letting $F=F_q$, simplifying $F_q(u)$ for $u\in [q_0,q_*]$ then completes the proof.
			\end{proof}
			
			\begin{proof}[Proof of Theorem \ref{theorem:massless RS phase transition theorem}]
				Note that if we let $x=2/\beta\sqrt{\mu t}$, then we find that (\ref{eqn: does not have a massless phase transition}) is equivalent to
				\[\limsup_{\mu\to 0}B''\left(\frac{2}{\beta\sqrt{\mu t}}\right)\frac{2}{\sqrt{\mu^3 t}}=\infty.\label{eqn: ignore-no phase}\]
				In particular, this statement is true either for all $\beta>0$ or none. Now assume that (\ref{eqn: does not have a massless phase transition}) holds.
				Then by (\ref{eqn: ignore-no phase}), the RSB criterion (\ref{eqn:g''>0}) holds for all sufficiently small $\mu$, and so the massless model is RSB for all $\beta$. 
				
				Now assume that (\ref{eqn: does not have a massless phase transition}) fails. Then the function  $B''(x)x^3$ is bounded above on $(0,\infty)$ by some number $c$. However, using the same transformation as before, this implies that
				\[\sup_{x\in (0,\infty)}B''\left(\frac{2}{\beta\sqrt{\mu t}}\right)\frac{2}{\sqrt{\mu^3 t}}\le \frac{\beta^3 t c}{4}.\]
				This shows that for $\beta<\left(\frac{4}{t}\right)^{1/3}$, the equation (\ref{eqn:Larkin equation}) is not solvable in $\mu$, so the model is RS for all $\mu$ by Corollary \ref{corr:intro:RS larkin}. This shows that the massless model is RS for sufficiently small $\beta$ if (\ref{eqn: does not have a massless phase transition}) doesn't hold. 
				
				All that is left is to show that if (\ref{eqn: does not have a massless phase transition}) fails, then the massless model is still RSB for sufficiently large $\beta$. For this we will use Theorem \ref{theorem:intro:1-d continuum T>0:RS}, and in particular the function $g_{\beta,\mu}$ defined in (\ref{eqn:ignore-luna-10}). We will show that for $\beta$ sufficiently large, we have that
				\[\lim_{\mu\to 0}\left(g_{\beta,\mu}(1)-g_{\beta,\mu}\left(\frac{1}{\sqrt{\mu t}}\right)\right)>0\label{eqn:ignore-23849236}\]
				This suffices to show that the massless model is RSB for sufficiently large $\beta$. We record that
				\[g_{\beta,\mu}\left(\frac{1}{\sqrt{\mu t}}\right)=\beta^2 B\left(\frac{2}{\beta\sqrt{\mu t}}\right)-\frac{\sqrt{\mu t}}{ t}-\frac{1}{\sqrt{\mu t}}\left(2\beta B'\left(\frac{2}{\beta\sqrt{\mu t}}\right)+\mu\right).\label{eqn:ignore-3482793}\]
				We want to take the limit of this as $\mu\to 0$. All terms clearly go to zero except for the second-to-last. However, by writing $B'$ as an integral of $B''$ as in (\ref{eqn:ignore-integral_b}), we see that $\lim_{x\to \infty} xB'(x)=0$, which shows that the second to last term in (\ref{eqn:ignore-3482793}) also vanishes in the limit. From (\ref{eqn:ignore-3482793}) we see that
				\[\lim_{\mu \to 0}g_{\beta,\mu}\left(\frac{1}{\sqrt{\mu t}}\right)=0.\]
				One may similarly, but more easily, show that
				\[\lim_{\mu\to 0}g_{\beta,\mu}(1)=\beta^2 B\left(\frac{2}{\beta\sqrt{t}}\right)-\frac{1}{t}.\]
				As $\lim_{\beta \to \infty} B\left(\frac{2}{\beta\sqrt{t}}\right)=B(0)>0$, this suffices to demonstrate (\ref{eqn:ignore-23849236}).
			\end{proof}
			
			We now move on to the corollaries to these results. However, before proving Corollary \ref{corr:intro:RS larkin}, we will prove a useful supporting lemma, which is analogous to our Lemma II.6.2 in the case of finite $L$. For this lemma, recall $g_{\beta,\mu}$ is the function from Theorem \ref{theorem:intro:1-d continuum T>0:RS}.
			
			\begin{lem}
				\label{lem:RS supplement}
				Fix $(\beta, t)$. If $\mu'\ge 0$ is such that for all $\mu\ge \mu'$, one has $g_{\beta,\mu}'' \left(\frac{1}{\sqrt{\mu t}}\right)\le 0$, then the model is RS for the parameters $(\beta,t,\mu')$.
			\end{lem}
			\begin{proof}
				Note first that while $g_{\beta,\mu}$ depends on the choice of $\mu$, the second derivative of this function, $g_{\beta,\mu}''$, does not depend on $\mu$. In particular, the conditions of the lemma immediately imply that for any $x\le \frac{1}{\sqrt{\mu' t}}$, we have that $g''_{\beta,\mu'}(x)\le 0$.
				
				This shows that $g_{\beta,\mu'}$ is concave on $\left[0,\frac{1}{\sqrt{\mu' t}}\right]$. As $g'_{\beta,\mu'}\left(\frac{1}{\sqrt{\mu' t}}\right)=0$, this implies that $\frac{1}{\sqrt{\mu' t}}$ maximizes $g_{\beta,\mu'}$ on the interval $\left[0,\frac{1}{\sqrt{\mu' t}}\right]$. By Theorem \ref{theorem:intro:1-d continuum T>0:RS}, this shows that the model is RS.
			\end{proof}
			
			\begin{proof}[Proof of Corollary \ref{corr:intro:RS larkin}]
				Fix $(\beta,t)$. The left hand side of (\ref{eqn:Larkin equation}) goes to zero as $\mu \to \infty$, i.e. 
				\[ \lim_{\mu \to \infty}B''\left(\frac{2}{\beta\sqrt{\mu t}}\right)\frac{2}{\sqrt{\mu^3 t}}=0.\label{eqn:ignore-34872983}\]
				In particular, if (\ref{eqn:Larkin equation}) has no solutions, then one has for all $\mu>0$ that
				\[ B''\left(\frac{2}{\beta\sqrt{\mu t}}\right)\frac{2}{\sqrt{\mu^3 t}}<1.\]
				This equation may be rewritten as $g''_\beta \left(\frac{1}{\sqrt{\mu t}}\right)\le 0$ (see (\ref{eqn:g''})). By Lemma \ref{lem:RS supplement}, this implies that the model is RS for each $\mu>0$, establishing the first bullet point. Again using (\ref{eqn:ignore-34872983}), the second bullet point follows by a similar argument,'
			\end{proof}
			
			Next, we will prove Corollary \ref{corr:FRSB}. As both the requirements and results of Theorem \ref{theorem:intro:1-d continuum T>0:FRSB} are fairly explicit, most of the work will just be solving these equations for the Parisi measure.
			
			\begin{proof}[Proof of Corollary \ref{corr:FRSB}]
				The first step will be to analyze the Larkin equation (\ref{eqn:Larkin equation}), which in this case is given by
				\[\gamma(\gamma+1)\left(1+\frac{2}{\beta \sqrt{\mu t}}\right)^{-\gamma-2}\frac{2}{\sqrt{\mu^3 t}}=1. \label{eqn:Larkin eqn frsb}\]	
				We observe that
				\[\lim_{\mu\to 0}\left(1+\frac{2}{\beta \sqrt{\mu t}}\right)^{-\gamma-2}\frac{2}{\sqrt{\mu^3 t}}=\infty \text{ and } \lim_{\mu\to \infty}\left(1+\frac{2}{\beta \sqrt{\mu t}}\right)^{-\gamma-2}\frac{2}{\sqrt{\mu^3 t}}=0,\]
				so that for any fixed $(\beta,t)$ the equation (\ref{eqn:Larkin eqn frsb}) has at least one solution in $\mu$. However, notice that for $c>0$ and $x>0$
				\[\frac{d}{dx}\left((1+cx)^{-\gamma-2}x^3\right)=x^2(1+cx)^{-\gamma-3}\left(3+c(1-\gamma)x\right)>0. \label{eqn:ignore-37893489}\]
				In particular, with $c=\frac{2}{\beta \sqrt{ t}}$ and $x=\frac{1}{\sqrt{\mu}}$ we see that the left hand side of (\ref{eqn:Larkin eqn frsb}) is decreasing in $\mu$. In particular, (\ref{eqn:Larkin eqn frsb}) is precisely one solution in $\mu$ for all fixed values of $(\beta,t)$. When combined with (\ref{corr:intro:RS larkin}), this yields all statements except for the explicit form for the pair $(\qc,\zetac)$.
				
				To find this, we will first solve for triple $(q_0,q_*,\qc)$. We observe that the equations (\ref{eqn:theorem:3 FRSB equations}) may be rewritten as
				\[\qc-q_*=\frac{1}{\beta\sqrt{\mu_{Lar}(\beta;t)t}}, \; \frac{2\gamma(\gamma+1)}{(1+2(\qc-q_0))^{\gamma+2}}=\sqrt{\mu^{3}t},\; q_0= \frac{\mu^3 t \gamma}{ (1+2(\qc-q_0))^{\gamma+1}}.\]
				Recalling the constant $c_0=(\frac{2\gamma(\gamma+1)}{\sqrt{t}})^{1/(\gamma+2)}$, we may rewrite the second equation as
				\[\qc-q_0=\frac{1}{2}\left(\frac{c_0}{\mu^{\frac{3}{2(2+\gamma)}}}-1\right).\]
				Using this, and recalling the other constant $c_1=\frac{\gamma t}{c_0^{1+\gamma}}$, we may rewrite the third equation as
				\[q_0=c_1\mu^{\frac{3(\gamma+1)}{2(2+\gamma)}+3}.\]
				Finally, the first equation gives the form of $ \qc-q_*$, which easily yields our claimed equations for the triple $(q_0,q_*,\qc)$. Now, finally, we note that in this case
				\[U_B(x)=\frac{(1+2x)^{(\gamma+2)/3}}{(2t\gamma(\gamma+1))^{1/3}}.\label{eqn:UB in gamma FRSB}\]
				Observing the version of $\zeta$ written in terms of $U_B$ (namely (\ref{eqn:zeta def in U})), this may easily be simplified to prove the desired result.
			\end{proof}
			
			Now finally, we comment on Remark \ref{remark:gamma=1 remark}. As we noted during the proof of Theorem \ref{theorem:intro:1-d continuum T>0:FRSB}, the only place where we use strict convexity is in showing that the equation given for $\zetac$ is FRSB. In particular, this shows that the results of Theorem \ref{theorem:intro:1-d continuum T>0:FRSB} hold for $B(x)=(1+x)^{-1}$, except the measure given in (\ref{eqn:theorem:equation for zeta}) is 1RSB and not FRSB. Moreover, the second paragraph of the proof of Corollary \ref{corr:FRSB} works to show the (\ref{eqn:FRSB triple eqns explicit}) and (\ref{eqn:FRSB zeta eqns explicit}) still hold as well.
            
            \subsection{Acknowledgements}

            P.K acknowledges the support of NSF grant DMS-2202720.
			
			\appendix
			
			\section{Discrete Approximation Results \label{section: approximation proofs}}
			
			In this appendix, we prove our results concerning the discrete limit of functions involving the Laplacian, following the heuristics given in Section \ref{section: continuum limits}. More concretely, the primary purpose of this section is to prove Propositions \ref{proposition: divergent factor}, \ref{proposition:continuum free energy proof:continuum functions-1}, \ref{proposition:continuum free energy proof:continuum functions-3}, and \ref{prop: second approximation: the one for G}.
			
			All of the functions considered in these propositions can be defined in terms of the function $R_{1; L,t}(\mu)$ by some combination of integration, differentiation, and inversion. Using general results, we will essentially be able to derive each limit from very precise asymptotics for $R_{1; L,t}(\mu)$ (Lemma \ref{lem:continuum:1d-limit result for resolvent}). The key idea here is that we want to treat (\ref{eqn:continuum functions: linear functions of f}) as an integral approximation. However, due to the singular nature of the resolvent (especially at $\mu=0$), we will find it easier to instead represent $R_{1; L,t}(\mu)$ in terms of a less singular function - the heat kernel - which will make this approximation easier. We note that when $A$ is positive semi-definite and $\mu>0$
			\[\int_0^\infty e^{-t\mu}\tr\left(e^{-tA}\right)dt=\tr(\mu I+A)^{-1}.\label{eqn: resolvent identity involving exponential}\]
			Traces of other powers may also be recovered from the identity
			\[\tr(\mu I+A)^{-z}=\frac{1}{\Gamma(z)}\int_0^\infty t^{z-1} e^{-t\mu}\tr\left(e^{-tA}\right)dt.\label{eqn: resolvent identity involving exp general}\]
			So we only need to obtain precise estimates for $\tr\left(e^{-tA}\right)$ to obtain results for the trace of the resolvent. We do this in the following lemma.
			
			\begin{lem}
				\label{lem:bounds on the trace of exponential}
				There is an absolute constant $C>0$ such that for any $L\in \N$ and $t>0$,
				\[\left|L^{1/2}\tr\left(e^{t\Delta_L}\right)-\frac{1}{2\sqrt{\pi t}}\right|\le C\left(\frac{1}{L^{1/2}}+\frac{e^{-tL^{1/2}}}{\sqrt{t}}+\frac{1}{\sqrt{tL}}\right).\]	
			\end{lem}
			\begin{proof}
				Before proceeding to the proof, we note that
				\[2\int_0^\infty e^{-4t\pi^2 x^2}dx=\int_{-\infty}^\infty e^{-4t\pi^2 x^2}dx=\frac{1}{2\sqrt{\pi t}}.\label{eqn:trace continuum integral}\]
				The idea of the proof is to show that when one expands the quantity $L^{1/2}\tr\left(e^{t\Delta_L}\right)$ as a sum of eigenvalues, this sum approximates the integral on the left. Note by (\ref{eqn:continuum functions: linear functions of f}) that
				\[L^{1/2}\tr\left(e^{t\Delta_L}\right)=\frac{2}{L^{1/2}}\sum_{k=1}^{[(L-1)/2]}e^{-4tL\sin^2(\pi k/L)}+\frac{1}{L^{1/2}}\left(1+I_{L\text{ is even}}e^{-4tL}\right). \label{eqn:formula to trace of exponential}\]
				The second term is less than $2/L^{1/2}$, so we focus on bounding the first term. Our first step will be to show that we may neglect the tail of this sum. For this, note that for $x\in (0,1/2]$, one has that $\sin^2(\pi x)\ge x^2$. From this we see that
				\[\frac{1}{L^{1/2}}\sum_{k=\ell+1}^{[(L-1)/2]}e^{-4tL\sin^2(\pi k/L)}\le \frac{1}{L^{1/2}}\sum_{k=\ell+1}^{[(L-1)/2]}e^{-4t\left(\frac{k}{L^{1/2}}\right)^2}.\]
				We recognize the term on the right as the right Riemann sum with spacing $\Delta x=L^{-1/2}$ for the integral $\int_{\ell L^{-1/2}}^{[(L-1)/2]L^{-1/2}}e^{-4tx^2}dx$. As the function $e^{-4tx^2}$ is decreasing on $[0,\infty)$, this sum is bounded above by the integral itself, so that
				\[\frac{1}{L^{1/2}}\sum_{k=\ell+1}^{[(L-1)/2]}e^{-4t\left(\frac{k}{L^{1/2}}\right)^2}\le \int_{\ell L^{-1/2}}^{[(L-1)/2]L^{-1/2}}e^{-4tx^2}dx.\label{eqn:ignore-tail-1}\]
				To bound the term on the right, we recall an elementary bound, which follows by applying the Chernoff bound to a Gaussian random variable: for any $x\ge 0$ and $\alpha>0$, we have that
				\[\int_x^\infty e^{-\alpha y^2}dy\le \sqrt{\frac{\pi}{\alpha}} e^{-\alpha x^2}.\label{eqn:erf-bound}\]
				From this we see that
				\[\int_{\ell L^{-1/2}}^{[(L-1)/2]L^{-1/2}}e^{-4tx^2}dx\le \int_{\ell L^{-1/2}}^{\infty}e^{-4tx^2}dx \le \sqrt{\frac{\pi}{4t}}e^{-4t\ell^2L^{-1}}.\label{eqn:ignore-tail-2}\]
				If we take $\ell=[L^{3/4}]$, it is easy to see that this term is bounded above by $\frac{C}{\sqrt{t}}e^{-t L^{1/2}}$. Noting the last two terms of (\ref{eqn:formula to trace of exponential}) are bounded by $2/L^{1/2}$, we see that to complete the proof it suffices to show that
				\[\left|\frac{1}{L^{1/2}}\sum_{k=1}^{[L^{3/4}]}e^{-tL\sin^2(\pi k/L)}-\frac{1}{4\sqrt{\pi t}}\right|\le CL^{-1/2}+\frac{C}{\sqrt{tL}}
				.\label{eqn:lemma-red-2202}\]
				Let us denote $\ell=[L^{3/4}]$. By using the inequality $\sin^2(x)\le x^2$ we see that
				\[\frac{1}{L^{1/2}}\sum_{k=1}^{\ell}e^{-4tL\sin^2(\pi k/L)}\ge \frac{1}{L^{1/2}}\sum_{k=1}^{\ell}e^{-4t\left(\frac{\pi k}{L^{1/2}}\right)^2}.\]
				We may again employ an integral comparison (this time using a left Riemann sum) to see that
				\[\frac{1}{L^{1/2}}\sum_{k=1}^{\ell}e^{-4t\left(\frac{\pi k}{L^{1/2}}\right)^2}\ge \int_{L^{-1/2}}^{\ell L^{-1/2}}e^{-4t\pi^2x^2}dx\ge \int_{0}^{\ell L^{-1/2}}e^{-4t\pi^2x^2}dx-L^{-1/2}.\]
				Using (\ref{eqn:erf-bound})
				\[\int_0^{\ell L^{-1/2}}e^{-4t\pi^2x^2}dx\ge \int_0^{\infty}e^{-4t\pi^2x^2}dx-\frac{1}{4\sqrt{\pi t}}e^{-4\pi^2 \ell^2L^{-1/2}},\]
				and so by using (\ref{eqn:trace continuum integral}) we see that
				\[\frac{1}{L^{1/2}}\sum_{k=1}^{\ell}e^{-4tL\sin^2(\pi k/L)}-\frac{1}{4\sqrt{\pi t}}\ge -\frac{1}{4\sqrt{\pi t}}e^{-\pi^2 [L^{3/4}]^2L^{-1/2}}-L^{-1/2},\label{eqn:ignore-202032}\]
				which is more than sufficient to show the lower bound needed for (\ref{eqn:lemma-red-2202}). For the upper bound, we use the global inequality $\sin^2(x)\ge x^2(1-x^2)$ to obtain that
				\[\frac{1}{L^{1/2}}\sum_{k=1}^{\ell}e^{-tL\sin^2(\pi k/L)}\le \frac{1}{L^{1/2}}\sum_{k=1}^{\ell}e^{-t \left(\frac{\pi k}{L^{1/2}}\right)^2\left(1-L^{-2}\pi^2 k^2\right)}\] 
				\[\le \frac{1}{L^{1/2}}\sum_{k=1}^{\ell}e^{-t(1-L^{-2}\pi^2 \ell^2) \left(\frac{\pi k}{L^{1/2}}\right)^2} \le \int_{0}^{\infty}e^{-t(1-L^{-2}\pi^2 \ell^2) \pi^2x^2}dx.\]
				Evaluating the integral on the right, and noting that $(1-x)^{-1/2}\le 1+x$ for $x\in (0,1/4)$, we see that if $2\pi \ell\le L$, we have the final bound of
				\[\int_{0}^{\infty}e^{-t(1-L^{-2}\pi^2 \ell^2) \pi^2x^2}dx=\frac{1}{2\sqrt{\pi t}}\frac{1}{\sqrt{1-L^{-2}\pi^2 \ell^2}}\le \frac{1}{2\sqrt{t\pi}}+\frac{L^{-2}\pi^2 \ell^2}{2\sqrt{t\pi}}.\]
				Thus, for sufficiently large $L$, this gives the upper bound of 
				\[\frac{2}{L^{1/2}}\sum_{k=1}^{\ell}e^{-tL\sin^2(\pi k/L)}-\frac{1}{2\sqrt{\pi t}}\le \frac{L^{-2}\pi^2 \ell^2}{\sqrt{t\pi}}.\]	
				Combined with (\ref{eqn:ignore-202032}) this shows (\ref{eqn:lemma-red-2202}).
			\end{proof}
			
			Combining this estimate with (\ref{eqn: resolvent identity involving exponential}), we are able to obtain a bound for $R_{1; L,t}$.
			\begin{lem}
				\label{lem:continuum:1d-limit result for resolvent}
				There is an absolute constant $C>0$ such that for any $L\in \N$ and $\mu,t>0$
				\[\left|R_{1;L,t}(\mu)-\frac{1}{\sqrt{t\mu}}\right|\le \frac{C}{L^{1/4}}\left(\frac{1}{L^{1/4}\mu }+\frac{1}{\sqrt{t}}+\frac{1}{L^{1/4}\sqrt{\mu t}}\right), \label{eqn:R1 d=1 bound}\]
			\end{lem}
			\begin{proof}	
				We observe that $R_{1;L,t}(\mu L^{-1/2})$ may be rewritten as
				\[L^{1/2}\tr(\mu I -t \Delta_L)^{-1}=\int_0^{\infty}e^{-\mu y}L^{1/2}\tr(e^{y t \Delta_L})dy.\]
				In particular, we have that
				\[\left|\int_0^{\infty}e^{-\mu y}L^{1/2}\tr(e^{y t \Delta_L})dy-\int_0^{\infty}\frac{e^{-\mu y}}{\sqrt{\pi y t}}dy\right|\le C\int_0^{\infty}e^{-\mu y}\left(\frac{1}{L^{1/2}}+\frac{e^{-ytL^{1/2}}}{\sqrt{yt}}+\frac{1}{\sqrt{yt L}}\right)dy.\label{eqn:ignore-R1-error}\]
				Noting that $\int_0^{\infty}e^{-\alpha^{-2} x}x^{-1/2}dx=\sqrt{\pi}\alpha$, we see that the left-hand side coincides with the left-hand side (\ref{eqn:R1 d=1 bound}), and the integral on the right-hand side evaluates to
				\[\frac{1}{\mu L^{1/2}}+\frac{\sqrt{\pi}}{\sqrt{\mu+tL^{1/2}}}+\frac{\sqrt{\pi}}{\sqrt{t\mu L}}.\]
				This immediately gives the stated bound.
			\end{proof}
			
			Our next goal is to prove Proposition \ref{proposition: divergent factor}, though we will actually prove a more precise result which holds to order $o(1)$. For this, we employ the identity
			\[\frac{d}{d\mu}\log\det(\mu I- A)=L \tr((\mu I-A)^{-1}),\]
			which will allow us to understand $\log\det(\mu I- t\Delta_L)$ by using the resolvent bounds in Lemma \ref{lem:continuum:1d-limit result for resolvent} as well as a computation of $\log\det(\mu I- t\Delta_L)$ as $\mu \to 0$.
			
			\begin{lem}
				\label{lem:continuum:asymptotics of the determinant}
				For every choice of $\epsilon>0$, there is $C>0$ such that for any $L\in \N$ and $\mu,t\in [\epsilon,\epsilon^{-1}]$ we have that
				\[\left|\frac{1}{L^{1/2}}\log\det\left(\mu I-t\Delta_L\right)-L\log(L)-L^{1/2} \log t-\frac{2\sqrt{\mu}}{\sqrt{t}}\right|\le \frac{C}{L^{1/4}}.\label{eqn: one point determinant asymptotic}\]
			\end{lem}
			\begin{proof}
				We start by studying the behavior around $\mu=0$. We note that
				\[\lim_{\mu\to 0}\left(\log\det(\mu I- t\Delta_L )-\log \mu\right)=\log \mathrm{det}_+(-t\Delta_L),\label{eqn:psudo-det-limit}\]
				where $\mathrm{det}_+$ denotes the pseudo-determinant (i.e., the product of all non-zero eigenvalues). This pseudo-determinant is quite easy to compute here by using Kirchhoff's matrix-tree theorem \cite{Matrix-Tree}, which states that for any connected graph $G$,
				\[\mathrm{det}_+(\Delta_G)=|G|\times \#\{\Gamma \subseteq G:\Gamma \text{ is a spanning tree}\},\]
				where $\Delta_G$ is the standard graph Laplacian. As the underlying graph for $\Lambda_L$ here is $\mathbb{Z}/L\mathbb{Z}$, which is cyclic, it is easy to see that the number of spanning trees coincides with the number of vertices. However, as our Laplacian is weighted and follows the physics sign convention, it is actually the matrix $-L^{-1}\Delta_L$ to which we apply this theorem. Altogether we obtain that $\mathrm{det}_+(-L^{-1}\Delta_L)=L^2$, so that
				\[\log \mathrm{det}_+(-t\Delta_L)=(L-1)(\log(L)+\log t)+2\log(L).\label{eqn:kirkoff for L}\]
				Now for any $\epsilon>0$ we have that
				\[L^{-1/2}\log \mathrm{det}(\mu I-t\Delta_L)=L^{-1/2}\log \mathrm{det}(\epsilon I-t\Delta_L)+\int_{\epsilon}^\mu R_{1;L,t}\left(\mu'\right)d\mu',\]
				and so we find by using (\ref{eqn:psudo-det-limit}) and the identity $\int_\epsilon^\mu\frac{d\mu'}{L^{1/2}\mu'}=L^{-1/2}\log(\mu)-L^{-1/2}\log(\epsilon)$ that 
				\begin{align}
					\lim_{\epsilon\to 0}&\bigg(L^{-1/2}\log \mathrm{det}(\epsilon I-t\Delta_L)+\int_{\epsilon}^\mu R_{1;L,t}\left(\mu'\right)d\mu'\bigg)\\
					=&L^{-1/2}\log \mathrm{det}_+(-t \Delta_L)+L^{-1/2}\log (\mu) +\int_{0}^\mu \bigg( R_{1;L,t}\left(\mu'\right)-\frac{1}{L^{1/2}\mu'}\bigg)d\mu'\\
					=&\frac{(L-1)}{L^{1/2}}(\log(L)+\log t)+2L^{-1/2}\log(L)+L^{-1/2}\log (\mu)+\int_{0}^\mu \bigg(R_{1;L,t}\left(\mu'\right)-\frac{1}{L^{1/2}\mu'}\bigg)d\mu',\label{eqn:ignore-23784623}
				\end{align}
				where in the last equality we have used (\ref{eqn:kirkoff for L}).
				We observe that the first three terms in the rightmost expression may be rearranged as 
				\[L^{1/2}\left(\log(L)+\log t\right)+L^{-1/2}\log(L)+L^{-1/2}\log \left(\frac{\mu}{t}\right).\label{eqn:ignore-0234762347}\]
				If $\mu,t\in [\epsilon,\epsilon^{-1}]$, the last two terms are bounded above by $CL^{-1/4}$ for some $C:=C(\epsilon)>0$, and so only the first term is important. Using this, we see that to show (\ref{eqn: one point determinant asymptotic}) it suffices to show the following bound:
				\[\left|\int_{0}^\mu \left(R_{1;L,t}\left(\mu'\right)-\frac{1}{L^{1/2}\mu'}\right)d\mu'-\frac{2\sqrt{\mu}}{\sqrt{t}}\right|\le \frac{C}{L^{1/4}}.\label{eqn:ignore-0238234}\]
				To do this, we first control the value of this integral around $\mu'=0$. For this, we note from the description of the eigenvalues of $-\Delta_L$ given by (\ref{eqn:eigenvalues of laplacian}), $-\Delta_L$ has zero as an eigenvalue with multiplicity one, and all of the non-zero eigenvalues of $-\Delta_L$ are greater than or equal to $4L\sin^2(\pi L^{-1})$. Again using the bound $\sin^2(\pi x)\ge x^2$ for $x\in (0,1/2)$, we see that $\sin^2(\pi L^{-1})\ge L^{-2}$. In particular, this shows that
				\[0\le R_{1;L,t}(\mu;t)-\frac{1}{L^{1/2}\mu}\le \left(\frac{L-1}{L}\right)L^{1/2}\frac{1}{\mu+4tL^{-1}}\le \frac{L^{1/2}}{\mu+4tL^{-1}}.\]
				In particular
				\[0\le \int_{0}^{\mu e^{-L^{1/4}}} \left(R_{1,L}\left(\mu' L^{-1/2};tL^{1/2}\right)-\frac{1}{L^{1/2}\mu'}\right)d\mu'\le \int_0^{\mu e^{-L^{1/4}}}\frac{L^{1/2}}{\mu'+4tL^{-1}}d\mu' \]
				\[=L^{1/2}\log \left(1+\frac{\mu e^{-L^{1/4}}}{4tL^{-1}}\right)\le L^{3/2}\frac{\mu}{4t}e^{-L^{1/4}},\]
				where in the last inequality we have used that $\log(1+x)\le x$. Note that the rightmost term is far smaller than $CL^{-1/4}$. Next we consider the integral on $(\mu e^{-L^{1/4}},\mu)$. Using Lemma \ref{lem:continuum:1d-limit result for resolvent} we have that
				\[\left|\int_{\mu e^{-L^{1/4}}}^\mu \left(R_{1;L,t}\left(\mu' \right)-\frac{1}{\sqrt{t\mu'}}\right)d\mu'\right|\le \frac{C}{L^{1/4}}\int_{\mu e^{-L^{1/4}}}^\mu \left(\frac{1}{L^{1/4}\mu' }+\frac{1}{\sqrt{t}}+\frac{1}{L^{1/4}\sqrt{\mu' t}}\right)d\mu' \]
				\[\le \frac{C}{L^{1/4}} \left(\frac{\log(\mu/\mu e^{-L^{1/4}})}{L^{1/4} }+\int_{0}^\mu \left(\frac{1}{\sqrt{t}}+\frac{1}{L^{1/4}\sqrt{\mu' t}}\right)d\mu'\right)= \frac{C}{L^{1/4}}\left(1+ \frac{\mu}{\sqrt{t}}+\frac{2}{L^{1/4}}\sqrt{\frac{\mu}{t}}\right),\label{eqn:ignore-alex}\]
				which if $\mu,t\in [\epsilon,\epsilon^{-1}]$ is obviously bounded by $C_1 L^{-1/4}$ for some $C_1>0$. Thus if we note that
				\[\int_{\mu e^{-L^{1/4}}}^\mu\frac{1}{L^{1/2}\mu'}d\mu'=\frac{1}{L^{1/2}}\log \left(\frac{\mu}{\mu e^{-L^{1/4}}}\right)=\frac{L^{1/4}}{L^{1/2}}=\frac{1}{L^{1/4}},\] 
				we see that
				\begin{align}&\left|\int_{0}^\mu \left(R_{1,L}\left(\mu' L^{-1/2};tL^{1/2}\right)-\frac{1}{L^{1/2}\mu'}-\frac{1}{\sqrt{\mu t}}\right)d\mu'\right|\\
					&\le L^{5/2}e^{-L^{1/4}}\frac{\mu}{4t}+\frac{C_1}{L^{1/4}}+\int_0^{\mu e^{-L^{1/4}}}\frac{1}{\sqrt{\mu' t}}d\mu'+\int_{\mu e^{-L^{1/4}}}^\mu\frac{1}{L^{1/2}\mu'}d\mu'\\
					&=L^{5/2}e^{-L^{1/4}}\frac{\mu}{t}+\frac{C_1}{L^{1/4}}+\frac{2\sqrt{\mu} e^{-L^{1/4}/2}}{\sqrt{t}}+L^{-1/4}.
				\end{align}
				This is again bounded by $C_{2}L^{-1/4}$ for some large $C_2:=C_2(\epsilon)>0$.
				Combining this with the equation
				\[\int_0^\mu \frac{d\mu'}{\sqrt{\mu' t}}d\mu'=\frac{2\sqrt{\mu}}{\sqrt{t}}\]
				demonstrates (\ref{eqn:ignore-0238234}) and thus (\ref{eqn: one point determinant asymptotic}).
			\end{proof}
			Proposition \ref{proposition: divergent factor} immediately follows from this lemma. We now consider Propositions \ref{proposition:continuum free energy proof:continuum functions-1} and \ref{proposition:continuum free energy proof:continuum functions-3}. As the proofs of these propositions are interconnected, we will give their proofs simultaneously. The claim that Proposition \ref{proposition:continuum free energy proof:continuum functions-1} makes for $R_{1; L,t}$ follows immediately from Lemma \ref{lem:continuum:1d-limit result for resolvent}. Moreover, as all the other functions in these propositions can be obtained from $R_{1; L,t}$ by taking a variety of derivatives and inverses, we will be able to obtain all the remaining statements using general results.
			
			\begin{proof}[Proof of Propositions \ref{proposition:continuum free energy proof:continuum functions-1} and \ref{proposition:continuum free energy proof:continuum functions-3}]
				First, we establish some generalities. In this lemma, we will say that a sequence of functions $f_n:(0,\infty)\to \R$ converges to $f:(0,\infty)\to \R$ compactly if for each compact subset $K\subseteq (0,\infty)$, the restricted functions $f_n|_{K}$ converge uniformly to $f|_{K}$. If we assume that each $f_n$ is decreasing and that $f$ is continuous, then Dini's theorem implies that if $f_n$ converges to $f$ point-wise, then $f_n$ also converges to $f$ compactly. An identical result holds if each function is increasing. 
				
				Now assume that each $f_n$ is also continuously differentiable and convex, and that $f$ is also differentiable and convex. Then by Theorem 25.7 of \cite{convexanalysis}, if $f_n$ converges to $f$ pointwise, we have that $f_n'$ converges to $f'$ pointwise as well. Thus, by Dini's theorem, we have that $f_n'$ converges to $f'$ compactly.
				
				Now instead assume that each $f_n$ is a decreasing homeomorphism $f_n:(0,\infty)\to (0,\infty)$, and that $f:(0,\infty)\to (0,\infty)$ is a decreasing homeomorphism as well. Then by Proposition X.10 of \cite{BourbakiII}, if $f_n$ converges point-wise to $f$, we also have that the inverse functions $f_n^{-1}$ converges point-wise to $f^{-1}$, and so by Dini's theorem $f_n^{-1}$ converges compactly to $f^{-1}$.
				
				As noted, Lemma \ref{lem:continuum:1d-limit result for resolvent} implies that $R_{1;L,t}(\mu)$ converges to $\frac{1}{\mu^2 t}$ compactly (as $L\to \infty$). We also note that $R_{1;L,t}'(\mu)=-R_{2;L,t}(\mu)$ and the inverse function of $R_{1;L,t}$ is $K_{L;t}$, and similarly $\frac{d}{d\mu}\frac{1}{\sqrt{\mu t}}=\frac{-1}{2\sqrt{\mu^3 t}}$ and the inverse function of $\frac{1}{\sqrt{\mu t}}$ (in $\mu$) is $\frac{1}{\mu^2 t}$. We also note that each $R_{1; L,t}:(0,\infty)\to (0,\infty)$ is continuously differentiable, decreasing, convex, homeomorphism (and similarly for the function $\frac{1}{\mu^2 t}$). In particular, this satisfies all of the above generalities, which shows that $K_{L,t}(\mu)$ converges to $\frac{1}{\mu^2 t}$ compactly, and that $R_{2; L,t}(\mu)$ converges to $\frac{-1}{2\sqrt{\mu^3 t}}$ compactly as well. This suffices to show all the claims involving $K_{L,t}$ and $R_{2; L,t}$.
				
				Moreover, the functions $K_{L,t}:(0,\infty)\to (0,\infty)$ can also easily be seen to be continuously differentiable, decreasing, homeomorphisms. Moreover, by differentiating the identity $R_{1;L,t}(K_{L,t}(\mu))=\mu$, one may obtain that 
				\[K_{L,t}'(\mu)=\frac{-1}{R_{2;L,t}(K_{L,t}(\mu))},\label{eqn:ignore-34387823}\]
				and differentiating again
				\[K''_{L,t}(\mu)=\frac{-2 (L^{-1}\tr((K_{L,t}(\mu)L^{-1/2}I-tL^{-1/2}\Delta_L)^{-3}))K'_{L,t}(\mu)}{R_{2;L,t}(\mu)}>0,\]
				so that $K_{L,t}$ is convex as well. As the same can be said for $\frac{1}{\mu^2 t}$, we see that the above generalities imply that $K_{L,t}'(\mu)$ converges to $-\frac{2}{\mu^3t}$ compactly. Finally, we observe from (\ref{eqn:ignore-34387823}) that $-K_{L,t}':(0,\infty)\to (0,\infty)$ is also a continuously differentiable, decreasing homeomorphism. These facts also hold for $\frac{2}{\mu^3t}$. Thus the inverse function of $-K_{L,t}'$, namely $U_{L,t}$, converges to the inverse function of $\frac{2}{\mu^3t}$, namely $\left(\frac{2}{\mu t}\right)^{1/3}$, compactly. These convergences verify the remaining claims.
			\end{proof}
			
			Now that the proof of Propositions \ref{proposition:continuum free energy proof:continuum functions-1} is complete, all that is left is to prove Proposition \ref{prop: second approximation: the one for G}. The main ideas in this proof will roughly mimic those above, except that instead of working with the trace, we work with individual matrix elements. To do this, we first need a version of (\ref{eqn:continuum functions: linear functions of f}) for matrix elements. For this, note for each $0\le k\le L-1$, the vector $v_k=[e^{2\pi i\frac{kl}{L}}]_{l}$ is an eigenvalue of $\Delta_L$ with eigenvalue $2(\cos(2\pi k/L)-1)=-4L\sin^2\left(k\pi/L\right)$. These vectors are easily seen to be linearly independent, and thus these eigenvectors form a complete set. We also note that for each $k$, we have that $\|v_k\|=L^{1/2}$. Thus, by the spectral theorem, we see that for any matrix-valued function $f: M_{L}(\R)\to M_{L}(\R)$ which satisfies $f(U^{-1}AU)=U^{-1}f(A)U$, we have that
			\[[f(\mu I-t\Delta_L)]_{x,y}=\frac{1}{L}\sum_{k=0}^{L-1}f(\mu+4tL\sin^2(k\pi/L))\exp \left(\frac{2\pi i k(x-y)}{L}\right).\]
			However, $f$ is real-valued, so by taking the real part, we see that
			\[[f(\mu I-\Delta_L)]_{x,y}=\frac{1}{L}\sum_{k=0}^{L-1}f(\mu+4tL\sin^2(k\pi/L))\cos \left(\frac{2\pi k(x-y)}{L}\right).\]
			Now note that $\cos\left(2\pi (L-k)(x-y)/L\right)=\cos\left(2\pi k(x-y)\right)$. Then, in the same way that we obtained (\ref{eqn:continuum functions: linear functions of f}), we may obtain the formula
			\[\begin{split}[f(\mu I-t\Delta_L)]_{x,y}=&\frac{2}{L}\sum_{k=1}^{[(L-1)/2]}f(\mu+4tL\sin^2(k\pi/L))\cos \left(\frac{2\pi k(x-y)}{L}\right)\\&+\frac{1}{L}\left(f(\mu)+(-1)^{x-y}I_{L \text{ is even}}f(\mu+4tL)\right).\label{eqn:entries of function of matrix}
			\end{split}\]
			
			This expression gives us a way to represent $G_{x,t, L}(\mu)$ (defined in (\ref{eqn:def:GL})) as a sum. However, as before, it is much easier to estimate a different quantity, and then relate it to $G_{x,t, L}(\mu)$. For this, we have the following formula for $G_{x,t,L}(\mu)$ in terms of entries of the matrix exponential:
			\[G_{x,t,L}(\mu)=\int_0^\infty e^{-\mu t} L^{1/2}[e^{t\Delta_L}]_{x,0}dt. \label{eqn: resolvent identity for exp entry}\]
			This is similar in form to the representation we used for this matrix's trace in (\ref{eqn: resolvent identity involving exp general}). 
			As such, we begin by estimating the entries of this matrix exponential.
			\begin{lem}
				There is $C>0$ such that for any choice of $L\in \N$, $x,y\in \Lambda_L$ and $t\in (0,\infty)$, we have that
				\[\left| L^{1/2}[e^{t \Delta_L}]_{x,y}-\frac{e^{-(x-y)^2/4t}}{2\sqrt{\pi t}}\right|\le C\left(\frac{1}{ L^{1/2}}+\frac{e^{-tL^{1/2}}}{\sqrt{t}}+\frac{1}{\sqrt{tL}}+\frac{\sqrt{t}+|x-y|}{L^{1/4}}\right).\]
			\end{lem}
			\begin{proof}
				Before beginning the proof, we note that 
				\[2\int_{0}^\infty e^{-4t\pi^2 y^2 }\cos(2\pi x y)dy=\int_{-\infty}^\infty e^{-4t\pi^2 y^2+2\pi i x y}dy=\frac{e^{-x^2/4t}}{2\sqrt{\pi t}}. \label{eqn:continuum identity for matrix elements of exponential}\]
				Thus, the idea of the proof is to show that the spectral expansion of $ L^{1/2}[e^{t\Delta_L}]_{x,y}$ approximates the integral on the left.
				
				As $[e^{t\Delta_L}]_{x,y}=[e^{t\Delta_L}]_{|x-y|,0}$, we may assume that $y=0$ without any loss in generality. Next, we use (\ref{eqn:entries of function of matrix}) to find that
				\[ L^{1/2}[e^{t\Delta_L}]_{x,0}=\frac{2}{L^{1/2}}\sum_{k=1}^{[(L-1)/2]}e^{-4tL\sin^2(\pi k/L)}\cos(2\pi k L^{-1/2}x)+\frac{1}{L^{1/2}}\left(1+I_{L\text{ is even}}e^{-4tL}\right). \label{eqn:formula for entries of exponential}\]
				We begin by estimating the tail of this sum. As $|\cos(y)|\le 1$, we may use the tail estimates from the proof of Lemma \ref{lem:bounds on the trace of exponential}. Explicitly, if we let $\ell=[L^{3/4}]$,  (\ref{eqn:ignore-tail-1}) and (\ref{eqn:ignore-tail-2}) show that 
				\[\frac{1}{L^{1/2}}\sum_{k=\ell+1}^{[(L-1)/2]}e^{-4tL\sin^2(\pi k/L)}|\cos(2\pi kL^{-1/2}x)|\le \frac{1}{L^{1/2}}\sum_{k=\ell+1}^{[(L-1)/2]}e^{-4tL\sin^2(\pi k/L)}\le \sqrt{\frac{\pi}{4t}}e^{-4t\ell^2L^{-1}}.\]
				Thus, using trivial bounds for the last two terms in (\ref{eqn:formula to trace of exponential}), we  see that to show the desired bound it suffices to show that
				\[\left|\frac{1}{L^{1/2}}\sum_{k=1}^{\ell}e^{-4tL\sin^2(\pi k/L)}\cos(2\pi kL^{-1/2}x)-\frac{e^{-\frac{x^2}{4t}}}{4\sqrt{\pi t}}\right|\le C\left(\frac{e^{-tL^{1/2}}}{\sqrt{t}}+\frac{1}{\sqrt{tL}}+\frac{\sqrt{t}+|x|}{L^{1/4}}\right). \label{eqn:ignore-red}\]
				At this point, the proof of Lemma \ref{lem:bounds on the trace of exponential} does not generalize directly due to the fact that $\cos(2\pi kL^{-1/2}x)$ is not monotone in $x$ and not necessarily positive. However, we can still use the main idea, which is to view our sum as a Riemann sum. In particular, let us define the function
				\[f_{L}(y)=e^{-4tL\sin^2(\pi y/L^{1/2})}\cos(2\pi x y).\]
				Then we see that the sum on the left-hand side of (\ref{eqn:ignore-red}) is the right Riemann sum with spacing $\Delta x=L^{-1/2}$ for the integral $\int_{0}^{L^{-1/2}\ell} f_L(y)dy$. We recall that for a right Riemann sum (with $\Delta x=(b-a)/n$) we have the error bound
				\[\left|\sum_{k=1}^n f(a+k\Delta x)-\int_a^b f(y)dy\right|\le \frac{(b-a)^2}{2n}\sup_{y\in [a,b]}|f'(y)|.\label{eqn:Riemann-error-bound}\]
				To bound $|f'_L(y)|$, we first compute
				\[
				\begin{split}
					f'_L(y)=&e^{-4tL\sin^2(\pi yL^{-1/2})}4tL\sin(\pi y L^{-1/2})\cos(\pi y L^{-1/2})2\pi L^{-1/2}\cos(2\pi x y)\\&-e^{-4tL\sin^2(\pi yL^{-1/2})}\sin(2\pi xy)2\pi x.
				\end{split}\label{eqn:ignore-102}\]
				Employing the bounds $0\le |\cos(y)|,|\sin(y)|\le 1$ on all but the copies of $\sin(\pi yL^{-1/2})$ in the first term, we see that
				\[|f'_L(y)|\le e^{-4tL\sin^2(\pi yL^{-1/2})}4t L|\sin(\pi y L^{-1/2})|2\pi L^{-1/2}+2\pi |x|.\]
				For $\alpha>0$, the function $e^{-\alpha y^2}\alpha |y|$ is maximized at $y=1/\sqrt{2\alpha}$, so that $e^{-\alpha y^2}\alpha |y|\le e^{-1/2}\sqrt{\alpha/2}\le \sqrt{\alpha}$. We may use this to bound on the right of (\ref{eqn:ignore-102}) to see that
				\[|f_L'(y)|\le (2\pi \sqrt{t}+2\pi |x|).\]
				Thus by (\ref{eqn:Riemann-error-bound})
				\[\begin{split}
					\bigg|\frac{1}{L^{1/2}}&\sum_{k=1}^{\ell}e^{-4tL\sin^2(\pi k/L)}\cos(2\pi kL^{-1/2}x)\\
					&-\int_{0}^{L^{-1/2}\ell} e^{-4tL\sin^2(\pi y/L^{1/2})}\cos(2\pi x y)dy\bigg|\le \frac{\ell}{2L}(2\pi \sqrt{t}+2\pi|x|).
				\end{split} \label{eqn:ignore-18-2}
				\]
				The rightmost side is bounded above by $C L^{-1/4}(\sqrt{t}+|x|)$ for large $L$. We now seek to replace the integrand with an $L$-independent quantity. For we recall that for $y\in [0,1/2]$, we have that $y^2(1-y^2)\le \sin^2(y)\le y^2$. In particular, we have for any $y\in [0,L^{1/4}]$ that
				\[e^{-4t\pi^2 y^2}\le e^{-4tL\sin^2(\pi y/L^{1/2})}\le e^{-4t\pi^2 (y^2-y^4/L)}\le e^{-4t(1-L^{-1/2})\pi^2 y^2}.\]
				Again using that $|\cos(y)|\le 1$ we have that
				\[
				\begin{split}\bigg|\int_0^{L^{-1/2}\ell}& e^{-4tL\sin^2(\pi y/L^{1/2})}\cos(2\pi x y)dy-\int_0^{L^{-1/2}\ell} e^{-4t\pi^2 y^2 }\cos(2\pi x y)dy\bigg|\\
					&\le \int_0^{L^{-1/2}\ell} \left(e^{-4t\pi^2 y^2}-e^{-4t(1-L^{-1/2})\pi^2 y^2}\right)dy \le \int_0^{\infty} \left(e^{-4t(1-L^{-1/2})\pi^2 y^2}-e^{-4t\pi^2 y^2}\right)dy\\
					&=\frac{1}{4\sqrt{\pi t (1-L^{-1/2})}}-\frac{1}{4\sqrt{\pi t}}\le \frac{1}{4\sqrt{\pi t L}},
				\end{split} \label{eqn:ignore-18-3}
				\]
				where in the last line we have used the fact that $(1-x)^{-1/2}\le 1+x$ for $x\in [0,1/4]$, which holds if $L$ is large. Now to finally control this $L$-independent integral, we see from (\ref{eqn:erf-bound}) that 
				\[\left|\int_{L^{-1/2}\ell}^\infty e^{-4t\pi^2 y^2 }\cos(2\pi x y)dy\right|\le \int_{L^{-1/2}\ell}^\infty e^{-4t\pi^2 y^2 }\le  \frac{1}{4\sqrt{\pi t}}e^{-4t\pi^2L^{-1}\ell^2}\le \frac{1}{\sqrt{t}}e^{-tL^{1/2}}, \label{eqn:ignore-18-4}\]
				where the last inequality holds for large $L$. As the statement is clear for any fixed $L$, combined (\ref{eqn:continuum identity for matrix elements of exponential}), (\ref{eqn:ignore-18-2}), (\ref{eqn:ignore-18-3}) and (\ref{eqn:ignore-18-4}) complete the proof of (\ref{eqn:ignore-red}), and thus yield the desired result.
			\end{proof}
			
			With this bound established, we may obtain estimates for $G_{x,t, L}(\mu)$ by employing (\ref{eqn: resolvent identity for exp entry}).
			
			\begin{lem}
				\label{eqn: G asymptotics strong}
				There is $C>0$, such that for all $L\in \N$, $x\in \Lambda_L$ and $\mu,t\in (0,\infty)$ we have that
				\[\bigg|G_{x,t,L}(\mu)-\G_{x,t}(\mu)\bigg|\le \frac{C}{L^{1/4}}\left(\frac{1}{L^{1/4}\mu }+\frac{1}{\sqrt{t}}+\frac{1}{L^{1/4}\sqrt{\mu t}}+\frac{\sqrt{t}}{2\mu^{3/2}}+\frac{x}{\mu}\right).\label{eqn:R1 matrix d1 bound}\]
			\end{lem}
			\begin{proof}
				We need the following integral formula (see 3.471.9 and 8.432.1 of \cite{TableOfInt}):
				\[\int_0^\infty x^{-1/2}e^{-a x-b x^{-1}}dx=\sqrt{\frac{\pi}{a}}e^{-2\sqrt{ab}} \text{ for }a,b>0.\]
				From this, we conclude that
				\[\int_0^\infty e^{-\mu z}\frac{e^{-x^2/4tz}}{2\sqrt{\pi tz}}dz=\frac{1}{2\sqrt{t\mu}}e^{-x\sqrt{\frac{\mu}{t}}}=\G_{x,t}(\mu).\label{eqn:resolvent entry integral formula}\]
				Equipped with this, the proof of (\ref{eqn:R1 d=1 bound}) of Lemma \ref{lem:continuum:1d-limit result for resolvent} may be generalized immediately to show (\ref{eqn:R1 matrix d1 bound}) with the only change being that the analogue of (\ref{eqn:ignore-R1-error}) picks up an additional error term of
				\[\int_0^\infty e^{-\mu z}\left(\frac{\sqrt{tz}+|x-y|}{L^{1/4}}\right) dz=\frac{\sqrt{\pi t}}{2\mu^{3/2}L^{1/4}}+\frac{|x-y|}{\mu L^{1/4}}.\]
				Adding this to the error term in (\ref{eqn:R1 d=1 bound}) yields (\ref{eqn:R1 matrix d1 bound}).
			\end{proof}	
			
			We are now ready to give the proof of Proposition \ref{prop: second approximation: the one for G}.
			
			\begin{proof}[Proof of Proposition \ref{prop: second approximation: the one for G}]
				The claim involving $G_{x,t, L}$ immediately follows from Lemma \ref{eqn: G asymptotics strong}. Now with these, note that 
				\[G_{x,t,L}'(\mu)=-L^{1/2}[(\mu I-t\Delta_L)^{-2}]_{x,0}<0,\]
				so that $G_{x,t,L}$ is decreasing. Moreover, 
				\[G_{x,t,L}''(\mu)=2L^{-1}[(\mu I-t\Delta_L)^{-3}]_{x,0}>0,\]
				so $G_{x,t,L}$ is convex. Thus, the desired claim for $G_{x,t, L}'$ follows from the general facts established in the proof of Propositions \ref{proposition:continuum free energy proof:continuum functions-1} and \ref{proposition:continuum free energy proof:continuum functions-3}.
			\end{proof}
			
			\bibliographystyle{abbrv}
			\bibliography{mainbib}
			
		\end{document}